\documentclass[11pt]{amsart}
\usepackage{fullpage}
\usepackage{graphicx}
\usepackage{amscd}
\usepackage{amsfonts}
\usepackage{amsthm}
\usepackage{amsmath}
\usepackage{amssymb}
\usepackage{amsbsy}
\usepackage{tikz-cd}
\usepackage{hyperref}
\usepackage[arrow,matrix,curve,color]{xy}
\usepackage{enumerate}
\usepackage{pb-diagram,pb-xy}
\usepackage[normalem]{ulem}
\usepackage{color}
\usepackage[normalem]{ulem}
\usepackage{tikz}
\usepackage{xspace}
\definecolor{light-blue}{rgb}{0.8,0.85,1}
\definecolor{light-red}{rgb}{1,.4,.4}
\definecolor{purp}{rgb}{.7,.3,1}
\definecolor{yel}{rgb}{1,1,.5}
\definecolor{cy}{rgb}{0,1,1}
 \usepackage{color}
\definecolor{darkgreen}{cmyk}{1,0,1,.2}
\definecolor{m}{rgb}{1,0.1,1}

\usepackage{colortbl}
\usepackage{multirow}
\usepackage{array}

\theoremstyle{plain}
\newtheorem{theorem}{Theorem}[section]
\newtheorem{corollary}[theorem]{Corollary}
\newtheorem{lemma}[theorem]{Lemma}
\newtheorem{proposition}[theorem]{Proposition}
\newtheorem*{thmA}{Theorem A}
\newtheorem*{thmB}{Theorem B}
\newtheorem*{thmC}{Theorem C}
\newtheorem*{thmD}{Conjecture D}
\newtheorem*{thmE}{Theorem E}
\newtheorem*{thmF}{Theorem F}
\newtheorem*{thmG}{Theorem G}
\newtheorem*{thmH}{Theorem H}
\newtheorem*{thmI}{Theorem I}

\def\Di{\mathfrak{D}\kern-6.5pt/}
\def\Spi{\mathfrak{S}\kern-6.5pt/}

\theoremstyle{definition}
\newtheorem{remark}[theorem]{Remark}
\newtheorem{definition}[theorem]{Definition}

\newtheorem{conjecture}[theorem]{Conjecture}

\newcommand{\Diff}{\mathrm{Diff}}

\DeclareMathOperator*{\dirlim}{dirlim}

\def \Ch {{\rm Ch}}
\newcommand{\Pdo}{\mathrm{\Psi}}
\newcommand{\CP}{\mathbb{CP}}
\newcommand{\HP}{\mathbb{HP}}
\newcommand{\p}{\partial}

\newcommand{\spin}{\mathsf{spin}}
\newcommand{\spinc}{\mathsf{spin}^c}

\newcommand{\MSpin}{\mathsf{MSpin}}
\newcommand{\MSpinc}{\mathsf{MSpin}^c}

\newcommand{\BP}{\mathsf{BP}}
\newcommand{\K}{\mathsf{K}}
\newcommand{\ko}{\mathsf{ko}}
\newcommand{\KU}{\mathsf{KU}}
\newcommand{\ku}{\mathsf{ku}}

\newcommand{\sF}{\mathsf{F}}

\newcommand{\co}{\colon\,}

\newcommand{\bR}{\mathbb R}
\newcommand{\bC}{\mathbb C}

\newcommand{\bF}{\mathbb F}
\newcommand{\bH}{\mathbb H}

\newcommand{\bZ}{\mathbb Z}

\newcommand{\bP}{\mathbb P}
\newcommand{\bN}{\mathbb N}

\newcommand{\cA}{\mathcal A}

\newcommand{\cL}{\mathcal L}

\newcommand{\cO}{\mathcal O}
\newcommand{\cP}{\mathcal P}
\newcommand{\cR}{\mathcal R}
\newcommand{\cS}{\mathcal S}

\newcommand{\Sp}{\mathop{\rm Sp}}
\newcommand{\Spin}{\text{\rm Spin}}
\newcommand{\Spinc}{\text{{\rm Spin}}^c}

\newcommand{\wM}{\widetilde M}
\newcommand{\wW}{\widetilde W}

\newcommand{\tM}{\widetilde M}
\newcommand{\tX}{\widetilde X}

\newcommand{\rR}{\mathrm R}

\newcommand{\pt}{\text{\textup{pt}}}
\newcommand{\lp}{\textup{(}}
\newcommand{\rp}{\textup{)}}

\newcommand{\ind}{\operatorname{ind}}

\newcommand{\Ind}{\operatorname{Ind}}
\newcommand{\Hom}{\operatorname{Hom}}

\newcommand{\Tr}{\operatorname{Tr}}

\renewcommand\Im{\operatorname{Im}}

\newcommand{\Pos}{\operatorname{Pos}}

\newcommand{\rank}{\operatorname{rank}}

\newcommand{\del}{\text{\textup{del}}}
\newcommand{\id}{\text{id}}
\newcommand{\pr}{\text{pr}}
\newcommand{\per}{\operatorname{\mathbf{per}}}

\newcommand{\ord}{\operatorname{ord}}

\newcommand\SG{{\bf {\rm S}}}

\newcommand{\abs}[1]{\left\lvert#1\right\rvert} 
\newcommand{\comm}[1]{}
\newcommand{\tw}{\mathrm{tw}}
\newcommand{\gen}{\mathrm{gen}}
\makeatletter
\@namedef{subjclassname@2020}{%
  \textup{2020} Mathematics Subject Classification}
\makeatother

\title{Classification of spin$^c$ manifolds with\\
  generalized positive scalar curvature}
\author{Boris Botvinnik}
\address{Department of Mathematics\\
University of Oregon\\
Eugene OR 97403-1222, USA} 
\email[Boris Botvinnik]{botvinn@uoregon.edu}
\urladdr{http://pages.uoregon.edu/botvinn/}
\author{Paolo Piazza}
\address{Dipartimento di Matematica\\ 
  Sapienza Universit\`a di Roma\\ Piazzale Aldo Moro 5\\ 00185 Roma, Italy} \email[Paolo
  Piazza]{piazza@mat.uniroma1.it}
\urladdr{http://www1.mat.uniroma1.it/people/piazza/} 
\author{Jonathan Rosenberg}
\address{Department of Mathematics\\
University of Maryland\\
College Park, MD 20742-4015, USA} 
\email[Jonathan Rosenberg]{jmr@umd.edu}
\urladdr{http://www2.math.umd.edu/\raisebox{-3pt}{~}jmr/}
  
\begin{document}
\begin{abstract}
Suppose $M$ is a closed $n$-dimensional spin$^c$ manifold with
spin$^c$ structure $\sigma$ and associated spin$^c$ line bundle $L$.
If one fixes a Riemannian metric $g$ on $M$ and a connection
$\nabla_L$ on $L$, the \emph{generalized scalar curvature}
$R^{\text{gen}}$ of $(M,L)$ is $R_g - 2|\Omega_L|_{\text{op}}$, where
$|\Omega_L|_{\text{op}}$ is the pointwise operator norm of the
curvature $2$-form $\Omega_L$ of $\nabla_L$, acting on spinors. In a
previous paper, we showed that positivity of $R^{\text{gen}}$ is
obstructed by the non-vanishing of the index of the spin$^c$ Dirac
operator on $(M,g,L,\nabla_L)$, and that in some cases, the vanishing
of this index guarantees the existence of a pair $(g,\nabla_L)$ with
positive generalized scalar curvature.  Building on this and on
surgery techniques inspired by those that have been developed in the
theory of positive scalar curvature on spin manifolds, we show that if
$\dim M = n \ge 5$, if the fundamental group $\pi$ of $M$ is in a
large class including surface groups and finite groups with periodic
cohomology, and if $M$ is totally non-spin (meaning that the universal
cover is not spin), then $(M,L)$ admits positive generalized scalar
curvature if and only if the generalized $\alpha$-invariant of $(M,L)$
vanishes in the $K$-homology group $K_n(B\pi)$.  We also develop an
analogue of Stolz's sequence for computing the group of concordance
classes of positive generalized scalar curvature metrics, and connect
this to the analytic surgery sequence of Roe and Higson.  Finally, we
give a number of applications to moduli spaces of positive generalized
scalar curvature metrics.
\end{abstract}


\maketitle
\vspace*{-5mm}

\tableofcontents
\renewcommand{\baselinestretch}{1.25}\normalsize
\vspace*{-5mm}

\section{Introduction}
\label{sec:intro}
The classification theory for spin manifolds with positive scalar
curvature (which we will often abbreviate psc for short), especially
on manifolds of dimension $\ge 5$, is at this point very highly
developed. (For surveys, see for example
\cite{MR1268192,MR1818778,MR2408269}.) The first known obstruction for
existence of psc metrics comes from the Lichnerowicz-Schr\"odinger
formula for the square of the spin Dirac operator:
$ D^2 = \nabla\nabla^* + \frac{1}{4}R_g$,
where $R_g$ is the scalar curvature of
the Riemannian metric $g$.
In fact, the entire set of known obstructions to such metrics 
in dimensions $>4$ arises
from the index theory of the spinor Dirac operator, coupled to flat
vector bundles (in the generalized sense of Mishchenko-Fomenko, with
fibers that are projective modules over the real group $C^*$-algebra),
or from the minimal hypersurface method of Schoen and
Yau \cite{MR535700}, while constructions of psc metrics often come
from the surgery method of Gromov and Lawson \cite{MR577131} or Schoen
and Yau
\cite{MR535700}. In addition, much is known about the topology of the
space of psc metrics --- see for example
\cite{MR1339924,MR2366359,Xie-Yu-Moduli,PSZ,MR3268776,MR3956897,MR4466049,MR3681394}. 

In the papers
\cite{MR4649186,MR4817674}, two of us started to develop a parallel
theory for spin$^c$ manifolds $M$, where $R_g$
is replaced by the (matrix-valued) function on spinors
$R_{g,L}^{\tw}:=R_g+2ic(\Omega_L)$ which we call the \emph{twisted scalar
curvature}. Here the complex line bundle $L$ is the one associated to the chosen 
spin$^c$ structure $\sigma$ on $M$ (implicitly equipped
with a hermitian metric and a unitary connection $\nabla_L$) and $\Omega_L$ is
the curvature $2$-form of $\nabla_L$. 

Again, the Lichnerowicz-Schr\"odinger formula for the square of the
spin$^c$ Dirac operator
\begin{equation*}
D^2 = \nabla\nabla^* + \frac{1}{4}R_{g,L}^{\tw}
\end{equation*}  
was a starting point of the theory.  Then we define the
\emph{generalized scalar curvature} $R^{\gen}_{g,L}$ of $(M,g,L,\nabla_L)$
to be $R_g - 2|\Omega_L|_{\text{op}}$, where $|\Omega_L|_{\text{op}}$
is the pointwise operator norm of $\Omega_L$ acting on spinors.  We
showed in \cite[Lemma 3.1]{MR4817674} that positivity of the operator
$R_{g,L}^{\tw}$ is equivalent to positivity of the generalized scalar
curvature function $R^{\gen}_{g,L}$, which we
sometimes abbreviate as gpsc --- ``generalized psc''.
Thus the positivity of the curvatures
$R_{g,L}^{\tw}$ and $R^{\gen}_{g,L}$ is obstructed by the non-vanishing of the
index of the spin$^c$ Dirac operator on $(M,g,L,\nabla_L)$. 
The vanishing of the index is therefore
a necessary condition for the existence of a generalized psc 
metric, but one is of course also interested in sufficient conditions.
The case when the spin$^c$ manifold $M$ is simply connected was
tackled in \cite{MR4649186}. 
Let 
$M$ be a spin$^c$ manifold with a fixed spin$^c$ structure 
$\sigma$. 
We denote by $L$ the  line bundle associated to $\sigma$.%
\footnote{See Appendix \ref{sec:spinc-structures} for a 
brief introduction to spin$^c$ structures, 
including the way $L$ is derived from $\sigma$.}

\medskip
\noindent
{\bf Notation.}  {\it We shall write $(M,\sigma)$ for a spin$^c$
manifold with spin$^c$ structure $\sigma$.  If we want to record that
$L$ is the associated complex line bundle, then we write
$(M,\sigma,L)$ or even $(M,L)$ if the spin structure $\sigma$ is
understood. Correspondingly, we shall denote an element in
$\Omega^{\spinc}_n$ either as $[(M,\sigma)]$ or as
$[(M,L)]$.\footnote{If $H^2 (M,\bZ)$ does not contain 2-torsion
elements then there is no ambiguity in employing the notation
$[(M,L)]$; see Appendix \ref{sec:spinc-structures} for more on this.}
If $f\co M\to X$, the associated bordism class in
$\Omega^{\spinc}_n(X)$ is denoted $[(M,\sigma,f)]$ or $[(M,L,f)]$.}

\medskip
\noindent
Let $[(M,\sigma)]\in \Omega^{\spinc}_n$ be the bordism class of
$(M,\sigma)$. Recall that $(M,\sigma)$ is spin$^c$ bordant to
$(M',\sigma')$ if there exists a spin$^c$ manifold with boundary $W$
and spin$^c$ structure $\sigma_W$ such that $\partial
(W,\sigma_W)=(M,\sigma)\sqcup -(M',\sigma')$. Notice that if $L$ and
$L'$ are the line bundles associated to $(M,\sigma)$ and
$(M',\sigma')$ then $W$ comes equipped with a line bundle
$\mathcal{L}$ restricting to $L$ and $L'$ on the boundary.  Recall now
that there exists a well defined index homomorphism
\begin{equation}\label{eq:alphac}
\beta^c\co \Omega^{\spinc}_n\to KU_n.
\end{equation}
(The map $\beta^c$ also extends to a natural transformation
of homology theories.) Here is the result for the simply-connected case:
\begin{thmA}\label{thm:bj01}{\cite[Corollary 5.2]{MR4649186}}
Let $(M,\sigma)$ be a simply connected spin$^c$ manifold which is not
spin, $\dim M\geq 5$. Then $(M,\sigma)$ admits gpsc if and only if the
index $\beta^c([(M,\sigma)])$ vanishes in $KU_n$.
\end{thmA}
As in the spin case, the
key ingredients for the proof of the above result were a relevant 
surgery lemma and a detailed homotopy-theoretical
analysis of the kernel of the index homomorphism
$\beta^c\co \Omega^{\spinc}_n\to KU_n$.

In this paper, we will make the first steps towards a classification of
non-spin spin$^c$-manifolds with positive generalized scalar curvature
when the fundamental group is non-trivial. We say that a spin$^c$
manifold $(M,\sigma)$ is \emph{totally non-spin} if its universal cover is
non-spin. It turns out that this condition is crucial for the whole
theory to work.

Here is a relevant Spin$^c$ Bordism Theorem which shows that the
existence question of gpsc depends only on the spin$^c$ bordism class
of $(M,\sigma)$ in  $\Omega^{\spinc}_n(B\pi)$, where
$\pi=\pi_1 M$.
\begin{thmB}
\label{thm:spincbordismA}
Let $(M,\sigma,c\co M\to B\pi)$ be a connected closed $n$-dimensional 
spin$^c$ manifold which is totally non-spin, where $c\co M\to B\pi$ is a
classifying map for the universal $\pi_1 M=\pi$ 
cover of $M$ and $n\geq 5$. Let $L$ be the associated spin$^c$ 
line bundle. Assume that there exists a triple 
$(M',\sigma',c'\co M'\to B\pi)$ in the same bordism class in
$\Omega^{\spinc}_n(B\pi)$ with a metric $g'$ on $M'$ and a hermitian
metric $h'$ and a unitary connection on the associated line bundle $L'$ 
such that $R^{\tw}_{g',L'}>0$.
Then $M$ admits a Riemannian metric
$g$ and $L$ admits a hermitian bundle metric $h$ and a unitary
connection such that $R^{\tw}_{g,L}>0$.
\end{thmB}
Similarly to the spin case, we have a subgroup
$\Omega^{\spinc,+}_n(B\pi)$ of $\Omega^{\spinc}_n(B\pi)$ which contains
all classes $[(M,L,c)]\in \Omega^{\spinc}_n(B\pi)$ which admit
positive generalized scalar curvature; see
Corollary \ref{cor:possubgroup}. In turn, the group
$\Omega^{\spinc,+}_n(B\pi)$ can be reduced 
to a much smaller group $ku^+_n(B\pi)$:
\begin{thmC}
Let $(M,\sigma)$ be a totally non-spin spin$^c$-manifold, with 
$\pi_1 M = \pi$ and classifying map $c\co M\to B\pi$. Assume that $n\ge 5$.
Then the question of whether or not $M$ admits generalized psc only
depends on the class of $c_*([M,\sigma])\in ku_n(B\pi)$, 
where $[M,\sigma]\in ku_n (M)$
is the $ku$-valued spin$^c$-fundamental class of $M$.
In other words,
there is a subgroup $ku^+_n(B\pi)\subset ku_n(B\pi)$ such that $(M,\sigma)$
admits generalized psc if and only if $c_*([M,\sigma])\in ku^+_n(B\pi)$.
\end{thmC}
Here $ku_*$ is connective $K$-homology, the
homology theory associated to the spectrum $\ku$
which is the connective truncation of the complex $K$-theory spectrum
$\KU$.  This is defined by the universal property that $ku_n=0$
for $n<0$ and that that there is a map
of spectra $\per\co \ku\to\KU$ which is an isomorphism 
on homotopy groups in
degrees $\ge 0$, such that any map of spectra from a connective spectrum
to $\KU$ factors uniquely through $\per$. 
See \cite[Part III]{MR1324104} for more details.

The subgroup $ku^+_n(B\pi)$ in Theorem B is the image 
of $\Omega^{\spinc,+}_n(B\pi)$
under the natural map $\alpha^c\co \Omega^{\spinc}_n(B\pi)\to ku_n(B\pi)$.
\footnote{The maps $\beta^c$ and $\alpha^c$ are almost the same,
except that $\alpha^c$ takes values in connective $K$-homology and $\beta^c$
takes values in periodic $K$-homology.  We have $\beta^c=\per\circ \alpha^c$,
with $\per$ the periodization map from connective to periodic
$K$-homology.}

\medskip

\noindent
{\bf Notation.}
{\it Note that if $(M,\sigma)$ is a spin$^c$-manifold, then we 
denote the fundamental class in $\Omega^{\spinc}_n (M) $ by 
$[(M,\sigma)]$, whereas we denote the  fundamental class in 
connective $K$-homology $ku_n (M)$
by $[M,\sigma]$ {\lp}i.e., without the round parentheses{\rp}. 
The natural map 
$\alpha^c\co \Omega^{\spinc}_n (M)\to 
ku_n (M)$ sends $\alpha^c [(M,\sigma)]$  to $[M,\sigma]$.
When this won't cause confusion, 
we often write $\alpha^c [(M,\sigma)]$ to mean the image of
the $[(M,\sigma)]$ in $ku_n(\pt)$.}

\medskip
Let $\mathbf{per}\co \ku \to \KU$ be the periodization map, induced by inverting
the Bott element $\beta\in ku_2$. We also need the \emph{complex
assembly map} $A\co K_*(B\pi)\to K_*(C^*_r(\pi))$ from the $K$-homology of
$B\pi$ to the $K$-theory of the (reduced) group $C^*$-algebra of
$\pi$. As in the spin case we have a composition:
\begin{equation*}
ku_n(B\pi)\xrightarrow{\mathbf{per}}K_n(B\pi) \xrightarrow{A}
K_*(C^*_r(\pi)).
\end{equation*}
Index theory implies that the element 
$A\circ\per\circ c_*([M,\sigma])\in
K_*(C^*_r(\pi))$ (which is nothing but the higher index class of 
the Mishchenko-Fomenko spin$^c$ Dirac operator)
is an obstruction to the existence of gpsc on
$(M,\sigma,L)$ (see Theorem \ref{thm:index} below).  As far as the 
converse is concerned, we have a spin$^c$
version of the Gromov-Lawson-Rosenberg Conjecture:
\begin{thmD}[Gromov-Lawson-Rosenberg Conjecture, spin$^c$ version]
Let $(M,\sigma)$ be a closed connected totally non-spin spin$^c$
manifold with $\dim M= n$, fundamental group $\pi$ and classifying map
$c\co M \to B\pi$. Let $L$ be the associated spin$^c$ line bundle.
Then for $n\ge 5$, $(M,\sigma,L)$ admits positive generalized scalar
curvature if and only if $A\circ\per\circ c_*([M,\sigma])=0$ in
$K_*(C^*_r(\pi))$.
\end{thmD}
We will show that Conjecture D holds in a large number of
cases; in particular, we prove (see Theorem \ref{thm:GLRpercoh} below):
\begin{thmE}
Let $(M,\sigma)$ be a closed, connected, totally non-spin manifold
with $\dim M = n\geq 5$, and finite fundamental
group $\pi$ having periodic cohomology.  Then $(M,\sigma)$ admits positive
generalized scalar curvature if and only
if $\alpha^c(M,\sigma)=0$ in $ku_n$.
\end{thmE}
While we have promising evidence for Conjecture D in many cases,
we can modify the proof given by Thomas Schick in \cite{MR1632971} to
show that Conjecture D is false for at least some infinite groups with
torsion.
\begin{thmF}
There exist totally non-spin spin$^c$ $5$-manifolds with fundamental
group $\pi=\bZ^4\times \bZ/p$, where $p$ a prime, which have vanishing
generalized Dirac index in $K_5(C^*(\pi))$ but which do not admit
positive generalized scalar curvature, and do not admit positive
scalar curvature at all.
\end{thmF}
Next, we construct and analyze a spin$^c$ analogue of the Stolz' exact
sequence:
\begin{equation}
\label{eq:Stolz-int}
\cdots\to{\rR}^{\spinc}_{*+1} (X)\xrightarrow{\partial}
{\rm Pos}^{{\rm spin}^c}_* (X)\xrightarrow\
\Omega^{\spinc}_* (X)\xrightarrow{\iota} {\rR}^{\spinc}_{*} (X)\to\cdots
\end{equation} 
The groups ${\rm Pos}^{{\rm spin}^c}_* (X)$ and ${\rR}^{\spinc}_{n}
(X)$ are defined similarly to the spin case. We prove a few
technical results: the Extension Theorem
(Theorem \ref{th:extension-theorem}) and, in particular, building on the spin$^c$ bordism theorem 
and results of Puglisi-Schick-Zenobi, the following
\begin{thmG}
Let $f\co X\to Y$ be  a 2-equivalence. Then the induced homomorphism
\begin{equation*}
f_* \co  {\rR}^{\spinc}_{*} (X)\rightarrow  {\rR}^{\spinc}_{*} (Y)
\end{equation*} 
is an isomorphism.
\end{thmG}
We define a concordance set $\cP^c (M,L)$ of gpsc for a spin$^c$
manifold $M$ with associated spin$^c$ line bundle $L$ 
admitting gpsc and show, once
again using the spin$^c$ bordism theorem, that it carries a
(non-natural) group structure if $M$ is totally non-spin. Continuing
the analogy with the spin case, we prove the crucial result that there
exists a free and transitive action of the group ${\rR}^{\spinc}_{n}
(B\pi)$ on $\cP^c (M,L)$ (see Theorem \ref{thm:classification} and
Proposition \ref{prop:action}). 

We emphasize that while most of the statements concerning the groups
${\rm Pos}^{{\rm spin}^c}_* (X)$ and ${\rR}^{\spinc}_{*} (X)$ are
parallel to the spin case, the proofs do involve new ideas and new
complications.  In particular, it turns out that the condition for
spin$^c$ manifolds to be totally non-spin is crucial for all
constructions. 

The exact sequence (\ref{eq:Stolz-int}) is intimately related to the Higson-Roe analytic surgery sequence:
\begin{thmH}
Let $\Gamma$ be a finitely presented group.  Then there is a
universal commutative diagram with exact rows:
\begin{equation*}\label{eq:StolzToAna}
\begin{tikzcd}
\cdots \to \Omega^{\spinc}_{n+1} (B\Gamma) \ar[r] \ar[d, "\beta^c"]&
R^{\spinc}_{n+1}(B\Gamma) \ar[r] \ar[d, "\Ind^c_\Gamma"]
&
\Pos^{\spinc}_n (B\Gamma) \ar[r] \ar[d, "\rho^c_\Gamma"] &
\Omega^{\spinc}_n (B\Gamma) \ar[d, "\beta^c"] \to \cdots
\\ 
\cdots \to K_{n+1} (B\Gamma) \ar[r] & K_{n+1} ( C^* (E\Gamma)^\Gamma) \ar[r] & 
K_{n+1}(D^* (E\Gamma)^\Gamma) \ar[r] & K_{n} (B\Gamma)  \to \cdots
\end{tikzcd}
\end{equation*}
with the $\Ind^c_\Gamma$ and $\rho^c_\Gamma$  homomorphisms defined using coarse index theory.
\end{thmH}

See section \ref{sec:mapping}, in particular Theorem \ref{theo:commute-r-pos}, for more details
and Appendix \ref{appendix1} for 
the definitions of $\Ind^c_\Gamma$ and $\rho^c_\Gamma$.

Next, we analyze the rho invariant $\rho_{\Gamma}\co {\rm Pos}^{{\rm
spin}^c}_* (B\Gamma)\to \SG^\Gamma_* (E\Gamma)$ for cyclic and
elementary abelian finite groups $\Gamma$, where $\SG^\Gamma_*
(E\Gamma)$ is the \emph{structure group}, the natural target for the
rho map. (See Propositions \ref{prop:strsetZp}
and \ref{prop:strsetZpr}, Theorem \ref{thm:rangeofrho}.)

We also analyze the moduli space $\mathcal{R}^{c,+}(M,\sigma,L)/ {\rm
Diff} (M,\sigma)$ of gpsc; see section \ref{sec:psc} for relevant
definitions.  In particular, we prove
\begin{thmI}
Let $(M,\sigma,L)$ be spin$^c$ and totally non-spin and assume that
$(g,\nabla_L)$ is gpsc. Assume that $\pi_1 (M)\equiv \Gamma$ has an
element of finite order and that $\dim
M=4k+3$, $k\ge 1$. Then the set
$\pi_0 \left(\mathcal{R}^{c,+}(M,\sigma,L)/ {\rm Diff}
(M,\sigma) \right)$ has infinite cardinality.
\end{thmI}
Theorem I says that in many cases, not only
is the set of connected components of gpsc structures on $(M,\sigma)$
infinite, but it even remains infinite after dividing out by the 
action of the spin$^c$-preserving diffeomorphism group.\footnote{Incidentally,
in contrast with the spin case, we shall argue in Section \ref{sec:psc} that  ${\rm Diff}
(M,\sigma)$ is {\em not} always a finite index subgroup of ${\rm Diff}
(M)$.} The proof of Theorem I employs
crucially the spin$^c$-bordism theorem and it is presented in Appendix \ref{appendix2}. 
Using higher rho numbers, as in \cite{PSZ}, and employing 
heavily our classification theorem, Theorem \ref{thm:classification}, we  give sharper results on moduli spaces
assuming in addition that the fundamental group is Gromov hyperbolic. See Theorem \ref{theo:moduli1}
and Theorem \ref{theo:moduli2}. Proofs of these sharper results can be found in Appendix \ref{appendix2}; the discussion there can also serve as a quick introduction to the results in \cite{PSZ}.

\subsection*{Acknowledgments}
Part of this work was done during visits of B.\ B.\ and 
J.\ R.\ at Sapienza Universit\`a di Roma and 
of B.\ B.\ and P.\ P.\ at the Brin Mathematics Research Center. We thank these institutions for their hospitality and financial support that made these visits possible. Boris Botvinnik was
partially supported by Simons collaboration grant 708183. 
We thank Vito Felice Zenobi for interesting discussions.

\section{The spin$^c$ Bordism Theorem}
\label{sec:bordism}
The first thing we need to do is to generalize some
of 
\cite[Section 4]{MR4649186} to the non-simply connected case.
First we remind the reader of a key tool, the surgery theorem
from \cite{MR4649186}.
As in the Introduction we use the notation $(M,\sigma)$ for a spin$^c$ manifold,
with $\sigma$ the chosen spin$^c$ structure on $M$. 
Recall that $\sigma$ determines a choice of an orientation on $M$
and a complex line bundle $L$ on $M$ with $c_1(L)$ reducing mod $2$ to
$w_2(M)$.  When $H^1(M;\bZ/2)=0$, the orientation 
together with the choice of
the spin$^c$ line bundle $L$ are equivalent to the choice of spin$^c$ 
structure $\sigma$ (see  Appendix \ref{sec:spinc-structures} for an explanation).

\begin{definition}[{\cite[Definition 4.1]{MR4649186}}]
\label{def:spincsurgery}
Let $(M, \sigma)$ be a closed spin$^c$ manifold.
We say that a spin$^c$ manifold $(M', \sigma')$
is \emph{obtained from
$(M,\sigma)$ by spin$^c$ surgery in codimension $k$} if there is a
sphere $S^{n-k}$ embedded in $M$ with trivial normal bundle, $M'$ is
the result of gluing in $D^{n-k+1}\times S^{k-1}$ in place of
$S^{n-k}\times D^k$, and there 
is a spin$^c$ structure $\sigma''$ on the
trace of the surgery---a bordism from $M$ to $M'$---restricting to $\sigma$
on $M$ and to $\sigma'$ on $M'$.
\end{definition}
\begin{theorem}[{\cite[Theorem 4.2]{MR4649186}}]
\label{thm:Crelle4.2}
Let $(M,\sigma)$ be a closed 
$n$-dimensional spin$^c$ manifold. Assume that $M$ admits a
Riemannian metric $g$ and that the associated spin$^c$ line bundle
$L$ admits a hermitian bundle metric $h$
and a unitary connection such that $\frac14 R_g + \frac{i}{2}c(\Omega_L) > 0$.
Let $(M', \sigma')$ be obtained from $(M,\sigma)$ 
by spin$^c$ surgery in codimension $k\ge 3$.
Then there is a metric $g'$ on $M'$, and the
spin$^c$ line bundle $L'$ on $M'$ admits a hermitian bundle
metric $h'$ and a unitary connection, such that
$\frac14 R_{g'} + \frac{i}{2}c(\Omega_{L'}) > 0$.
\end{theorem}
We say that a
manifold $M$ is \emph{totally non-spin} if its universal cover is
non-spin. 
The main result of this section, which generalizes \cite[Theorem 4.3]{MR4649186},
is the following.
\begin{theorem}[Spin$^c$ Bordism Theorem]
\label{thm:spincbordism}
Let $(M,\sigma)$ be a connected closed 
$n$-dimensional spin$^c$ manifold
which is totally non-spin, with $n\geq 5$.
Assume that $\pi_1 M =\pi$ and that we have
fixed a classifying map $c\co M\to B\pi$ for the universal cover of $M$.  
Consider $(M,\sigma,c)\in \Omega^{\spinc}_n(B\pi)$ and assume that there exists
$(M',\sigma',c')$ in the same bordism class in $\Omega^{\spinc}_n(B\pi)$
with a metric $g'$ on $M'$ and a hermitian metric $h'$ and
a unitary connection on the associated line bundle
$L'$ such that
$\frac14 R_{g'} + \frac{i}{2}c(\Omega_{L'}) > 0$.  Then
$M$ admits a Riemannian metric
$g$ and the spin$^c$ line bundle $L$ for $M$
admits a hermitian bundle metric $h$
and a unitary connection such that $\frac14 R_{g} + \frac{i}{2}c(\Omega_L) > 0$.
\end{theorem}
\begin{proof}
The argument is modeled on the proof of \cite[Theorem 2.13]{MR866507}.
We can assume that $M'$ is connected, since if it isn't we can do
spin$^c$ surgery on embedded $0$-spheres to arrange this.  
Using the assumption that $(M,\sigma,c)$ and
$(M',\sigma',c')$ define the same class in $\Omega^{\spinc}_n(B\pi)$,
choose a compact spin$^c$
bordism $(W^{n+1},\sigma'')$ from $M'$ to $M$, along with a map $f\co
W\to B\pi$ extending $c$. Let $\cL$ be the
spin$^c$ line bundle
on $W$, restricting to $L$ on $M$ and to $L'$ on $M'$. Note
that we have a diagram of groups
\[
\xymatrix{\pi_1(M') \ar[r] \ar[dr]_{c_*'}& \pi_1(W) \ar[d]_{f_*}& \ar[l] \pi_1(M)\ar[dl]_(.55)\cong^(.6){c_*}\\
& \pi}\,
\]
which shows that the homomorphism $\pi_1(W)\xrightarrow{f_*} \pi$ is a
split surjection.  Begin by doing spin$^c$ surgeries on $1$-spheres
with trivial normal bundles in the interior of the manifold $W$,
generating $\ker(f_*\co \pi_1(W)\to \pi)$,
changing the bundle $\cL$ and the spin$^c$
structure as needed,
using Theorem \ref{thm:Crelle4.2}, to reduce to the case where $W$
also has fundamental group $\pi$. In order to use
Theorem \ref{thm:Crelle4.2}, we need the inclusion $M\hookrightarrow
W$ to be $2$-connected, which will require further surgery on $W$.
Since $\pi_1(M)\to \pi_1(W)$ is an isomorphism, we can pass to the
universal covers $\wM \hookrightarrow \wW$, and can then identify
$\pi_2$ with $H_2$ via the Hurewicz theorem.  We need to use the
assumption that $M$ is spin$^c$ and totally non-spin.  This implies
that $w_2(\wM)\ne 0$ and can be identified with a map
$H_2(\wM)\twoheadrightarrow \bZ/2$.  Furthermore, although $\pi_2(M)$
and $\pi_2(W)$ may not be finitely generated as abelian groups, they
are finitely generated as $\bZ[\pi_1(M)]$-modules because they can be
identified with $H_2(\wM)$ and $H_2(\wW)$.  The commuting diagram of
$\bZ[\pi_1(M)]$-modules with exact rows
\[
\xymatrix{
\pi_2(M)\ar[r]\ar[d]^{w_2(\wM)} & \pi_2(W) \ar[r]\ar[d]^{w_2(\wW)}
& \pi_2(W,M)\ar[r] & 0\\
\bZ/2 \ar@{=}[r] & \bZ/2 \ar[r] &0&}\,,
\]
along with the surjectivity of the homomorphism
$w_2(\wM)\co\pi_2(M)\to\bZ/2$, shows that $\pi_2(W,M)$ is generated as
a $\bZ[\pi_1(M)]$-module by finitely many classes in the kernel of
$w_2(\wW)$, i.e., is generated by classes of finitely many $2$-spheres
in the interior of $W$ on which the normal bundle is trivial. Since
$\dim W\ge 6$, we can take these $2$-spheres to be embedded and do
surgery on them to kill off the group
$\pi_2(W,M)$. We just need to check that this procedure is compatible
with the spin$^c$ condition.  The line bundle $\cL$ on $\wW$ has first
Chern class reducing mod $2$ to $w_2(\wW)$.  Since our surgeries were
done on $2$-spheres in the kernel of $w_2(\wW)$, the class $c_1(\cL)$
restricted to one of these $2$-spheres is \emph{even}, and we can
modify $\cL$ by tensoring with the square of another line bundle so
that $\cL$ will be trivial on the $2$-sphere and thus will extend
across the surgery without changing the spin$^c$ condition (which is
that $c_1$ of the line bundle reduces mod $2$ to $w_2$).
So we can assume that $M\hookrightarrow W$ is $2$-connected.

To conclude the proof, we want to apply Theorem \ref{thm:Crelle4.2}.
But that requires knowing that we can decompose $W$ so that all
spin$^c$ surgeries done in going from $M'$ to $M$
through $W$ are in
codimension at least three.  For this we apply \cite[Theorem
3]{MR290387}, which says that if the inclusion $M\hookrightarrow W$ is
2-connected, then it is geometrically 2-connected, which means exactly
that all surgeries going from $M'$ to $M$ through $W$ are in
codimension at least three.
\end{proof}

\begin{corollary}
\label{cor:possubgroup}
Fix a finitely presented group $\pi$.  Then for $n\ge 5$, there is a
subgroup $\Omega^{\spinc,+}_n(B\pi)$ of the bordism
group $\Omega^{\spinc}_n(B\pi)$ such that if $(M^n,\sigma)$ is a
totally non-spin spin$^c$ manifold with classifying map $c\co M\to
B\pi$ for the universal cover of $M$, then $(M,\sigma)$ admits
positive generalized scalar curvature if and only if the class of
$(M,\sigma,c)$ in $\Omega^{\spinc}_n(B\pi)$ lies in
$\Omega^{\spinc,+}_n(B\pi)$.
\end{corollary}
\begin{proof}
We simply take $\Omega^{\spinc,+}_n(B\pi)$ to consist of all classes
of $(M,\sigma,c\co M\to B\pi)$ for which $M$ is a connected totally non-spin
spin$^c$ manifold with gpsc.  This is closed under multiplication by
$-1$ since that comes from reversing the spin$^c$ structure.  We
need to show it is closed under addition. Now addition of bordism
classes comes from disjoint union, but given $M\xrightarrow{c} B\pi$
and $M'\xrightarrow{c'} B\pi$ with $c_*$ and $c'_*$ isomorphisms on
$\pi_1$, we can do spin$^c$ surgery on $M\sqcup M'$ to make it
connected and to reduce the fundamental group down to $\pi$
(the connected sum will have fundamental group $\pi*\pi$).  The
resulting manifold will again be totally non-spin and will have gpsc
by Theorem \ref{thm:spincbordism}.  Finally, any totally non-spin
spin$^c$ manifold in the same bordism class as something in
$\Omega^{\spinc,+}_n(B\pi)$ will again have gpsc, again by
Theorem \ref{thm:spincbordism}.   
\end{proof}

\section{Reduction to $\ku$}
\label{sec:ku}
In this section, we want to improve Theorem \ref{thm:spincbordism} to
something which is easier to use, parallel to \cite[Theorem
4.11]{MR1268192}. 
Here is the main result of this section:
\begin{theorem}[Reduction to $ku$-theory]\label{thm:main-ku}
Let $(M,\sigma)$ be a totally non-spin spin$^c$-manifold.
Assume that $M$ has fundamental group $\pi$ and classifying map $c\co
M\to B\pi$ and that $n\ge 5$. Then the question of whether or not $M$
admits generalized psc only depends on the
class $c_*([M,\sigma])\in ku_n(B\pi)$, 
where $[M,\sigma]\in ku_n (M)$ is the
$ku$-valued spin$^c$-fundamental class of $M$.
More precisely, there
is a subgroup $ku^+_n(B\pi)\subset ku_n(B\pi)$ such that $(M,\sigma)$
admits generalized psc if and only
if $c_*([M,\sigma])\in ku^+_n(B\pi)$.
\end{theorem}

\noindent
We notice that the subgroup $ku^+_n(B\pi)$ is simply the 
image of $\Omega^{\spinc,+}_n(B\pi)$ under the natural map
$\alpha^c\co \Omega^{\spinc}_n(B\pi)\to ku_n(B\pi)$.

\smallskip
\noindent
Before giving a proof of Theorem \ref{thm:main-ku}, we have to recall
a few results and constructions concerning spin and spin$^c$ bordism.
\subsection{Background on spin and spin$^c$ bordism} 
Recall that
the group $\Spinc(n)$ is defined to be $\Spin(n)\times_{\bZ/2} U(1)$.
After passing to classifying spaces and then to their Thom spectra,
this gives the decomposition
\begin{equation}
\label{eq:Mspinc}
\MSpinc = \MSpin\wedge \Sigma^{-2}\CP^{\infty}.
\end{equation}
This has a nice geometric interpretation.  Given a spin$^c$
manifold $(M,\sigma)$ of dimension $n$ and a map $M\to X$ for some
space $X$, the bordism class of $(M,\sigma)\to X$ corresponds to
a class in 
\[
\pi_n(\MSpinc\wedge X_+)=\pi_n(\MSpin\wedge \Sigma^{-2}\CP^{\infty}\wedge X_+)
=\pi_{n+2}(\MSpin\wedge \CP^{\infty}\wedge X_+) =
\widetilde\Omega^{\spin}_{n+2}(X\times \CP^{\infty}),
\]
which represents the spin bordism class of a spin $(n+2)$-manifold
$N$ equipped with a complex line bundle $L'$
and a map to $X$, so that $M$ is obtained from $N$ by dualizing $L'$.  
(Roughly speaking, $M$ is the vanishing locus of
a generic section of $L'$.)

Now denote by
$\ko$ and $\ku$ the spectra classifying connective real and complex 
$K$-theory, respectively. We have index maps of spectra,
coming from the canonical $\ko$- and $\ku$-orientations of
spin and spin$^c$ manifolds:
\begin{equation*}
\alpha\co \MSpin \to \ko, \qquad \alpha^c \co \MSpinc \to \ku. 
\end{equation*}
Let $\widehat{\MSpin}$ and $\widehat{\MSpinc}$
be the fibers of these maps (in the category
of spectra), following the notation in \cite{MR1189863},
so that we have fiber/cofiber sequences of spectra
\begin{equation}
\label{eq:MSpinfiber}
\ \ \ \widehat{\MSpin} \to \MSpin \xrightarrow{\alpha} \ko \ \ \
\mbox{and}
\end{equation}
\begin{equation}
\label{eq:MSpincfiber}
\widehat{\MSpinc} \to \MSpinc \xrightarrow{\alpha^c} \ku.
\end{equation}
A fundamental result of Stolz \cite{MR1189863,MR1259520} is that the
homotopy groups of $\widehat{\MSpin}$ are represented by spin
manifolds with a fiber bundle structure $\bH\bP^2\to M^{n+8}\to N^n$,
where $N$ is a spin $n$-manifold and the structure group of the bundle
is the isometry group $P\Sp(3)$ of the standard homogeneous metric on
$\bH\bP^2$.  Since such manifolds $M^{n+8}$ clearly have psc metrics
(obtained by taking any metric which agrees with the standard metric
on the $\bH\bP^2$ fibers and then shrinking the diameters of the
fibers), this led via \eqref{eq:MSpinfiber} to the proof of the
result \cite[Theorem 1.3]{MR4541288} that any class in
$\Omega^{\spin}_n(X)$ mapping under $\alpha$ to $0$ in $ko_n(X)$ is
represented by a psc manifold.  However, the proof is trickier than
one might at first guess because it is not clear that all such classes
come from manifolds with an $\bH\bP^2$-bundle structure; this is true
if $\pi$ is trivial but it's not clear if it's true for general $\pi$.

We will apply similar constructions below, but
first we need a few more facts about spin$^c$ bordism.  It is known
that classes $[(M,\sigma,L)]$ in
$\Omega^{\spinc}_n$ are determined by mod-$2$
Stiefel-Whitney numbers of $M$ along with integral cohomology
characteristic numbers $\langle x,[M]\rangle$, where $x$ is
a polynomial in the Pontryagin classes of $M$ and in $c_1(L)$
(\cite[\S8]{MR0234475} and 
\cite[Chapter XI]{MR0248858}).  From this one can see that
the first nontrivial spin$^c$ bordism groups are
$\Omega_2^{\spinc}\cong \bZ$, generated by 
$[(\bC\bP^1, \cO(2))]$,\footnote{Recall that for 
simply connected manifolds, an orientation and a choice of spin$^c$
line bundle already determine the spin$^c$ structure, so we omit
the $\sigma$ from the notation.}
and $\Omega_4^{\spinc}\cong \bZ\oplus \bZ$, generated by
$[(\bC\bP^1, \cO(2))]^2$ and $[(\bC\bP^2, \cO(1))]$.  The more
familiar generator $[(\bC\bP^2, \cO(3))]$ ($\cO(3)$ is the
anticanonical bundle of $\bC\bP^2$), is spin$^c$ bordant
to the sum of these two.  The advantage of the generator
$[(\bC\bP^2, \cO(1))]$ is that $\alpha^c([(\CP^2,\cO(1))])=0$;
see \cite{MR4649186,MR508087}. 

We use $(\CP^2,\cO(1))$ to construct the first transfer map
\begin{equation}\label{ku:eq1}
T_{\CP^2}: \Omega^{\spinc}_{n-4}(BSU(3)) \to \Omega^{\spinc}_{n}
\end{equation}
as follows.  We identify $\CP^2$ with the homogeneous space
$SU(3)/S(U(2)\times U(1))$; we have a fiber sequence
\begin{equation*}
S(U(2)\times U(1))\to SU(3)\to \CP^2
\end{equation*}
which in turn gives us a fiber bundle
\begin{equation}\label{ku:eq2}
BS(U(2)\times U(1))\to BSU(3)
\end{equation}
with the fiber $\CP^2$, which we call \emph{the universal geometric
$\CP^2$-bundle}. Note that the group $SU(3)$ acts transitively on $\CP^2$
(by isometries of the standard Fubini-Study 
metric on $\CP^2$), preserving the first
Chern class of the bundle $\cO(1)$.

Let $(B,\sigma)$ be a spin$^c$-manifold, $\dim B=n-4$, and let $f\co B\to BSU(3)$
represent an element in $\Omega^{\spinc}_{n-4}(BSU(3))$. By pulling
back the universal geometric $\CP^2$-bundle, we obtain a fiber bundle
$E\to B$ with fiber $\CP^2$. The manifold $E$ has dimension $n$ and
a spin$^c$ structure inherited from the spin$^c$ structure on $B$ defined
by $\sigma$ and the $\spinc$ structure on $\CP^2$ defined by $\cO(1)$.

Similarly we have the second transfer map 
\begin{equation}\label{ku:eq3}
T_{\HP^2} \co \Omega^{\spinc}_{n-8}(BPSp(3)) \to \Omega^{\spinc}_{n}.
\end{equation}
We identify $\HP^2$ with the homogeneous space
$PSp(3)/P(Sp(2)\times Sp(1))$ to obtain \emph{the universal geometric
$\HP^2$-bundle}:
\begin{equation}\label{ku:eq4}
BP(Sp(2)\times Sp(1))\to BSp(3),
\end{equation}
which, similarly to the above, determines the above transfer
homomorphism $T$. (The spin$^c$ line bundle on $\HP^2$ is trivial,
as $\HP^2$ is a $3$-connected spin manifold.)
We need the following result:
\begin{theorem}\label{thm:ku-bj2}\cite[Theorem 5.1]{MR4649186} 
Let $n\geq 4$. Then the transfer maps
\begin{equation*}
\Omega^{\spinc}_{n-4}(BSU(3)) \xrightarrow{T_{\CP^2}} \Omega^{\spinc}_{n} \ \ \mbox{and} \ \ 
\ \Omega^{\spinc}_{n-8}(BPSp(3)) \xrightarrow{T_{\HP^2}} \Omega^{\spinc}_{n}
\end{equation*}
have images lying in the kernel of $\alpha^c\co \Omega^{\spinc}_{n} \to
ku_n$. Furthermore, the kernel of $\alpha^c$ is additively generated
by the images of $T_{\CP^2}$ and $T_{\HP^2}$. 
\end{theorem}
Now we recall one more ingredient from the proof of
Theorem \ref{thm:ku-bj2}.  Namely \cite[Theorem 2.1]{MR3120631},
\begin{equation}\label{eq:split}
\ko\wedge \Sigma^{-2}\CP^{\infty}\simeq \bigvee_{\ell\geq 0} \Sigma^{4\ell} \ku.
\end{equation}
The proof of \eqref{eq:split} given in
\cite{MR3120631} is only $2$-local, but one can also prove this
after localizing at an odd prime $p$ as follows.  Recall that
after localizing at $p$ odd, $\ku\simeq \ko\vee \Sigma^2\ko$.
By \cite[Theorem 3.1.15]{MR860042} or its dual, if $\cA_p$
denotes the mod-$p$ Steenrod algebra, then $H^*(\ku; \bF_p)$
is (as an $\cA_p$-algebra) a direct sum
$\bigoplus_{0\le \ell<p-1}\Sigma^{2\ell}\cA_p\otimes_{E(1)} M$,
with $M=\bF_p$ a trivial module over the subalgebra 
$E(1)=E(Q_0, Q_1)$ of $\cA_p$,  where $Q_j$ are the
Milnor elements (with $Q_0$ the Bockstein and $Q_1=P^1Q_0-Q_0P^1$).
We have $H^*(\bC\bP^\infty;\bF_p)=\bF_p[u]$ with $u$ in degree $2$,
with trivial action of $E(1)$,
so as an $\cA_p$-module the reduced cohomology splits as 
$\bigoplus_{1\le \ell\le p-1}\cA_p u^\ell$.  Looking at the resulting
$\cA_p$-module structure on 
$H^*(\ku)\otimes \widetilde H^*(\Sigma^{-2}\CP^{\infty})$
(all cohomology over $\bF_p$) gives the result.

Smashing \eqref{eq:MSpinfiber} with $\Sigma^{-2}\CP^{\infty}$
we get another fiber sequence
\begin{equation}
\label{eq:MSpincfiber1}
\widehat{\MSpin}\wedge \Sigma^{-2}\CP^{\infty} \to \MSpinc 
\to \ko\wedge \Sigma^{-2}\CP^{\infty}.
\end{equation}
We denote the cofiber of \eqref{eq:MSpincfiber1} by
\begin{equation*}
\displaystyle
\widehat{\ku}:= \ko\wedge \Sigma^{-2}\CP^{\infty}\simeq
\bigvee_{\ell\geq 0} \Sigma^{4\ell}\ku
\end{equation*}
(note the use of \eqref{eq:split}), which we will use later.

Localized at the prime 2, the spectrum $\MSpin^c$ splits additively as a
direct sum of suspensions $\Sigma^{4n(J)}\ku$ of $\ku$ indexed by
partitions $J$ of the size $n(J)$ together with some suspensions of
the Eilenberg-Mac Lane spectra $H\bZ_2$, see \cite[\S2]{MR1166518}.  The lowest
dimensional summand is $\ku$, which corresponds to the empty partition,
and, at the 
prime 2, the index map $\alpha^c$ is the projection onto that summand.

The transfer maps $T_{\CP^2}$ and $T_{\HP^2}$ from
Theorem \ref{thm:ku-bj2} determine a map of spectra:
\begin{equation*}
\mathcal{T}_{\spinc}= T_{\CP^2}\vee T_{\HP^2}\co
\MSpinc\wedge(\Sigma^4BSU(3)_+\vee \Sigma^8BPSp(3)_+)\to \MSpinc.
\end{equation*}

Let
$\mathcal{T}_{\spin}\co \Omega^{\spin}_{n-8}(BPSp(3))\to \Omega^{\spin}_{n}$
be the original transfer used by Stolz, \cite{MR1189863}. This gives a map of
spectra
\begin{equation*}
\mathcal{T}_{\spin}\co \MSpin\wedge \Sigma^8BPSp(3)_+ \to \MSpin.
\end{equation*}
The transfer map $\mathcal{T}_{\spin}$ factors through the 
spectrum $\widehat{\MSpin}$, i.e.
there is a commutative diagram of spectra:
\begin{equation*}
\begin{tikzcd}
& \widehat{\MSpin}\ar[d,"i"]& \\
		\MSpin\wedge\Sigma^8 BPSp(3)_+
                \ar[r,"\mathcal{T}_{\spin}"]
                \ar[ur,"\widehat{\mathcal{T}}_{\spin}"]
                & \MSpin \ar[r,"\alpha"] & \ko
\end{tikzcd}
\end{equation*}
\begin{theorem}[{\cite{MR1259520}}]\label{thm:stolz2}
Localized at the prime $2$, the map
\begin{equation*}
\mathcal{T}_{\spin}\co \MSpin\wedge\Sigma^8 BPSp(3)_+\to \MSpin
\end{equation*}
is a split surjection. In particular, there is a $2$-local spitting of spectra:
$
\MSpin = \widehat{\MSpin}\vee \ko.
$
\end{theorem}
According to Theorem \ref{thm:ku-bj2}, the transfer map
$\mathcal{T}_{\spinc}$ also factors through the spectrum $\widehat{\MSpinc}$:
\begin{equation*}
\begin{tikzcd}
& \widehat{\MSpinc}\ar[d,"i^c"]& \\
		\MSpinc\wedge(\Sigma^4BSU(3)_+\vee\Sigma^8 BPSp(3)_+) 
                \ar[r,"\mathcal{T}_{\spinc}"]
                \ar[ur,"\widehat{\mathcal{T}}_{\spinc}"]
                & \MSpinc \ar[r,"\alpha^c"] & \,\ku.
\end{tikzcd}
\end{equation*}
Granath proved in his thesis the following $\spinc$ analog of
Theorem \ref{thm:stolz2}:
\begin{theorem}[{\cite[Theorem 1.8]{MR4639563}}]
\label{thm:stolz-spinc}
Localized at prime 2, the map
\begin{equation*}
\mathcal{T}_{\spinc}:
\MSpinc\wedge(\Sigma^4BSU(3)_+\vee\Sigma^8 BPSp(3)_+) 
\to \widehat{\MSpinc}
\end{equation*}
is a split surjection. In particular, there is a $2$-local spitting of
spectra: $\MSpinc = \widehat{\MSpinc}\vee \ku$.
\end{theorem}
In particular, this means that 2-locally the induced map
\begin{equation}\label{eq:transfer01}
\widehat{\mathcal{T}}_{\spinc}: (\MSpinc\wedge(\Sigma^4BSU(3)_+
\vee\Sigma^8 BPSp(3)_+))_n(X) \to  (\widehat{\MSpinc})_n(X)
\end{equation}
is surjective for any space $X$, and in particular for 
$X=B\pi$, for any group $\pi$.

\subsection{Proof of Theorem \ref{thm:main-ku}}
Theorem \ref{thm:main-ku} will follow from
Theorem \ref{thm:spincbordism} and Corollary \ref{cor:possubgroup},
once we show that the kernel of the map
$\alpha^c\co \Omega_n^{\spinc}(B\pi)\to
ku_n(B\pi)$ is represented in dimensions
$n\ge 5$ by spin$^c$ manifolds of gpsc.
We define $ku_n^+(B\pi):= \alpha^c(\Omega_n^{\spinc,+}(B\pi))$.
Let $(M,\sigma,c)$ be a totally non-spin spin$^c$
$n$-manifold with fundamental group $\pi$ and classifying map $c\co
M\to B\pi$. Then, according to
Corollary \ref{cor:possubgroup}, $M$ admits gpsc
$\Leftrightarrow$ $[(M,\sigma,c)]\in \Omega_n^{\spinc,+}(B\pi))$, but
if $\ker\alpha^c$ lies in
$\Omega_n^{\spinc,+}(B\pi)$, that's the same as
saying that the image of $[(M,\sigma,c)]$ in $ku_n(B\pi)$ lies in
$ku_n^+(B\pi)$.)

The 
problem splits into two cases: the $2$-local problem and the problem
when $2$ is inverted.

The $2$-local problem follows from surjectivity of the map
$\widehat{\mathcal{T}}_{\spinc}$ from (\ref{eq:transfer01}), i.e. that
\begin{equation*}
\Im(\widehat{\mathcal{T}}_{\spinc}\co \Omega^{\spinc}_n((\Sigma^4BSU(3)_+
\vee\Sigma^8 BPSp(3)_+)\wedge X_+)_{(2)} \to \Omega^{\spinc}_n(X)_{(2)})
\subset \Omega^{\spinc,+}_n(X)_{(2)}. 
\end{equation*}
Since the image of $\widehat{\mathcal{T}}_{\spinc}$ consists
of bordism classes which must have gpsc because of their
fiber bundle structure,
an element $([M,\sigma])\in \Omega^{\spinc,+}_n(B\pi)_{(2)}$ if and
only if its image $c_*([M,\sigma])\in ku^+_n(B\pi)_{(2)}$. 
This completes the 2-local part of the proof.

Now we consider the case when $2$ is inverted. This requires different
treatment since, at odd primes, $\MSpin$ is built up
from the Brown-Peterson spectrum $\BP$ and not from
$\ko$ and $H\bZ/2$. However, by the main theorem of
\cite{MR4541288}, for any space $X$, every spin manifold representing
a class in $\pi_j\left(\widehat{\MSpin}\wedge X\right)$
can be chosen to be a fiber bundle with fiber $\bH\bP^2$,
and thus admits positive scalar curvature.  To explain this
a little better, the transfer map $\widehat{\mathcal{T}}_{\spin}$ of Stolz
is unfortunately \emph{not} a split surjection onto
$\widehat{\MSpin}$ at odd primes, but it \emph{is} surjective on homotopy
groups.  Thus to get a bordism theory represented by
$\widehat{\MSpin}$ using psc manifolds at odd primes, one needs to do a 
kind of localization using the Baas-Sullivan theory of
bordism with singularities.  F\"uhring cleverly modifies this
by smoothing to work only with smooth manifolds, obtaining a
bordism theory $\mathcal P^a$ equivalent to $\widehat{\MSpin}$
at odd primes.  Since his theory is built out of spin manifolds
with an $\bH\bP^2$-bundle structure, all of the manifolds represented
by his theory have such a structure.  Since
$\bH\bP^2$ is $3$-connected, any complex line bundle on such
a manifold must be pulled back from the base, and so, when we dualize a line
bundle on such a manifold, the $\bH\bP^2$ fiber structure inherits to
the resultant spin$^c$ manifold, ensuring that it admits gpsc.
This implies that $\widehat{\MSpin}\wedge \Sigma^{-2}\CP^{\infty}$
is a spectrum representing a bordism theory of spin$^c$
manifolds with positive generalized scalar curvature.  So
going back to the fiber sequence \eqref{eq:MSpincfiber1},
we see that we are reduced to studying representatives
for $\widehat{\ku}_*(B\pi)$.

Now observe that $\widehat{\ku}$ inherits a ring spectrum
multiplication $\widehat{\ku}\wedge \widehat{\ku}\to \widehat{\ku}$
from $\MSpinc$, and the $\alpha^c$ map $\MSpinc\to \ku$
factors through $\widehat{\ku}$, giving a ring spectrum
map $\widehat{\alpha^c}\co \widehat{\ku}\to \ku$.  This map
splits (additively), because of the decomposition
$\widehat{\ku}\simeq \bigvee_{\ell\geq 0} \Sigma^{4\ell}\ku$.
Let $\sF$ be the fiber of the map $\widehat{\alpha^c}$:
\begin{equation*}
\sF \to \widehat{\ku}\xrightarrow{\widehat{\alpha^c}} \ku.
\end{equation*}
Clearly we have that $\sF\simeq \bigvee_{\ell\geq
1} \Sigma^{4\ell}\ku$. Hence
we just need to show that the fiber $\sF$ of $\widehat{\alpha^c}$,
$\sF\simeq \bigvee_{\ell\geq 1} \Sigma^{4\ell}\ku$, has the property
that $\pi_*(\sF\wedge B\pi_+)$ is represented  by spin$^c$ manifolds
of gpsc. Consider the simply connected spin$^c$ manifold
$(\bC\bP^2, \cO(1))$ of dimension $4$.  It is not in the image 
of the $\bH\bP^2$ transfer
(since its dimension is too small) but it is in the kernel of
$\alpha^c$. Since the signature of $\bC\bP^2$ is $1$, it 
represents a generator
of $\pi_4(\sF)$.  But $\sF$ is a $\ku$-module spectrum, and it's
clear that we have a filtration of $\ku$-module spectra
\[
\widehat{\ku} \supsetneq \sF \supsetneq \sF^2 \supsetneq \sF^3
\supsetneq \cdots, \quad \bigcap \sF^j = 0,
\]
with $\sF^j$ generated mod $\sF^{j+1}$ as a $\ku$-module spectrum by
the class of $(\bC\bP^2, \cO(1))^j$.  So Theorem \ref{thm:main-ku}
follows.  \hfill $\Box$

\section{The spin$^c$ Gromov-Lawson-Rosenberg Conjecture}
\label{sec:GLRc}
In this section we discuss methods for computing 
$\Omega^{\spinc,+}_n(B\pi)$ from Corollary \ref{cor:possubgroup}, or
its image $ku^+_n(B\pi)$ in $ku_n(B\pi)$, in parallel to results about
the classification of spin manifolds that admit psc metrics (see for
example \cite[\S4]{MR1818778} or \cite[\S1]{MR2408269}).  While, up to
now, we have worked entirely in the context of bordism and homotopy
theory, we now want to make a connection to index theory.  A few
prefatory comments might be in order.  A (compact)
$n$-dimensional spin$^c$ manifold $(M,\sigma)$ has a
homotopy-theoretic fundamental class $[M,\sigma]\in ku_n(M)$,
which is the image of the bordism class
$[(M,\sigma,\id)]\in \Omega^{\spinc}_n(M)$ under the natural
transformation $\alpha^c\co \MSpin^c\to \ku$.%
\footnote{This is described in
\cite[Chapter XI]{MR0248858}; it comes from the $KU$ Thom
class of spin$^c$ bundles along with the universal
property of $\ku$ mentioned above in section \ref{sec:intro}
right after the statement of Theorem C.} After
applying the \emph{periodization map} $\per\co \ku\to \K$, induced by
inverting the Bott element $\beta\in ku_2$, we obtain $\per([M,\sigma])\in
K_n(M)$, the $K$-homology fundamental class.
\begin{remark}
We note that the fundamental class 
$\per([M,\sigma])\in K_n(M)$
\emph{does} depend on the choice of spin$^c$ structure $\sigma$ on $M$.
Changing the spin$^c$ structure as explained in
Remark \ref{rem:changing-the-line-bundle} in 
Appendix \ref{sec:spinc-structures}, which changes the line bundle
$L$ to $L\otimes L'\otimes L'$ for some other line
bundle $L'$ on $M$, which is allowed since $L'\otimes L'$ is even, changes
$\per([M,\sigma])$ to its cap product with the class of $L'$ in $K^0(M)$.
\end{remark}
The homotopy-theoretic $K$-homology fundamental
class $\per([M,\sigma])$ can also be viewed as living in Baum-Douglas
$K$-homology, coming from spin$^c$ bordism by adding ``vector bundle
modification''.  In fact, homotopy-theoretic and ``geometric'' (Baum-Douglas)
$K$-homology both come from spin$^c$ bordism and are naturally isomorphic,
as proved in \cite{MR1624352}.  In turn, there is a natural isomorphism
from Baum-Douglas geometric $K$-homology to Kasparov analytic $K$-homology,
as proved in \cite{MR2330153}.  Composing these two isomorphisms sends the
homotopy-theoretic fundamental class $\per([M,\sigma])$, which is
the image of the spin$^c$ bordism
class $[(M,\sigma,\id)]\in \Omega^{\spinc}_n(M)$, to the Kasparov class
of the spin$^c$ Dirac operator on $M$.  This enables us to link
homotopy theory with index theory.

We will also need the complex \emph{assembly map}
$A\co K_*(B\pi)\to K_*(C^*_r(\pi))$ from $K$-homology of $B\pi$ to the
$K$-theory of the (reduced) group $C^*$-algebra of $\pi$.  This
assembly map is the same as the Baum-Connes assembly map if $\pi$ is
torsion-free, and the Baum-Connes conjecture asserts that it is an
isomorphism in that case.  (This conjecture is known in a large number
of cases, and there are no known counterexamples.  Counterexamples to
``Baum-Connes with coefficients'' exist, but they don't directly apply
to our situation.)
The Mishchenko-Fomenko index theorem
implies that $A$ sends $\per\circ c_*([M,\sigma])\in K_*(B\pi)$
or, equivalently, $c_* [D]\in K_*(B\pi)$, to the index class of the Dirac operator $D$ on $M$ twisted by the Mishchenko-Fomenko bundle 
$\mathcal{V}$: $\Ind (D_{\mathcal{V}})\in K_* (C^*_r \pi)$.

The main tool we have at our disposal is the
following Index Theorem:
\begin{theorem}[Index Theorem]
\label{thm:index}
Let $(M,\sigma)$ be a closed connected spin$^c$ manifold
with $\dim M= n$, fundamental group $\pi$ and classifying map $c\co
M \to B\pi$.  Then $A\circ \per\circ c_*([M,\sigma])\in K_n(C^*_r(\pi))$ is
an obstruction to the existence of positive generalized scalar
curvature on $(M,\sigma)$.  Thus in the notation of
Theorem \ref{thm:main-ku}, $ku^+_n(B\pi)\subseteq \ker (A\circ \per)$.
\end{theorem}
\begin{proof}
The proof is identical to that of \cite[Theorem 3.4]{MR842428},
except for replacement of the spin Dirac operator by the
spin$^c$ Dirac operator.
\end{proof}
An obvious conjecture, which we will show is true for
some fundamental groups and false for others,
is the following:
\begin{conjecture}[Gromov-Lawson-Rosenberg Conjecture, spin$^c$
version]
\label{conj:GLR}
Let $(M,\sigma)$ be a closed connected totally non-spin
$\spin^c$ manifold with $\dim M= n$, fundamental group $\pi$ and
classifying map $c\co M \to B\pi$. Then for $n\ge 5$, $(M,\sigma)$ admits
positive generalized scalar curvature if and only if $A\circ\per\circ
c_*([M,\sigma])=0$ in $K_*(C^*_r(\pi))$.
\end{conjecture}
Thanks to Theorem \ref{thm:index}, this is true in a large
number of cases:
\begin{corollary}
\label{cor:assembinj} 
Let $\pi$  be a finitely presented group 
for which the complex assembly map $A\co K_*(B\pi)$
$\to K_*(C^*_r(\pi))$
and the periodization map $\per\co ku_*(B\pi)\to K_*(B\pi)$ are both
injective.  Then Conjecture \ref{conj:GLR} holds for $\pi$.
\end{corollary}
\begin{proof}
This is identical to the proof of \cite[Theorem 4.13]{MR1818778}.
The ``only if'' direction is the Index Theorem, Theorem
\ref{thm:index}, while the ``if'' direction is reduced
(via injectivity of the maps $A$ and $\per$) to
Theorem \ref{thm:main-ku}.
\end{proof}
Examples of groups satisfying the conditions of 
Corollary \ref{cor:assembinj} are surface groups, free
groups, and free abelian groups.  Any torsion-free group
for which the Baum-Connes Conjecture or the Strong Novikov
Conjecture is known to hold has
injective assembly map $A\co K_*(B\pi)\to K_*(C^*_r(\pi))$.
Injectivity of the periodization map can be studied by comparing
the Atiyah-Hirzebruch spectral sequences for $ku_*(B\pi)$
and for $K_*(B\pi)$ --- see \cite[Figure 5]{MR4703039}.
Additional examples where Conjecture \ref{conj:GLR}
is true may be found in \cite[\S2]{MR1778107}.

\begin{remark}
\label{rem:finitegp}
One important difference from the spin case is that
the assembly map
\begin{equation*}
A^{\spin}\co KO_*(B\pi)\to KO_*(C^{*,\bR}_r(\pi))
\end{equation*}
can be highly non-trivial for even-order finite groups (see
\cite{MR1133900}), whereas for finite groups $\pi$,
the complex $K$-theory
$K_*(C^*_r(\pi))$ is free abelian and thus the
assembly map in complex $K$-theory is trivial (i.e.,
it vanishes except on the summand of $\bZ$ coming from the
trivial group, since $\widetilde{K}_*(B\pi)$ is torsion). Thus Conjecture \ref{cor:assembinj}
for a finite group $\pi$ is equivalent to the statement that
if $(M,\sigma)$, $\dim M \ge 5$, is a closed connected totally non-spin
$\spinc$ manifold with finite fundamental group $\pi$, then
$(M,\sigma)$ admits generalized psc if and only if the universal
cover $\wM$ (with the pull-back line bundle) admits gpsc.
Indeed, the 
``only if'' direction is clear, and the ``if'' direction
assumes that $A\circ\per\circ c_*([M,\sigma])=0$, which just says that
$\alpha^c(M,\sigma)=0$ in $ku_n\cong KU_n$, which is the condition for
gpsc on the universal cover according to Theorem A.
\end{remark}
In spite of the triviality of the assembly map for finite
groups (see Remark \ref{rem:finitegp}), we can often
prove Conjecture \ref{conj:GLR} for certain finite groups
anyway.  One useful tool is 
\cite[Proposition 1.5]{MR1057244} along with its
Corollary 1.6, which reduces the problem for finite groups
to the case of $p$-groups (the proof carries over from the
spin case to the spin$^c$ case without change).

Now we are ready for one of our major results, the proof of
Conjecture \ref{conj:GLR} for finite groups with periodic cohomology.
Recall that a finite group $\pi$ has periodic
cohomology if and only if all of the Sylow subgroups of $\pi$ are either
cyclic or generalized quaternion, see 
\textup{\cite[Theorem XII.11.6]{MR1731415}}.  
The following theorem
is modeled on the main result of \cite{MR1484887}, though the
proof is much easier because $\ku$ is much less complicated than $\ko$.
\begin{theorem}
\label{thm:GLRpercoh}
Let $(M,\sigma)$ be a closed, connected, totally non-spin manifold
with $\dim M = n\geq 5$, and finite fundamental
group $\pi$ having periodic cohomology.  Then $(M,\sigma)$ admits positive
generalized scalar curvature if and only
if $\alpha^c(M,\sigma)=0$ in $ku_n$.
\end{theorem}
\begin{proof}
Obviously the condition $\alpha^c(M,\sigma)=0$ is
necessary, since the map $\per\co ku_n\to KU_n$
is an isomorphism for $n\ge 0$,
and only provides a restriction when $n$ is even, so we
need to prove the sufficiency.  This means we need to 
show that $\widetilde{ku}_n(B\pi)$, the subgroup in $ku_n(B\pi)$ that maps
to $0$ in $ku_n$, is contained in $ku_n^+(B\pi)$, the classes
represented by gpsc.  Then the result will follow from Theorem 
\ref{thm:main-ku}. By Kwasik and Schultz
\cite[Proposition 1.5]{MR1057244}, we can assume $\pi$ is
a $p$-group.  Thus either $\pi$ is cyclic or else $p=2$
and $\pi$ is generalized quaternion.  First assume
$\pi=\bZ/p^r$,
for some prime $p$ and some $r\ge 1$.  Note that the
Atiyah-Hirzebruch spectral sequence 
\begin{equation}
\label{eq:AHSScyclic}
\widetilde H_t(B\bZ/p^r, ku_q) \Rightarrow 
\widetilde{ku}_{t+q}(B\bZ/p^r)
\end{equation}
has nonzero entries only for $q\ge 0$ even and for
$t>0$ odd, so it collapses at $E_2$ for parity reasons.
Thus for $k\ge 1$, $|ku_{2k-1}(B\bZ/p^r)|=p^{kr}$
and the group $\widetilde{ku}_{2k}(B\bZ/p^r)$ vanishes.
By comparison of the spectral sequences \eqref{eq:AHSScyclic} for
$B\bZ/p^r$ and for $B\bZ/p^{r-1}$, we 
obtain for each $2k-1$, $k\ge 1$, the exact sequence
\begin{equation}
\label{eq:transfer}
0 \to ku_{2k-1}(B\bZ/p^{r-1}) \xrightarrow{i_*}
ku_{2k-1}(B\bZ/p^r) \xrightarrow{\text{tr}}
ku_{2k-1}(B\bZ/p)\to 0,
\end{equation}
where $i\co \bZ/p^{r-1}\to \bZ/p^r$ is the inclusion and $\text{tr}$
is the transfer (we learned this argument from \cite{MR1057244}).
Since the homomorphism $i_*$ and the transfer both
send classes in $ku^+_{2k-1}(\cdot)\subset
ku_{2k-1}(\cdot)$ to $ku^+_{2k-1}(\cdot)\subset ku_{2k-1}(\cdot)$, we can use
\eqref{eq:transfer} and induction on $r$ to reduce to the
case $\pi=\bZ/p$.

Now the groups $ku_*(B\bZ/p^r)$ were computed in
\cite[Theorem 3.1]{MR716975}, and we use that calculation here.

Recall that the group $ku_{2k-1}(B\bZ/p^r)$ is an abelian
group of order $p^k$. In fact, it is a direct sum of $N=\min(k, p-1)$
cyclic groups of orders $t(i)=p^{\bar a-s}$, where
(see \cite[p.\ 770]{MR716975})
\begin{equation*}
\begin{array}{c}
1\le i\le N, \ \ 0\le d<p^s(p-1), \ \
k-p^s+1=a_{s,k}p^s(p-1)+b_{s,k}, \ \ 
\\
\\
0\le b_{s,k}<p^s(p-1)\ \ \mbox{and}\ \ \bar a =
\left\{\begin{array}{ll}
a_{s,k}+1 \ \ & \mbox{if} \ \ d<b_{s,k},
\\
a_{s,k} \ \ & \mbox{if}  \ \ d\ge b_{s,k}
\end{array}\right.
\end{array}
\end{equation*}
Denote by $N^{2k-2j}$ a spin$^c$ manifold (here we omit a choice of
spin$^c$ structure on $N^{2k-2j}$) such that $\alpha^c(N^{2k-2j})$ is
a generator of $ku_{2k-2j}$. Then the $E_2=E_\infty$ terms
in \eqref{eq:AHSScyclic} are generated by spin$^c$ manifolds of the
form $L^{2j-1}\times N^{2k-2j}$, where $L^{2j-1}=S^{2j-1}/\bZ_{p^{\bar
a-s}}$ is a lens space.
These all have gpsc \textbf{except} when $j=1$ and $L^{2j-1}$
collapses to a circle $S^1$.  So we only need to worry about the
contribution of $H_1(B\bZ/p, ku_{2k-2})$.  Furthermore, it's enough
just to show that $ku_3(B\bZ/p)$ is represented by lens spaces since
the higher-dimensional
$ku$-groups $ku_{2k-1}(B\bZ/p)$ are generated by
this and by manifolds of the form $L^{2j-1}\times N^{2k-2j}$ with
$j\ge 2$.

Specializing Hashimoto's Theorem to dimension $3$, we
obtain that $ku_3(B\bZ/p)$ is cyclic if $p=2$ and is isomorphic to
$\bZ/p\oplus \bZ/p$ for $p$ odd.  When $p=2$,
$ku_3(B\bZ/2)\cong \bZ/4$ and is generated by the class of the psc
spin manifold $\bR\bP^3$.  Also, the relative eta
invariant \cite{MR1339924} shows that that for $p\ge 5$ odd, there are
different $3$-dimensional lens spaces $L^3(p;a)$ with linearly
independent mod-$p$ characteristic classes, which then generate
$ku_3(B\bZ/p)\cong \bZ/p\oplus \bZ/p$.

The case $p=3$ requires different treatment, since in this case the
image of the relative eta invariant only has rank one. Indeed,
by \cite[Theorem 2.1]{MR1339924}, the rank of the image of the
relative eta invariant for manifolds with fundamental group $\bZ/p$
for $p$ an odd prime in dimensions $3$ mod $4$ is $\rank R_0^+(\bZ/p)
= \frac{p-1}{2}$.  However, the group $\Omega^\spin_5(B\bZ/3)$ is
cyclic of order $9$ \cite[Theorem 36.1]{MR176478}, generated by a lens
space, and $ku_5(B\bZ/3)$ is generated by that lens space and by the
image of $L^3(3;1)\times (\bC\bP^1, \cO(2))$, both of which are 
manifolds with gpsc.  Putting everything together, we conclude
that $ku^+_{2k-1}(B\pi)$ coincides with $ku_{2k-1}(B\pi)$ for $\pi$ a
cyclic $p$-group and $k\ge 3$.

Next we consider the remaining case of $p=2$ and $\pi=Q_{2^k}$ generalized
quaternion of order $2^k$, $k\ge 3$.  In this case the
homology of $\pi$ with integer coefficients is $\bZ/2\oplus\bZ/2$
in degrees $1$ mod $4$, $\bZ/2^k$ in degrees $3$ mod $4$
\cite[Ch.\ XII, \S7]{MR1731415}, and $0$ in even positive
degrees.  Thus the Atiyah-Hirzebruch spectral sequence 
\begin{equation}
\label{eq:AHSSquat}
\widetilde H_t(BQ_{2^k}, ku_q) \Rightarrow 
\widetilde ku_{t+q}(BQ_{2^k})
\end{equation}
again collapses at $E_2$, and $ku_{2r-1}(BQ_{2^k})$ is an iterated
extension of copies of $\bZ/2^k$ and $\bZ/2\oplus\bZ/2$, while reduced
$ku$ vanishes in even degrees.  We need to show that
$ku_{2r-1}^+(BQ_{2^k})=ku_{2r-1}(BQ_{2^k})$, $r\ge 3$.  Now we have
some obvious manifolds which represent elements in the subquotients
$E^\infty_{2s-1,2(r-s)} = H_{2s-1}(BQ_{2^k}, ku_{2(r-s)})$, namely the
quaternionic lens space
$L^{2s-1}_Q:=S^{2s-1}/Q_{2^k}$ (when $s$ is even, so
we are in dimension $3$ mod $4$) and the real lens spaces
$M_C^{2s-1}:=S^{2s-1}/C$, where $C$ a cyclic subgroup of $Q_{2^k}$ (when $s$
is odd, so we are in dimension $3$ mod $4$).  We can take
$L^\infty_Q=\varinjlim L^{2s-1}_Q$
as a model for $BQ_{2^k}$, and inside this space, $L^{2s-1}_Q$ is the
($2s-1$)-skeleton and, as a homology cycle, represents a generator of
$H_{2s-1}(BQ_{2^k})\cong \bZ/2^k$.  Since $L^{2s-1}_Q$ is spin, psc, and
\emph{a fortiori} spin$^c$, we see that $E^\infty_{2s-1,2(r-s)}$
is represented by spin$^c$ manifolds of gpsc for $s$ even.

It remains to deal with $E^\infty_{2s-1,2(r-s)}$
for $s$ odd.  However, for any cyclic subgroup $C$ of $\pi=Q_{2^k}$,
the inclusion $C\hookrightarrow\pi$  induces the map
$C\to \pi/[\pi,\pi]\cong \bZ/2\oplus\bZ/2$ on $H_1$.  So
if we write 
\[
\pi = \langle x, y \mid x^{2^{k-2}}=y^2, xyx = y\rangle,
\]
where $x=\exp\left(\frac{\pi i}{2^{k-2}}\right)$ 
(not the same $\pi$!) and $y=j$ inside
the usual quaternions $\bH$, choosing
$C_1=\langle x\rangle$ and $C_2=\langle xy\rangle$
we see that the map $H_1(C_1)\oplus H_1(C_2)\to H_1(\pi)$
induced by the inclusions of $C_1$ and $C_2$ is surjective.
Since the restriction of the periodicity element
in $H^4(\pi;\bZ)$ to $H^4(C_\ell;\bZ)$, $\ell=1,2$,
is also a periodicity element \cite[Proposition 11.4]{MR1731415},
the map $H_{2s-1}(C_1)\oplus H_{2s-1}(C_2)\to H_{2s-1}(\pi)$
is again surjective for all odd $s$.  But by the case of
cyclic groups which we have already handled, $H_{2s-1}(C_\ell)$,
$\ell=1,2$, $s\ge 2$, is represented by spin$^c$ lens spaces of gpsc.
There is only one last step, since $H_1(C_\ell)$ is \emph{not}
represented by a gpsc manifold.  However, as we saw before,
$ku_3(B\bZ/2^t)$ is generated classes of gpsc manifolds, for any
$t$.  We claim that $ku_3(B\pi)$ is generated by the classes
of spin$^c$ manifolds with gpsc, namely the $3$-dimensional
quaternionic lens space $L^3$ and the spherical space forms
$S^3/C_\ell$.  To see this, consider the
commutative diagram
\begin{equation*}
\xymatrix{H_1(BC_1;ku_2)\oplus\rule[-8pt]{0pt}{8pt}
H_1(BC_2;ku_2)\ar@{->>}[r]\ar@{>>->}[d]& 
H_1(B\pi\rule[-8pt]{0pt}{8pt};ku_2)\ar@{>>->}[d]\\
ku_3(BC_1)\oplus ku_3(BC_2)\ar[r]& \,ku_3(B\pi), }
\end{equation*}
where the downward arrows are the edge homomorphisms in the
Atiyah-Hirzebruch spectral sequence.  The diagram shows that the image
of $H_1(B\pi;ku_2)$ in $ku_3(B\pi)$ is generated by the images of
$ku_3(BC_1)$ and $ku_3(BC_2)$, which are represented by gpsc
manifolds.  In higher dimensions $2s-1$, $s\ge 3$, the image of
$H_1(B\pi;ku_{2s-2})$ in $ku_{2s-1}(B\pi)$ factors through
$ku_3(B\pi)$ via Bott periodicity (take a spin$^c$
$3$-manifold with gpsc and take its product with a power of
$(S^2,\cO(2))$ representing the Bott element). We
obtain $ku_{2s-1}(B\pi)=ku_{2s-1}^+(B\pi)$ as required,
hence the theorem follows from Theorem \ref{thm:main-ku}.
\end{proof}
While Theorem \ref{thm:GLRpercoh} provides promising
evidence for  Conjecture \ref{conj:GLR} in many cases,
we can modify the proof given by Thomas Schick in \cite{MR1632971} to show
that Conjecture \ref{conj:GLR} is false for
at least some infinite groups with torsion.
\begin{theorem}
\label{thm:Schickargument}
There exist totally non-spin spin$^c$ $5$-manifolds
with fundamental group $\pi=\bZ^4\times \bZ/p$,
$p$ a prime, which have vanishing generalized Dirac index 
in $K_5(C^*_r(\pi))$ but which do not admit positive
generalized scalar curvature, in fact which do not admit
positive scalar curvature at all.
\end{theorem}
\begin{proof}
Fix a prime $p$ and recall from the proof of Theorem
\ref{thm:GLRpercoh} that ${ku}_{2k-1}(B\bZ/p)$
is torsion of order $p^k$ for any $k\ge 1$.  Furthermore,
the Atiyah-Hirzebruch spectral sequence for computing
${ku}_*(B\bZ/p)$ collapses at $E^2$, so the edge 
homomorphism ${ku}_{2k-1}(B\bZ/p)\to H_{2k-1}(B\bZ/p;\bZ)
\cong \bZ/p$ is surjective.  Now take $\pi=\bZ^4\times \bZ/p$
and define $N^5=(T^4\# \bC\bP^2)\times S^1$, which is
a totally non-spin spin$^c$ manifold with line bundle
coming from $\cO(3)$ on the $\bC\bP^2$ summand.  Now define
\[
f\co N^5=(T^4\# \bC\bP^2)\times S^1
\xrightarrow{c\times g}
T^4\times B\bZ/3=B\pi,
\]
where $c$ is the classifying map for $T^4\# \bC\bP^2$ (collapsing out
the $\bC\bP^2$) and $g\co S^1\to B\bZ/p$ induces a surjection on
$\pi_1$.  Do spin$^c$ surgery on an embedded circle in $(N,f)$ to
create a totally non-spin spin$^c$ manifold $M^5$ with fundamental
group exactly $\pi$. In particular, the map $f$
determines a map $f_1 \co M\to B\pi$ inducing isomorphism of the
fundamental group. Note that $(M,f_1)$ represents a $p$-torsion
element in $\Omega^{\spinc}_5(B\pi)$, so its image in $ku_5(B\pi)$ is
also $p$-torsion and goes to the homology class $x_1\times x_2\times
x_3\times x_4\times y$ in $H_5(B\pi;\bZ)\cong \bZ^4\times \bZ/p$,
where $x_1,\cdots, x_4$ are a basis for $H_1(T^4;\bZ)\cong \bZ^4$ and
where $y$ is a generator of the cyclic group
$H_1(B\bZ/p;\bZ)\cong \bZ/p$.  Now we argue exactly as in
in \cite{MR1632971}.  The manifold $M$ must have vanishing index since
its spin$^c$ bordism class is $p$-torsion but $K_5(C^*(\pi))$ is
torsion-free.  So if Conjecture
\ref{conj:GLR} were true, $M$ would admit gpsc.
But capping with the cohomology classes dual to
$x_2,\cdots, x_4$ sends the homology class of $M$
to $x_1\times y\ne 0 \in H_2(B\pi;\bZ)$, and also
would send psc classes to psc classes
by \cite[Theorem 1]{MR535700}.  This gives a contradiction
since any oriented psc $2$-manifold is $S^2$ and
cannot represent a non-zero homology class in $H_2(B\pi;\bZ)$.
\end{proof}

\section{The Stolz sequence in the spin$^c$ case}
\label{sec:Stolz}
In this section we define the
spin$^c$ analog of the Stolz surgery sequence
from \cite{StolzConc}
and we investigate its basic properties.

\begin{remark}\label{rem:gamma}
Stolz introduced the classifying space $BG(\gamma, n)$, which, for
spin manifolds with fundamental group $\pi$, is simply the product
$B\Spin(n) \times B\pi$. The obvious counterpart of this for
$n$-dimensional spin$^c$ manifolds with fundamental group  $\pi$ is the
classifying space $B\Spin^c(n) \times B\pi$, which we denote by
$B\gamma$ for short (when the $n$ is understood).
We notice that a spin$^c$ $n$-manifold $M$ with
fundamental group $\pi$ will always have a classifying map to
$B\gamma$; however, this can be a 2-equivalence only when $M$ is
totally non-spin.
In that case, the universal cover
$\wM$ will have to have $H^2(\wM;\bZ)\ne 0$, and by the universal
coefficient theorem and the Hurewicz theorem,
$H^2(\wM;\bZ)\cong \Hom(\pi_2(M),\bZ)$. So there will be a nonzero
map $\pi_2(M)\to\pi_2(B\Spinc (n))\cong\bZ$, but it may not be
surjective. We need one more classifying space:
$B\gamma'=K(\bZ/2,2)\times B\pi$.
For an $n$-dimensional spin$^c$
manifold $M$ with fundamental group $\pi$, we have a commutative
diagram
\begin{equation}\label{eq:classifying}
\xymatrix{& B\gamma\ar[d]^{(c_1\text{ mod 2}) \times Id}\\
M \ar[ur]^c \ar[r]^{c'} & \,B\gamma',}
\end{equation}
where $c$ is the classifying map of the universal cover of $M$
paired with the classifying map for the spin$^c$ structure and
$c'$ is classifying map of the universal cover of $M$
paired with the Stiefel-Whitney class $w_2$.  If $M$ is
totally non-spin, then $c'$ \emph{will} be a $2$-equivalence. 
\end{remark}

We summarize some of the relevant facts in the following
lemma. It is convenient to call a complex line bundle
$L$ \emph{even} if its first Chern class reduces mod $2$ to zero.
\begin{lemma}
\label{lem:Bgamma}
Let $(M,\sigma,L)$ be an $n$-dimensional connected spin$^c$ manifold,
$n\ge3$, with fundamental group $\pi$.  Let $B\gamma=B\Spinc(n)\times
B\pi$ and $B\gamma'=K(\bZ/2,2)\times B\pi$.  Then $M$ has classifying
maps $c\co M\to B\gamma$ and $c'\co M\to B\gamma'$ related as in
\textup{(\ref{eq:classifying})}.  If $M$ is totally non-spin, then $c'$ is a
$2$-equivalence, i.e., induces an isomorphism on $\pi_1$ and a
surjection on $\pi_2$.  Furthermore, after 
changing the spin$^c$ structure $\sigma$, which changes the spin$^c$ line
bundle from $L$ to $L\otimes L'$, where $L'$
is even, one can arrange for $c$ also to be a
$2$-equivalence.
\end{lemma}
\begin{proof}
The only thing that needs proof is the last statement.  Since
$\pi_2(B\Spinc(n))\cong\bZ$, the homomorphism
$c_*\co \pi_2(M)\to \pi_2(B\Spinc(n))$  maps $\pi_2(M)$ onto
a subgroup of $\bZ$ of odd index, and then adding an even element onto
a generator of the image by using Remark
\ref{rem:changing-the-line-bundle}, we can arrange surjectivity.
\end{proof}
\begin{definition}\label{def:pos^c}
Let $X$ be a space (say a CW complex). We denote by 
${\rm Pos}^{\spinc}_n (X)$ the group of
equivalence classes of 4-tuples $((M, \sigma,L), \varphi\co M\to X,
g,\nabla_L)$, where $(M,\sigma,L)$ is a compact spin$^c$
manifold, $\dim M = n$, without boundary with
spin$^c$ structure $\sigma$ and associated line bundle $L$, $g$ a
Riemannian metric on $M$, and $\nabla_L$ a connection on $L$ such that
the twisted scalar curvature $\frac14 R_{g}
+ \frac{i}{2}c(\Omega_{L})$ of $(g,\nabla_L)$ is positive on all of
$M$.
%
Namely, two 4-tuples
\begin{equation*}
((M_0,\sigma_0,L_0),  \varphi_0\co
 M_0\to X, g_0,\nabla_{L_0})
 \ \  \mbox{and} \ \ ((M_1,\sigma_1,L_1), \varphi_1\co M_1\to X,
 g_1,\nabla_{L_1})
\end{equation*}
are said to be equivalent if there exists a
$4$-tuple $((W,\sigma_W,\mathcal{L}), \Phi: W\to X, \bar
g, \nabla_{\bar g})$, where $(W,\sigma_W,\mathcal{L})$ is spin$^c$
manifold with boundary $\partial (W,\sigma_W)=(M_0,\sigma_0)\sqcup -
(M_1,\sigma_1)$, endowed with a Riemannian metric $\bar g$, a
continuous map $\Phi\co W\to X$, and a connection
$\nabla_{\mathcal{L}}$ on the associated line bundle $\mathcal{L}$
such that
\[
\bar g|_{\partial W}=g_0\sqcup g_1, \quad \Phi
  |_{\partial M}= \varphi_0\sqcup \varphi_1\quad \cL|_{\partial
  W}=L_0\sqcup L_1, \quad \nabla_{\cL}|_{\partial
  W}= \nabla_{L_0}\sqcup \nabla_{L_1}
\]
and the generalized scalar  curvature of $(\bar g,\nabla_{\cL})$ is positive
on $W$. All the structures are required to be product-like
near the boundary.

The group law on ${\rm Pos}^{\spinc}_n (X)$ is induced by disjoint
union. The inverse comes from reversing the 
spin$^c$ structure, since $M\times I$ with the product metric
and bundle and connection pulled back from $M$ shows
that $((M,\sigma), \varphi\co M\to X, g, \nabla_L) \sqcup
((M,-\sigma), \varphi\co M\to X, g, \nabla_L)$ is a boundary.
\end{definition}

\begin{definition}\label{def:R^c}
Let $X$ be a space, say a CW complex. 
We denote by ${\rm R}^{\spinc}_n (X)$
the group of equivalence classes of 4-tuples $((M,\sigma,L),\varphi\co
M\to X, g_{\partial M}$, $\nabla_{L,{\partial}})$, where
$(M,\sigma,L)$ is a compact spin$^c$ manifold,
$\dim M=n$, with (possibly empty) boundary $\p M$ and associated line
bundle $L$, $g_{\partial M}$ is a Riemannian metric on $\partial M$,
and $\nabla_{L,{\partial}}$ a connection on $L|_{\partial M}$ such
that the generalized scalar curvature of $(g_{\partial
M},\nabla_{L,{\partial}})$ is positive on all of $\partial M$. Two
4-tuples
\begin{equation*}
((M_0, \sigma_0,L_0), \varphi_0\co M_0\to X, g_{\partial
M_0},\nabla_{L_0,\partial}) \ \ \mbox{and} \ \ ((M_1,\sigma_1, L_1), \varphi_1\co M_1\to X,
g_{\partial M_1},
\nabla_{L_1,\partial})
\end{equation*}
are equivalent if there exist a spin$^c$ manifold with corners 
$(W,\sigma_W,\mathcal{L})$ endowed
with a Riemannian metric $\bar g$ and a continuous map
$\Phi\co W\to X$ satisfying the following properties:
\begin{itemize}
\item the codimension 1 faces of $W$ are
$M_0$, $M_1$ and a bordism $V$ from $\partial M_0$ to $\partial M_1$;
\item  the only codimension 2 corners are  $\partial M_0=M_0\cap V$ and $\partial M_1=M_1\cap V$;
\item $\cL$  restricts to $L_j$ and $\Phi$ restricts to $\varphi_j$ on $M_j$, $j=0,1$;
\item there exists a Riemannian metric $g_V$ on $V$ and there exists a connection $\nabla_V$ on $L_V:=\mathcal{L}|_{V}$ 
such that 
their restrictions to $\partial V=\partial M_0 \sqcup
- \partial M_1$ are given by
\begin{equation*}
g_{\partial M_0}\sqcup g_{\partial
M_1}\;\;\text{and}\;\; \nabla_{L_0,\partial}\sqcup \nabla_{L_1,\partial}
\end{equation*}
and such that $(V, L_V)$ has positive generalized scalar curvature.
Moreover, all the structures are required to be product-like near the boundary.
\end{itemize}
The group law on ${\rR}^{\spinc}_n (X)$ is
induced by disjoint union.
\end{definition}

\noindent
{\bf Notation.} Unless confusion should
arise we shall write a cycle for the bordism group
${\rm Pos}^{\spinc}_*(X)$ simply as $(M, \varphi\co M\to X,
g,\nabla_L)$ and we shall write a cycle for the
group ${\rm R}^{\spinc}_* (X)$ simply as $(M, \varphi\co M\to X,
g_{\partial M},\nabla_{L,\partial})$.

\noindent
As for the classic spin-Stolz sequence
we observe that there are natural homomorphisms:
\begin{equation*}
{\rR}^{\spinc}_{*+1} (X)\xrightarrow{\partial} {\rm Pos}^{{\rm spin}^c}_* (X)\,; \quad
\Omega^{\spinc}_* (X)\xrightarrow{\iota} {\rR}^{\spinc}_{*} (X),
\end{equation*} 
the first induced by the passage to the boundary and the second
obtained by considering a closed manifold as a manifold with empty
boundary. We also have a forgetful homomorphism
\begin{equation*}
{\rm Pos}^{\spinc}_* (X)\rightarrow
\Omega^{\spinc}_* (X)
\end{equation*} 
which sends a class $[(M, \varphi\co M\to X,
g,\nabla_L)]$ in ${\rm Pos}^{\spinc}_* (X)$ to $[(M, \varphi\co M\to
X)]\in \Omega^{\spinc}_* (X)$.

\begin{proposition}
The groups and the homomorphisms that we have defined give a long
exact sequence
\begin{equation*}
\cdots\rightarrow {\rR}^{\spinc}_{*+1} (X)\xrightarrow{\partial} {\rm Pos}^{\spinc}_* (X)
\rightarrow
\Omega^{\spinc}_* (X)\xrightarrow{\iota} {\rR}^{\spinc}_{*} (X)\rightarrow \cdots
\end{equation*} 
\end{proposition}
\begin{proof}
The relatively simple proof
from \cite[(4.4)]{StolzConc} in the spin
case carries over to the spin$^c$ case.
\end{proof}
If $f\co X\to Y$ is a continuous map, then we have a covariant group
homomorphism $f_*$ induced on these groups, simply by replacing
$\varphi\co M\to X$ with $f\circ \varphi\co M \to Y$.
\begin{theorem}\label{thm:R^c-2-equivalence}
If $f:X\to Y$ is a 2-equivalence, then the induced
homomorphism
\begin{equation*}
f_* \co  {\rR}^{\spinc}_{*} (X)\rightarrow  {\rR}^{\spinc}_{*} (Y)
\end{equation*} 
is an isomorphism.
\end{theorem}
\begin{proof}
The proof given by Puglisi, Schick and Zenobi in the spin
case \cite[Theorem 4.1]{puglisistolzgroups} carries over to this more
general situation once we use the surgery theorem in the spin$^c$
case. Indeed, there are purely topological arguments
in \cite{puglisistolzgroups} that carry over to the spin$^c$ situation
unchanged; there are then two crucial applications of the
Gromov-Lawson surgery theorem, both for the proof of the surjectivity
and for the proof of the injectivity.  These two applications of the
surgery theorem are done on the trace of surgeries of codimension
at least three;
in particular, there are no cancellation arguments involved
and for this reason we can apply Theorem \ref{thm:Crelle4.2} and
proceed as in \cite{puglisistolzgroups} in order to complete the
proof.  We omit the details.
\end{proof}
\begin{corollary}
If $M$ is a totally non-spin spin$^c$ manifold
with fundamental group $\pi$, then the
natural maps 
\[\rR^{\spinc}_{*} (M)\xrightarrow{c'_*} 
    \rR^{\spinc}_{*} (B\gamma')
    \xrightarrow{(\textup{\text{pr}}_2)_*}  \rR^{\spinc}_{*} (B\pi),
\]
with $B\gamma'$ and $c'$ as in Remark \ref{rem:gamma}, are
isomorphisms.
\end{corollary}
\begin{proof}
Since $M$ is totally non-spin, it
guarantees that $c'\co M\to B\gamma'$ is a
$2$-equivalence, and the projection $\mathrm{pr}_2:
B\gamma' = K(\bZ/2,2)\times B\pi$ is also a $2$-equivalence.
\end{proof}
Next, we wish to prove the following extension result
(cf.\ \cite[Theorem 5.4(1)]{StolzConc}):
\begin{theorem}[Extension Theorem]
\label{th:extension-theorem}
Let $M$ be {\bf totally non-spin} spin$^c$ manifold with boundary
$\p M$, and $\dim M=n\geq 5$.
Let $\pi=\pi_1 (M)$ and let $\varphi\co M\to B\pi$ be a 2-equivalence;
for example $\varphi$ can be the classifying map for
the universal cover of $M$.  Then the following holds:
\begin{enumerate}
\item Let $(W^{n+1},\sigma,\cL)$ be a spin$^c$ manifold with corners,
whose boundary is decomposed as $\partial W=M\cup_{\partial M} V$,
where $M$ and $V$ have the same boundary and the inclusion
$M\hookrightarrow W$ is $2$-connected.  Then,
if a metric $g_V$ and a connection
$\nabla_V$ exist on $V$ such that $(g_V,\nabla_V)$ has gpsc, there exists a
pair $(g_W,\nabla_{\cL})$ extending $(g_V,\nabla_V)$ and with gpsc on
all of $W$.
\item A class
$[(M,\varphi\co M\to B\pi, g_{\partial M}, \nabla_{L_{\partial M}})]=0$ in $\rR^{{\rm spin}^c}_* (B\pi)$ if and only if
$(g_{\partial M}, \nabla_{L_{\partial M}})$ extends to $(g,\nabla_L)$ 
with gpsc on $M$.
\end{enumerate}
\end{theorem}

\begin{proof}
(1) See Figure \ref{fig:extension}.
Assume that the inclusion $M\hookrightarrow W$ is a $2$-equivalence. 
Then we can apply \cite[Theorem 3]{MR290387} to conclude that
that the inclusion is also a geometric $2$-equivalence, i.e.,
has no handles of index $\le 2$, and thus the bordism $W$ from
$V$ to $M$ (since Wall's result holds
in the category of manifolds with boundary and
bordisms with corners) only involves surgeries in codimension
$\ge 3$.  Thus applying Theorem \ref{thm:Crelle4.2}, as extended
to manifolds with boundary as in \cite{MR962295}, we know that there exists $(g_W,\nabla_{\cL})$ extending $(g_V,\nabla_V)$
and with generalized psc on all of $W$.
\begin{figure}[htb]
\begin{picture}(10,0)
\put(150,5){{\small $\p M$}}
\put(142,80){{\small $V$}}
\put(10,-5){{\small $M$}}
\put(20,70){{\small $W$}}
\end{picture}
\includegraphics[height=1.3in]{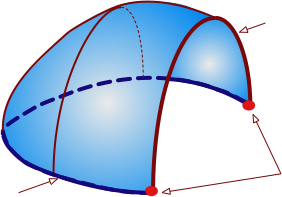}
\caption{Set-up of the Extension Theorem}
\label{fig:extension}
\end{figure}
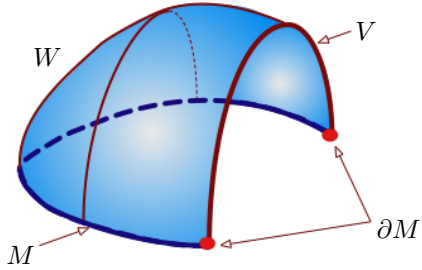

\noindent
(2) To say that a class $[(M,\varphi\co M\to B\pi, g_{\partial
M}, \nabla_{L_{\partial}})]=0$ in $\rR^{\spinc}_* (B\pi)$ means that
there exists a spin$^c$ manifold with corners $(W,\sigma_W,\cL)$ 
endowed with a map $\Phi\co
W\to B\pi$ satisfying the following properties:
\begin{itemize}
\item 
the  codimension 1 faces of $W$ are $M$ itself and a bordism $V$  
from $\partial M$  to the empty set;
\item  the only codimension 2 corner is $\partial M=M\cap V=\partial V$;
\item $\cL$  restricts to $L$ and $\Phi$ restricts to $\varphi$ on $M$;
\item there exists a Riemannian metric $g_V$ on $V$ restricting to 
$g_{\partial M}$ on $\partial V=\partial M$
and there exists a connection $\nabla_V$ on $L_V:=\mathcal{L}|_{V}$ 
restricting to $\nabla_{L_{\partial M}}$
on $\partial V=\partial M$ so that $(g_V,L_V)$ has gpsc.
\end{itemize}
Assume that $(g_{\partial M}, \nabla_{L_{\partial M}})$ 
extends to $(g,\nabla_L)$ with gpsc on $M$. Then
we can then choose $V=M$ and $W=M\times [0,1]$ (with the corners rounded)  
so that $\cL={\rm pr}^*_1 L$, ${\rm pr}_1=$ projection
from $M\times [0,1]$ onto $M$,
and then obtain that 
$[(M,\varphi\co M\to B\pi, g_{\partial M}, \nabla_{L_{\partial M}})]=0$ 
in $\rR^{\spinc}_* (B\pi)$.

Let us now prove the converse. Because we have 
already proved (1), it suffices 
to show that there exists a modification $W'$ of
$W$ with the additional property that $M\hookrightarrow W'$ is a 2-equivalence.
This, however, is established precisely as in the proof of Theorem \ref{thm:spincbordism}. Indeed, there exists a split surjection $\pi_1 (W)\to \pi$, given that $\pi_1 (M)$ is isomorphic to
$\pi$. Then
we can do spin$^c$ surgeries on 1-spheres with trivial normal bundles  in the interior of $W$, as in the proof of  Theorem \ref{thm:spincbordism}, so as to obtain $W''$
with $\pi_1(W'')$ isomorphic to $\pi_1 (M)$ and 
$w_2 (W'')= c_1 (\cL) \mod 2$. 
Next,  using the additional hypothesis
that $M$ is totally non-spin and proceeding as in the proof of Theorem \ref{thm:spincbordism}, we can modify $W''$ to $W'$
by additional spin$^c$ surgeries to
arrange so that $\pi_2 (M)$ surjects onto $\pi_2 (W')$, still ensuring that $w_2 (W')= c_1 (\mathcal{L}) \;{\rm mod} \; 2$.  This implies
that $M\hookrightarrow W'$ is a 2-equivalence.
\end{proof}

The following Lemma will be needed later when we 
want to apply the Extension Theorem
(Theorem \ref{th:extension-theorem}), in order to ensure that
the ``totally non-spin hypothesis'' is met.

\begin{lemma}
\label{lem:makingtotallynonspin}
Let $\pi$ be a finitely presented group.
Every class in $\rR^{\spinc}_n (B\pi)$ with $n\geq 6$ is represented by some 
4-tuple
$(M,\varphi\co M\to B\pi, g_{\partial M}, \nabla_{L_{\partial}})$ with $M$ spin$^c$,
$\varphi_*$ an isomorphism on $\pi_1$, and $M$ totally non-spin.
\end{lemma}
\begin{proof}
We start by showing that $0\in \rR^{\spinc}_n (B\pi)$ has such a
representative. To this end we consider a finite 2-complex $K$ with
$\pi_1 (K)$ isomorphic to $\pi$. Embed $K$ into $\bR^n$ and let $M'$
be a closed regular neighborhood of $K$ in $\bR^n$. $M'$ is a
codimension 0 submanifold of $\bR^n$, hence parallelizable and spin.
Note that $\pi_1 (\partial M')\cong \pi_1 (M')\cong \pi$. By
\cite[Theorem 2.1]{MR4703039}, $M'$ admits a psc metric,
which we can consider as a gpsc metric with respect to
the trivial line bundle and the flat connection on it.

The manifold $\bC \bP^2\times S^{n-4}$ with line bundle coming from
$\cO (1)$ on $\bC \bP^2$ is simply connected, totally non-spin, with
gpsc with respect to the product of the Fubini-Study metric on
$\bC\bP^2$ and the round metric on $S^{n-4}$, using the standard
connection on $\cO (1)$.  Define
\begin{equation*}
M:= M'\,\sharp\, (\bC \bP^2\times S^{n-4}),
\end{equation*}
Then $M$ comes equipped with a line bundle $L$, obtained from $\cO (1)$ on $\bC\bP^2$ in
the obvious way, ensuring that $M$ is indeed spin$^c$. 
Moreover, $M$ is totally non-spin, with
fundamental group $\pi$.
Choose a spin$^c$ structure $\sigma$ inducing such a line bundle.
  By the Surgery Theorem,
Theorem \ref{thm:Crelle4.2}, $(M,\sigma,L)$ admits gpsc via the $0$-surgery
that converts $M' \sqcup (\bC \bP^2\times S^{n-4})$ to $M$.

The cases of 
non-zero bordism classes can be handled similarly,
by first reducing via surgery to the case where the
classifying map to $B\pi$ is an isomorphism on $\pi_1$,
and then taking a connected sum with $\bC \bP^2\times S^{n-4}$.
\end{proof}

\section{Concordance classes and action of the $\rR$-group}
\label{sec:conc}
Let $(M,\sigma,L)$ be a connected compact spin$^c$ manifold, for the time
being without boundary. Let $\pi=\pi_1 (M)$. For a
cylinder $M\times [0,1]$, we denote by ${\rm pr}_1\co M\times [0,1]\to
M$ the projection on the first factor.
\begin{definition}\label{def:concordance^c}
Two pairs $(g_0,\nabla^0)$ and $(g_1,\nabla^1)$ on
$(M,\sigma,L)$ with gpsc are \emph{concordant} if there exist a metric $\bar
g$ on $M\times [0,1]$ and a connection $\nabla^{\cL}$ on $\cL={\rm
pr}_1^*L$,
such that
$(\bar g,\nabla^{\cL})$ has generalized positive scalar curvature,
restricts to $(g_j,\nabla^j)$ on $M\times \{j\}$, $j=0,1$, and is of
product type in a neighborhood of the boundary.

The concordance set $\cP^c (M,\sigma,L)$ (or just $\cP^c (M,L)$
when $\sigma$ is understood) is defined as the set of concordant classes of
pairs $(g,\nabla)$ with gpsc.
%
%
%
\end{definition}
Our objective in this section is to show how the concordance
set $\cP^c (M,L)$ can be
studied via the $\rR^{\spinc}_{*}(X)$-groups of
section \ref{sec:Stolz}, following the paradigm in
\cite[\S5]{StolzConc}.
The set $\cP^c (M,L)$ has a group structure 
\textbf{if $M$ is totally non-spin}. This group
structure depends on a choice of a reference pair $(g_0,\nabla^0)$
satisfying the gpsc condition. We describe this group structure now,
following the treatment in the spin case given by
Weinberger-Yu \cite[\S4]{MR3416114}.

Fix a reference pair $(g_0,\nabla^0)$ satisfying the gpsc
condition. Let $[(g,\nabla)]$ and $[(g',\nabla^{\prime})]$ be two
elements in $\cP^c (M,L)$; we want to define their sum.  Consider two
cylinders $M\times [0,1]$ and then take their
connected sum away from their boundaries
\begin{equation*}
(M\times [0,1])\,\# 
(M\times [0,1])\,.
\end{equation*}
The fundamental group of $(M\times [0,1])\,\sharp\, (M\times [0,1])$
is $\pi * \pi$; we do spin$^c$ surgeries in the interior so as to
ensure that we get a manifold $Z$ with fundamental
group $\pi$.  Thus $Z$ comes with a classifying map
into $B\pi$.  Now we consider the boundary of $Z$
which is the union of 4 copies of $M$; consider $(g,\nabla)$ and
$(g',\nabla^{\prime})$ on the first two copies, those corresponding to
the first cylinder, and $(g_0,\nabla^0)$ on the third copy.
Let $N$ be
equal, by definition, to the union of these 3 copies; $N$ is spin$^c$
bordant to the fourth copy, that is $M$, and this takes into account
the classifying maps into $B\pi$. Now apply the bordism theorem in the
spin$^c$ case, Theorem \ref{thm:spincbordism}, and get a metric $g''$
and a connection $\nabla''$ on the fourth component of $\partial X$,
which is by construction $M$, which together have gpsc. We set
\begin{equation}\label{group-structure-1}
[(g,\nabla)]+[(g',\nabla^{\prime})]:= [(g'',\nabla^{\prime\prime})]\,.
\end{equation}
This determines a well-defined sum on $\cP^c (M,L)$.
We are interested in giving lower bounds for the rank of this group
$\cP^c(M,L)$.

In the spin case there is a different but equivalent description of
this group structure; this pre-dates the Weinberger-Yu definition and
it is due to Stolz \cite{StolzConc}. For the sake
of clarity in our constructions, we briefly review the spin case (with
the obvious change in notations).

If $M$ is a spin manifold with $\dim M = n$ and
fundamental group $\pi$, then Stolz proves that there exists a free
and transitive action of $\rR^{\spin}_{n+1}
(M)\equiv \rR^{\spin}_{n+1} (B\pi)$ on $\cP (M)$, the set of
concordance classes of psc metrics on $M$.%
\footnote{Stolz employs the notation $\widetilde{\pi}_0 \cR^+ (M)$ for the set of concordance classes 
of psc metrics on $M$, since by the ``isotopy implies concordance''
lemma (e.g., \cite[Proposition 3.3]{MR1818778}), this is a quotient of ${\pi}_0 \cR^+ (M)$.}
This is achieved by defining a pairing
\begin{equation}\label{stolz-pairing}
\iota\co \cP (M)\times \cP (M)\to \rR^{\spin}_{n+1} (M)
\end{equation}
satisfying 
\begin{itemize}
  \item[(i)] $\iota([g],[g])=0$ and $\iota([g_0],[g_1])+ \iota([g_1],[g_2])= \iota([g_0],[g_2])$;
  \item[(ii)] for any fixed $[g_0]$ the map 
  \[ \iota_{[g_0]}\co \cP (M)\ni [g]\mapsto \iota([g_0],[g])\in \rR^{\spin}_{n+1} (M) \]
  is a bijection.
\end{itemize}
The action of $\rR^{\spin}_{n+1} (M)$ on $\cP (M)$ is then defined as
$x\cdot [g]:= \iota_{[g]}^{-1} (x)$, $x\in \rR^{\spin}_{n+1} (M)$,
$[g]\in\cP (M)$. Then the map $\iota([g_0],[g_1])$ is
given by the class
\begin{equation*}
[(M\times [0,1], {\rm pr}_1\co M\times [0,1]\to M,
g_0 \sqcup g_1)]
\in \rR^{\spin}_{n+1} (M).
\end{equation*}
We can then define a group structure on $\cP (M)$ by using the group
structure on ${\rm R}^{\spin}_{n+1} (M)$ and the free transitive action of
this group on $\cP (M)$. This group structure is equivalent to the one
defined by Weiberger and Yu; see Xie-Yu-Zeidler
\cite[Proposition 4.1]{MR4292958} for a proof of this equivalence. 

\begin{remark}\label{rem:actiononconc}
There is also still another way to understand the action of
$\rR^{\spin}_*(M)$
on the concordance classes $\cP (M)$ of psc metrics.
For simplicity we explain this just in the simply connected case with
$M=S^n$.\footnote{This is, in fact, the most general case for simply
connected spin manifolds $M^n$ with $n\ge 5$ 
and $\alpha(M)=0$, for then 
the space $\cR^+(M^n)$ is always homotopy
equivalent to $\cR^+(S^n)$; see \cite{MR4767495}.}
We choose a basepoint for $\cP(S^n)$ given by the concordance
class of the standard round metric.

Then a class in
$\rR^\spin_{n+1}(S^n)=\rR^\spin_{n+1}(\pt)$ is represented by an
$(n+1)$-dimensional spin manifold $W$ with boundary $M'$, together
with a psc metric $g'$ on $M'$.  Since $M'$ is null-bordant, $g'$
corresponds to a metric $g$ on $S^n$, which can be obtained by pushing
the metric $g$ through the bordism $W$ to the boundary of a small disk
$D^{n+1}$.  It's easy to see that this gives an identification of
$\rR^\spin_{n+1}(\pt)$ with $\cP(S^n)$.  See
Figure \ref{fig:concsphere}.  The resulting group structure on
$\cP(S^n)$ was also described slightly differently
in \cite[Proposition 3.1]{MR962295}.
\end{remark}
\begin{figure}[htb]
\begin{picture}(10,0)
\put(70,-10){{\small $M'^{n}$}}
\put(0,70){{\small $W^{n+1}$}}
\put(80,123){{\small $S^{n}$}}
\put(45,125){{\small $D^{n+1}$}}
\end{picture}
\includegraphics[height=1.7in]{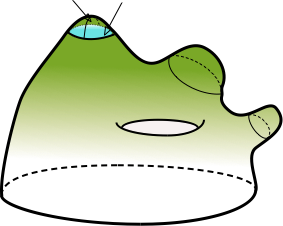}
\caption{The identification of $\rR_{n+1}^{\spin}(\pt)$
and $\cP(S^n)$}
\label{fig:concsphere}
\end{figure}

Now we want to carry the theory of the action of the
$\rR$-groups on concordance classes over to the spin$^c$ case.
Following Stolz's program, we would like to define a pairing
\begin{equation}\label{stolz-c-pairing}
  \iota\co \cP^c (M,L)\times \cP^c (M,L)\to \rR^{\spinc}_{n+1} (M)
\end{equation}
when $(M,\sigma,L)$ is spin$^c$ and totally non-spin and prove that this can
be done ensuring the analogues of (i) and (ii) above
(as listed right after equation \eqref{stolz-pairing}).  We follow
exactly the program of \cite[\S5]{StolzConc},
using \cite{MR192509} just as in Stolz's proof of
Proposition 5.8. 
\begin{proposition}[Cf.\ {\cite[Proposition 5.8]{StolzConc}}]
\label{prop:Stolz5.8}
Let $N$ be a connected manifold with boundary of dimension 
$n\ge 6$ which is simple homotopy equivalent to a connected
finite $2$-complex $K$ with fundamental group $\pi$,
and assume that $N$ is equipped with a spin$^c$
structure with spin$^c$ line bundle $L$
for which the associated map $c_N\co N\to B\gamma'$  
{\lp}in the notation of Remark \ref{rem:gamma}{\rp} 
is a $2$-equivalence. {\lp}This requires that $N$
be totally non-spin.{\rp}
Then every element of $\rR^{\spinc}_n(K)$ can be represented by
some $(N,g,L\otimes L', \nabla, h)$ with $g$ a psc metric on $\partial N$,
$h$ a simple homotopy equivalence $N\to K$,
$L'$ an even line bundle {\lp}one whose $c_1$ reduces
mod $2$ to $0${\rp} and $\nabla$ a connection on
$L\otimes L'$ giving $(\partial N, (L\otimes L')|_{\partial N})$
gpsc.
\end{proposition}
\begin{proof}
We follow Stolz's method.  Suppose a class in 
$\rR^{\spinc}_n(K)\cong \rR^{\spinc}_n(B\pi)$
is represented by some manifold $M^n$ with boundary and suitable
line bundle and a map to $B\pi$, 
with a metric and connection on $\partial M$
with gpsc.  
By Lemma \ref{lem:makingtotallynonspin}, we can
assume that $M$ is totally non-spin, and thus that the
classifying map $c_M\co M\to B\gamma'$ is a $2$-equivalence.
Since we have a homotopy equivalence $\alpha\co K\to N$ and
$c_N\co N\to B\gamma'$ is also a $2$-equivalence, there
is a map $j\co K\to M$ such that 
$c_N\circ\alpha\simeq c_M\circ j$. By the Embedding Theorem
in \cite[\S2]{MR192509}, since $\dim K = 2 < \dim M -3$
and the map $j\co K\to M$ is a $2$-equivalence, there is
an embedding $i\co N'\hookrightarrow M$ of
a compact codimension-$0$ submanifold $N'$ of $M$ with
$\pi_1(\partial N')\cong\pi_1(N')$ and a simple
homotopy equivalence $j'\co K\to N'$ with $j\simeq i\circ j'$.
Note that $\alpha\co K\to N$ and $j'\co K\to N'$ are both
thickenings of $K$ in the sense of \cite{MR192509}.
By \cite[Proposition 5.1]{MR192509}, thickenings in the
stable range are classified by their stable normal bundles,
which in this case are the same, and so 
$N'$ and $N$ are diffeomorphic.  Identifying $N$ with $N'$,
we get a decomposition of $M$ as $M=N\cup_{\partial N} C$
for some complement $C$, as shown in Figure \ref{fig:Stolz5-8}.
\begin{figure}[htb]
\vspace*{5mm}

\begin{picture}(10,30)
\put(120,-10){{\small $M$}}
\put(32,3){{\small $C$}}
\put(80,103){{\small $N$}}
\put(137,103){{\small $K$}}
\end{picture}
\includegraphics[height=1.4in]{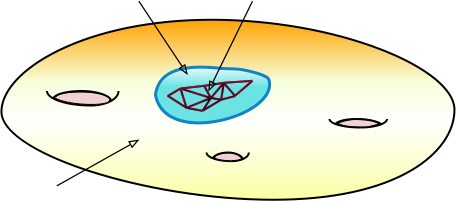}

\caption{Decomposition of $M=N\cup_{\partial N}C$,
where $N$ thickens $K$}
\label{fig:Stolz5-8}
\end{figure}
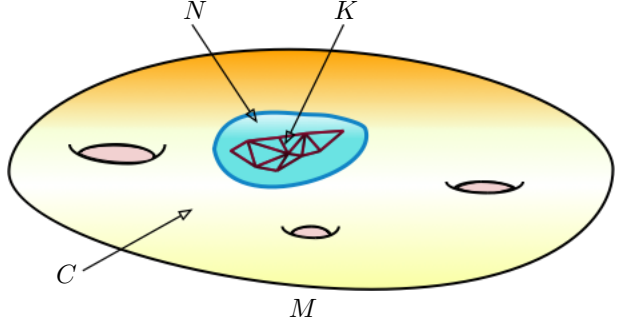
The inclusion of $\partial N$ into $C$ is a $2$-equivalence,
so we are in a position to apply the Extension Theorem,
Theorem \ref{th:extension-theorem}, to extend gpsc
from $\partial M$ across $C$ to $\partial N$.
Then the bordism coming from $C$ shows that our
original class in $\rR^{\spinc}_n(K)$ is equivalent
to the class represented by $N$ and the restriction of
the bundle, metric, and connection from $M$ to $N$.
Now this bundle on $N$ may not be the same as the one
original given to us in the hypothesis of the proposition,
but its Chern class must have the same reduction mod $2$, and hence
our class is represented by
some $(N,g,L\otimes L', \nabla)$ with $g$ a psc metric
on $\partial N$,
$L'$ an even line bundle {\lp}one whose $c_1$ reduces
mod $2$ to $0${\rp} and $\nabla$ a connection on
$L\otimes L'$ giving $(\partial N, (L\otimes L')|_{\partial N})$
gpsc.
\end{proof}
Now we are ready for the main classification theorem.  Unfortunately
there are no cases where we can compute the $\rR^{{\spin}^c}$ groups
exactly in high dimensions, but later we will at least give some
bounds on their size, and just the existence of a group acting simply
transitively on $\cP^c(M,L)$ already has some
interesting formal consequences.
\begin{theorem}[Classification Theorem]
\label{thm:classification}
Let $(M,\sigma,L)$ be a totally non-spin spin$^c$
$n$-manifold, $n\ge 5$, admitting a metric and connection with
generalized psc.  If $M$ has a boundary, fix a metric $h$ and
connection $\nabla_{\partial M}$ on $\partial M$ so that $(\partial M,
L|_{\partial M},h,\nabla_{\partial M})$ has gpsc; then all metrics and
connections considered on $M$ are to be of product type near the
boundary restricting to $(h,\nabla_{\partial M})$ on the boundary.  In
this situation, $\rR^{{\spin}^c}_{n+1}(M)$ acts simply transitively on
$\cP^c(M,L)$.
\end{theorem}
\begin{proof}
As with \eqref{stolz-pairing}, we define
a pairing
\begin{equation}\label{stolz-pairing1}
\iota\co \cP^c (M,L)\times \cP^c (M,L)\to \rR^{\spinc}_{n+1} (M)
\end{equation}
satisfying 
\begin{itemize}
  \item[(i)] $\iota([g,\nabla],[g,\nabla])=0$ and 
  $\iota([g_0,\nabla_0],[g_1,\nabla_1])+ \iota([g_1,\nabla_1],[g_2,\nabla_2])= \iota([g_0,\nabla_0],[g_2,\nabla_2])$;
  \item[(ii)] for any fixed $[g_0,\nabla_0]$ the map 
  \[
  \iota_{[g_0,\nabla_0]}\co \cP^c (M,L)\ni [g,\nabla]\mapsto 
  \iota([g_0,\nabla_0],[g,\nabla])
  \in \rR^{\spinc}_{n+1} (M)
  \]
  is a bijection.
\end{itemize}
The map $\iota$ associates to $[g_0,\nabla_0]$ and 
$[g_1,\nabla_0]$ the class 
\begin{equation*}
[(M\times [0,1], {\rm pr}_1\co M\times [0,1]\to M, 
(g_0,\nabla_0) \sqcup (g_1,\nabla_1))]
\in \rR^{\spin}_{n+1} (M).
\end{equation*}
Then the action of 
$\rR^{\spinc}_{n+1} (M)$ on $\cP^c(M,L)$ is given by
$x\cdot (g,\nabla) = \iota_{[g,\nabla]}^{-1}(x)$.

The definition of $\iota$ and property (i) are rather easy.
Given $[g_0,\nabla_0]$ and $[g_1,\nabla_1]$ with gpsc on $M$,
we consider $M\times I$ with the metric $h\times dt^2$
and connection pulled back from $\nabla_{\partial M}$
on $(\partial M)\times I$ and $(g_0,\nabla_0)$ and 
$(g_1,\nabla_1)$ on $M\times \{0\}$ and $M\times \{1\}$,
respectively.  This defines a class in $\rR^{\spinc}_{n+1}(M)$.
That $\iota([g,\nabla],[g,\nabla])=0$ follows from (2)
of Theorem \ref{th:extension-theorem}, and the identity
\begin{equation*}
\iota([g_0,\nabla_0],[g_1,\nabla_1])+ \iota([g_1,\nabla_1],[g_2,\nabla_2])= \iota([g_0,\nabla_0],[g_2,\nabla_2])
\end{equation*}
follows from gluing two copies of $M\times I$ together.
Injectivity of $\iota_{[g,\nabla]}$ follows from the
following argument.  Suppose $\iota([g,\nabla],[g_1,\nabla_1])
= \iota([g,\nabla],[g_2,\nabla_2])$.  Then
\[
\begin{aligned}
\iota([g_1,\nabla_1],[g_2,\nabla_2]) &=
\iota([g_1,\nabla_1],[g,\nabla])+\iota([g,\nabla],[g_2,\nabla_2]) \\
&= - \iota([g,\nabla],[g_1,\nabla_1])
+\iota([g,\nabla],[g_2,\nabla_2]) =0
\end{aligned}
\]
by (i), and so $[g_1,\nabla_1]$ and $[g_2,\nabla_2]$ are
concordant by the Extension Theorem,
Theorem \ref{th:extension-theorem}.

The hard part is the proof of surjectivity, and this is where we use
Proposition \ref{prop:Stolz5.8}.  Let $N$ be a thickening (in the
sense of \cite{MR192509}) of the $2$-skeleton of $M$ inside $M\times
I$. Note that since $M$ is totally non-spin of dimension $\ge 5$, $M$
contains an embedded $2$-sphere on which $w_2$ is non-zero, so that
the normal bundle of this $2$-sphere is \emph{not} trivial.  This
$2$-sphere can be chosen to lie in $N$, so $N$ is totally non-spin.
By Proposition \ref{prop:Stolz5.8}, any class in
$\rR^{{\spin}^c}_{n+1}(M)$ can be represented by a suitable
spin$^c$ structure on $N$ and a metric and
connection with gpsc on the boundary $\partial N$.
Decompose $M\times I$ as $N\cup_{\partial N}C$, where $C$
is the complement of $N$ in $M\times I$.  (See
Figure \ref{fig:Stolz5-4}.)

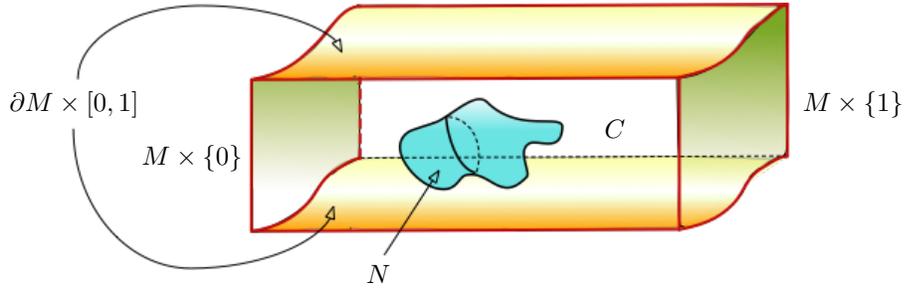
\begin{figure}[htb]


\begin{picture}(10,30)
\put(125,-5){{\small $N$}}
\put(215,50){{\small $C$}}
\put(-10,60){{\small $\p M \times [0,1]$}}
\put(40,40){{\small $M\times \{0\}$}}
\put(290,60){{\small $M\times \{1\}$}}
\end{picture}
\includegraphics[height=1.4in]{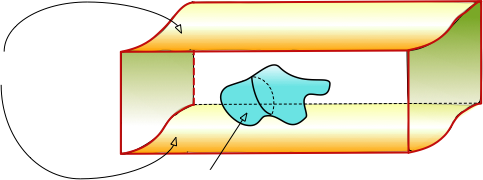}
\caption{Decomposition of $M\times I=N\cup_{\partial N}C$}
\label{fig:Stolz5-4}
\end{figure}
%

We take the gpsc data given by
$[g_0,\nabla_0]$ on $M\times \{0\}$, by
$(h\times dt^2,\nabla_{\partial M})$ on 
$\partial M\times I$, and by $(g_{\partial N},\nabla_{\partial N})$
on $\partial N$. Assuming $M\times \{1\}\hookrightarrow C$
is a $2$-equivalence, we can apply the Extension Theorem,
Theorem \ref{th:extension-theorem}, to extend gpsc
from $\partial M$ across $C$ to  obtain
$[g_1,\nabla_1]$ on $M\times \{1\}$.
Then $(N,g_{\partial N},\nabla_{\partial N})$
is equivalent in $\rR^{{\spin}^c}_{n+1}(M)$  to 
$\iota([g_0,\nabla_0],[g_1,\nabla_1])$, as required.
\end{proof}
\begin{proposition}
\label{prop:action}
Let $M$ be a totally non-spin spin$^c$ n-manifold, $n\geq 5$. Let us
fix $(g_0,\nabla^0)\in \cP^c (M,L)$. Then the group structure on
$\cP^c (M,L)$ given by \eqref{group-structure-1} and the group
structure defined by the free and transitive action of
$\rR^{\spinc}_{n+1} (M)$ on $\cP^c (M,L)$ agree.
\end{proposition}
\begin{proof}
The proof given by Xie-Yu-Zeidler \cite[Proposition 4.1]{MR4292958}
carries over with essentially no change.
\end{proof}

\section{The spin$^c$ Stolz sequence and the analytic surgery sequence}
\label{sec:mapping}
Following the spin case, we shall
explain in this section and in Appendix \ref{appendix1} how to map the
Stolz sequence in the spin$^c$ case to the Higson-Roe analytic surgery
sequence. The
spin case was first established in \cite{MR3286895} for $n$
odd. Alternative treatments, valid for any $n$, were subsequently
given by Xie-Yu \cite{xie-yu-advances}, Zeidler \cite{zeidler-j-top}
and Zenobi \cite{zenobi-advances} \cite{zenobi-jncg}.  The original
approach in \cite{MR3286895} was extended to the case $n$ even by
Zenobi in \cite{zenobi-jta}.  These results and the subsequent
treatments, in the context of spin manifolds with
positive scalar curvature, can be viewed as analogues of the original
seminal results of Higson and Roe regarding the mapping of the surgery
sequence in topology to their analytic surgery
sequence \cite{higson-roeI}, \cite{higson-roeII}, \cite{higson-roeIII}.

\subsection{Mapping the spin$^c$ Stolz sequence to the
analytic surgery sequence} In this subsection we only state the
relevant results. We refer
to Appendix \ref{appendix1} for the necessary definitions of the
groups and the maps entering in the diagrams of
Theorem \ref{theo:commute-r-pos} below.  The
Appendix also contains a comment about the proof,
noting its similarity to the one given for the spin case.

\begin{theorem}\label{theo:commute-r-pos}
For a compact space $X$ with fundamental group $\Gamma$ and universal covering  
$\tX$, there exists a well-defined and commutative diagram
with exact rows:
\begin{equation}\label{eq:StolzToAnaX}
\begin{tikzcd}
\cdots \to \Omega^{\spinc}_{n+1} (X) \ar[r] \ar[d, "\beta^c"]&
R^{\spinc}_{n+1}(X) \ar[r] \ar[d, "\Ind^c_\Gamma"]
&
\Pos^{\spinc}_n (X) \ar[r] \ar[d, "\rho^c_\Gamma"] &
\Omega^{\spinc}_n (X) \ar[d, "\beta^c"] \to \cdots
\\ 
\cdots \to K_{n+1} (X) \ar[r] & K_{n+1} ( C^*(\tX)^\Gamma) \ar[r] & 
K_{n+1}(D^*(\tX)^\Gamma) \ar[r] & K_{n} (X)  \to \cdots
\end{tikzcd}
\end{equation}
We also have a universal commutative diagram with exact rows:
\begin{equation}\label{eq:StolzToAna}
\begin{tikzcd}
\cdots \to \Omega^{\spinc}_{n+1} (B\Gamma) \ar[r] \ar[d, "\beta^c"]&
R^{\spinc}_{n+1}(B\Gamma) \ar[r] \ar[d, "\Ind^c_\Gamma"]
&
\Pos^{\spinc}_n (B\Gamma) \ar[r] \ar[d, "\rho^c_\Gamma"] &
\Omega^{\spinc}_n (B\Gamma) \ar[d, "\beta^c"] \to \cdots
\\ 
\cdots \to K_{n+1} (B\Gamma) \ar[r] & K_{n+1} ( C^* (E\Gamma)^\Gamma) \ar[r] & 
K_{n+1}(D^* (E\Gamma)^\Gamma) \ar[r] & K_{n} (B\Gamma)  \to \cdots
\end{tikzcd}
\end{equation}
\end{theorem}
\begin{remark}
It is well known that an oriented hypersurface $N$ of a spin$^c$ manifold $M$ 
inherits a spin$^c$ structure. With little effort one could also develop a 
{\it secondary partitioned  index theorem in the spin$^c$ case}, following the spin case
developed in \cite{MR3286895}.  This result, together with its
enhancements due to Zeidler \cite{zeidler-j-top}, might have some
nice geometric applications; see \cite{zeidler-j-top} for the spin
case.  Similarly, one could develop product formulas in the spin$^c$
context as in Zeidler \cite{zeidler-j-top} and
Zenobi \cite{zenobi-jta} and give, correspondingly, geometric
applications.
\end{remark}

\subsection{Alternative treatments}\label{subsect:alternative}
As already remarked, there are alternative treatments of the
Higson-Roe analytic surgery sequence and of its
mapping to the Stolz sequence. Here is a brief description of these:

\medskip
\noindent
\textbf{The approach through localization algebras.}\\
Developed by Xie and Yu in \cite{xie-yu-advances}, this approach  employs Yu's localization algebras and can be developed without further complications for the associated  real $C^*$-algebras. 
The  real version of the Higson-Roe surgery sequence can be written, after applying natural isomorphisms
to the Xie-Yu sequence,  as
\begin{equation}\label{real-HR}
\cdots\rightarrow KO_{n+1}(B\Gamma)\rightarrow KO_{n+1}(C^*_{r,\bR} \Gamma)\rightarrow \SG^\Gamma_{*,\bR} (E\Gamma) \rightarrow KO_{n}(B\Gamma)\rightarrow \cdots
\end{equation}
The original Stolz' surgery sequence can thus be mapped to this analytic sequence, see \cite{xie-yu-advances}.

\medskip
\noindent
{\bf The approach through Trout's treatment of K-theory.}\\ There is a
slightly different approach developed by Zeidler
in \cite{zeidler-j-top}; this employs Trout's approach to K-theory and
in principle contains both the real and the complex case.%
\footnote{Unfortunately, while the compatibility of this approach with 
the Higson-Roe sequence is fully discussed in the complex case 
(see the appendix in \cite{zeidler-j-top}), 
this has not been done in the real case.}

\medskip
\noindent
\textbf{The approach through the adiabatic groupoid.}\\  There is an
approach through the adiabatic groupoid, developed by Zenobi
in \cite{zenobi-advances}.  The compatibility with Higson-Roe is
discussed in great detail in \cite{zenobi-jncg}. This approach allows
for interesting generalizations of the Higson-Roe surgery sequence,
beyond the case of Galois coverings.  See \cite{zenobi-advances} and
also \cite{zenobi-jfa}.

\medskip
\noindent
{\bf The approach through pseudodifferential operators.}\\ A variant
of Zenobi's approach, particularly useful in discussing numeric
invariants associated to the rho class, is discussed in
\cite{PSZ}. It is valid when $X$ is a smooth compact
manifold with fundamental group $\Gamma$ and universal cover
$\widetilde{X}$. We give the necessary details in
Appendix \ref{appendix1} but we anticipate that this approach allows
for the mapping of the Higson-Roe sequence to a sequence in
noncommutative de Rham homology through appropriate Chern character
maps. This is used in order to define, under suitable assumptions on
the group $\Gamma$, for example if $\Gamma$ is Gromov hyperbolic, a
pairing
\begin{equation}\label{S-pairing-no-appendix}
\SG^\Gamma_*(\tX)\times HP^{*-1}(\mathbb{C}\Gamma;\langle x \rangle)\to \mathbb{C}
\end{equation}
Applied to the rho class this pairing defines {\it higher rho
numbers}, parametrized by elements in the groups
$HP^{*-1}(\bC\Gamma;\langle x \rangle)$.

\subsection{Spin versus spin$^c$}
We now make a few crucial remarks on the relationship between the spin case and the spin$^c$ case and discuss how
results for the former imply results for the latter.
Let $M$ be a spin manifold with a fixed spin structure; it is well known that $M$
has then a natural spin$^c$ structure, see \cite{lawson89:_spin}, with 
trivial
spin$^c$ line bundle; thus there is a natural group homomorphism
\begin{equation}\label{from-spin-to-spinc-pos}
j\co {\rm Pos}^{\spin}_* (B\Gamma)\rightarrow {\rm Pos}^{\spinc}_* (B\Gamma)
\end{equation}
\begin{equation}\label{from-spin-to-spinc-pos-map}
{\rm Pos}^{\spin}_* (B\Gamma)\ni[(M,\varphi\co M\to
B\Gamma, g)]
\overset{j}{\longmapsto} [(M,\varphi\co M\to B\Gamma, g,\nabla^{{\bf 1}})]\in {\rm Pos}^{\spinc}_* (B\Gamma)
\end{equation}
with ${\bf 1}:=M\times \bC$ the trivial complex line 
bundle and $\nabla^{{\bf 1}}$ the trivial connection on it.
Similarly, we have a natural group homomorphism 
\begin{equation}\label{from-spin-to-spinc-R}
\kappa:  {\rm R}^{\spin}_* (B\Gamma)\rightarrow  {\rm R}^{\spinc}_* (B\Gamma),
\end{equation}
\begin{equation}\label{from-spin-to-spinc-pos-map1}
 {\rm R}^{\spin}_* (B\Gamma)\ni[(M,\varphi: M\to B\Gamma, g_{\partial M})]
\xrightarrow{\kappa} [(M,\varphi: M\to B\Gamma, g_{\partial M},\nabla^{{\bf 1}}_\partial)]\in  {\rm R}^{\spinc}_* (B\Gamma).
\end{equation}
The following two diagrams are clearly commutative:

\begin{equation}\label{eq:1a}
\xymatrix{& {\rm Pos}^{\spinc}_n (B\Gamma)\ar[d]^{\rho^c_\Gamma}\\
{\rm Pos}^{\spin}_n (B\Gamma) \ar[ur]^j \ar[r]^{\;\;\;\;\;\rho_{\Gamma}} & \,\SG^\Gamma_* (E\Gamma),}\quad\quad
\xymatrix{& {\rm R}^{\spinc}_n (B\Gamma)\ar[d]^{\Ind^c_\Gamma}\\
{\rm R}^{\spin}_n (B\Gamma) \ar[ur]^\kappa \ar[r]^{\;\;\;\;\;\Ind_\Gamma} & \,K_* (C^*_r \Gamma).}
\end{equation}
%
%
%
%

\smallskip
\noindent
We can sharpen the commutativity of \eqref{eq:1a} as follows. 

\begin{proposition}\label{prop:spinKO-versus-spinc}
The following  diagrams are commutative

\begin{equation}\label{eq:1}
  \begin{diagram}
    \setlength{\dgARROWLENGTH}{1.95em}
  \node{{\rm Pos}^{\spin}_* (B\Gamma)}
           \arrow{e,t}{j}
           \arrow{s,l}{\rho_\Gamma}
  \node{{\rm Pos}^{\spinc}_* (B\Gamma)}
             \arrow{s,l}{\rho^c_\Gamma}
  \\
  \node{\SG^\Gamma_{*,\bR} (E\Gamma)}
  \arrow{e,t}{{\rm c}}
  \node{\SG^\Gamma_* (E\Gamma),}
  \end{diagram}
  \quad\quad
%
  \begin{diagram}
    \setlength{\dgARROWLENGTH}{1.95em}
  \node{ {\rm R}^{\spin}_* (B\Gamma)}
           \arrow{e,t}{\kappa}
           \arrow{s,l}{\Ind_\Gamma}
  \node{ {\rm R}^{\spinc}_* (B\Gamma)}
             \arrow{s,l}{\Ind^c_\Gamma}
  \\
  \node{KO_* (C^*_{r,\bR} \Gamma)}
  \arrow{e,t}{c}
  \node{K_* (C^*_r \Gamma),}
  \end{diagram}
\end{equation}
with $c$ in the bottom horizontal rows induced by complexification.
\end{proposition}
\begin{proof}
Consider for example the diagram on the right.  Regarding the left
vertical map we follow here the treatment in \cite[Remark
3.14]{xie-yu-advances}.  The map $\Ind_\Gamma$ employs the
$C\ell_n$-linear Atiyah-Singer operator $\Di$. Briefly, given a spin
manifold $(M,g)$ we consider $C\ell_n$, the real Clifford algebra
associated to the euclidean space $\mathbb{R}^n$, and $\ell\co {\rm
Spin}_n\to {\rm Hom} (C\ell_n,C\ell_n)$, the representation given by
left multiplication. We then consider the bundle of rank one
$C\ell_n$-modules given by $\Spi (M):= P_{{\rm Spin}} (M)\times_{\ell}
C\ell_n$. The bundle $\Spi (M)$, often denoted simply by $\Spi$,
inherits a Levi-Civita connection $\nabla$ and we can thus consider
$\Di$, the associated Dirac operator: $\Di := {\rm cl}\circ \nabla$.
It is a (right) $C\ell_n$-linear operator and
depends on $g$ and the chosen spin structure. See \cite[Chapter II,
Section 7]{lawson89:_spin} for more details.  There is,
correspondingly, a $\Gamma$-invariant operator $\Di_\Gamma$ on any
$\Gamma$-Galois covering of $M$.  The left vertical map associates to
a class $[(M,\phi:M\to B\Gamma,g_{\partial M})]\in
{\rm R}^{\spin}_* (B\Gamma)$ the index class associated to the
operator $\Di_\Gamma$ on the manifold with cylindrical ends,
$\widetilde{M}_\infty$, associated to the Galois covering defined by
$\phi$. The complexification of this class is equal to the class
defined by the complexification of the bundle $\widetilde{\Spi}$ over
$\widetilde{M}_\infty$ and of the Atiyah-Singer operator $\Di_\Gamma$
acting on the sections of this bundle. As explained in great detail
in \cite[Appendix 2]{zeidler-j-top} these complexifications can be
identified respectively with the $L^2$-sections of the complex spinor
bundle on $\widetilde{M}_\infty$ and the complex spin-Dirac operator
$D_\Gamma$ on it. Since this is precisely what is employed in the
definition of the right vertical map, the commutativity of the right
diagram follows immediately. The commutativity of the left diagram is
proved in a similar way.
\end{proof}
We use this Proposition as follows. Following
\cite{BC-fete} and \cite{MR3830202}, given a discrete group $\Gamma$,
we define the vector space
$F\Gamma:=\{f\colon \Gamma_{fin}\to\mathbb{C}\;\;\text{such
that}\;\;\vert{\rm supp(f)}\vert<\infty\}$, where $\Gamma_{fin}$ is
the set of elements of finite order in
$\Gamma$.  Then $F\Gamma$ becomes a
$\Gamma$-module under conjugation; moreover, $F\Gamma$ decomposes as
the direct sum $F\Gamma=F^0\Gamma\oplus F^1\Gamma$ with
$F^q\Gamma=\{f\in F\Gamma\mid f(\gamma)=(-1)^q f(\gamma^{-1})\}$, the
$\pm 1$-eigenspaces for the involution $\tau\colon F\Gamma\to F\Gamma$
induced by sending $\gamma$ to $\gamma^{-1}$. Finally, we consider 
the submodule
$F^q_0 (\Gamma)=\{f\in F^q (\Gamma)\;\;\text{such that}\;\;f(1)=0\}$. This was also denoted
as $F^q_{\rm del} (\Gamma)$ in \cite{PSZ} and we shall adopt this 
notation later in the article when we shall want to quote results from \cite{PSZ}.

Now, in the diagram \eqref{eq:1}, since the bottom maps are
rationally injective, commutativity implies that we can obtain lower
bounds on the ranks of ${\rm R}^{\spinc}_* (B\Gamma)$ and ${\rm
Pos}^{\spinc}_* (B\Gamma)$ by using lower bounds on the images of
$\Ind_\Gamma$ and $\rho_\Gamma$ on the left-hand side of the diagrams
(i.e., in the spin case).
Lower bounds in the spin case have
been obtained in the work of Xie, Yu and Zeidler \cite{MR4292958}, and
using their results and Proposition
\ref{prop:spinKO-versus-spinc}, we obtain in particular the following:

\begin{proposition}\label{prop:from-XYZ}
Let $p\in\{0,1,2,3\}$ and let $k\geq 0$. Assume that
$\Gamma$ satisfies the Novikov conjecture,
i.e., the Baum-Connes 
map $\mu\co K_*^\Gamma (\underline{E}\Gamma)\to K_* (C^*_r \Gamma)$
is rationally injective. Then the index map 
$\Ind^c_\Gamma\co {\rm R}^{\spinc}_{4k+p} (B\Gamma) \to K_* (C^*_{r} \Gamma)$ 
surjects onto 
a subgroup of rank  greater or equal to the rank of 
$\bigoplus_{l<k} \bigoplus_{q\in\{0,1\}} H_{4l-2q+p}(\Gamma,F^q \Gamma)$. 
More precisely, the image of the map
$\Ind^c_\Gamma\otimes\mathbb{C}$ contains the image under $c$ 
{\lp}the injective complexification map{\rp} of 
\[
\mu (({\rm ph}_\Gamma)^{-1} (\bigoplus_{l<k} \bigoplus_{q\in\{0,1\}} H_{4l-2q+p}(\Gamma,F^q \Gamma))
\]
with ${\rm ph}_\Gamma$ denoting the equivariant Pontryagin character, 
\[
{\rm ph}_\Gamma\co   KO^\Gamma_*(\underline{E}\Gamma)\otimes\mathbb{C} \xrightarrow{\simeq}
  \bigoplus_{k\in\mathbb{Z}} H_{*+4k}(\Gamma; F^0\Gamma) \oplus
  H_{*+2+4k}(\Gamma;F^1\Gamma).
  \]
Under the same assumptions
the rank of ${\rm Pos}^{\spinc}_{4k+p-1} (B\Gamma)$ is greater or equal to the rank of\\
$\bigoplus_{l<k} \bigoplus_{q\in\{0,1\}} H_{4l-2q+p}(\Gamma,F^q_0 \Gamma)$.
\end{proposition}

\noindent
These results have of course an  intrinsic interest; moreover they will be very useful in Section \ref{sec:psc}, when we shall give results
on moduli spaces.

\section{The rho map for finite groups}
\label{sec:genrho}
We introduce a short notation for the structure group $\SG^\Gamma_*
(E\Gamma)$. We set $\cS_*(\Gamma):= \SG^\Gamma_* (E\Gamma)\,$.
%
In this section we discuss some computations of the map
\begin{equation}
\label{eq:rho}
\rho_\Gamma\co \Pos^{\spinc}_d(B\Gamma)\to \cS_d(\Gamma)
\end{equation}
for certain finite groups $\Gamma$.  (It
is natural to begin with finite groups because, as shown
in \cite[equation (1.5)]{MR3286895}, the map $\rho_\Gamma$
automatically vanishes when the fundamental group $\Gamma$ is
torsion-free and satisfies the Baum-Connes Conjecture.)
More precisely, we compute
the structure group $\cS_d(\Gamma)$ in a few cases and give results on
the image of $\rho_\Gamma$, as we currently have no methods for computing
$R_d^{\spinc}(B\Gamma)$ and $\Pos^{\spinc}_d(B\Gamma)$ exactly, only for
giving lower bounds on their size.

\subsection{The case of $\bZ/p$, $p$ a prime}
We begin with the very simplest example of a nontrivial finite group,
namely $\Gamma=\bZ/p$ with $p$ a prime.  This group is abelian and is
(non-canonically) self-dual, so its group $C^*$-algebra or complex
group ring is (additively) just isomorphic to a direct sum of $p$
copies of $\bC$ indexed by
$\widehat\Gamma$, the character group of $\Gamma$.  
One of these copies
corresponds to the trivial representation.  The $K$-homology
$K_*(B\Gamma)$ is a direct sum of one copy of
$K_*(\pt)$ and of a torsion group concentrated in odd degrees (this
immediately follows from the Atiyah-Hirzebruch spectral sequence since
$\widetilde H_*(\Gamma,\bZ)$ is $\bZ/p$ in odd degrees and $0$ in even
degrees).  The assembly map $K_*(B\Gamma)\to K_*(\bC\Gamma)$ is an
isomorphism $K_*(\pt)\to K_*(\pt)$ onto the summand corresponding to
the trivial representation and must be $0$ on the torsion. In fact we
can say more.  The arguments that follow are similar to, but a little
easier than, the ones in
\cite[\S2]{MR1133900} (since the case of $KO$ is harder than that
of complex $K$-theory).
\begin{proposition}
\label{prop:strsetZp}  
If $\Gamma=\bZ/p$ with $p$ a prime, then 
$\widetilde K_0(B\Gamma)=0$ and
$K_1(B\Gamma)\cong \mu_{p^\infty}^{p-1}$, where
$\mu_{p^\infty}$ is the union of the groups of $p^r$-th
roots of unity, as $r\to\infty$.  The analytic structure group
$\cS_d(\Gamma)$ for $\Gamma$ is $0$ for $d$ even and isomorphic
to a direct sum of $p-1$ copies of $\bZ[\frac{1}{p}]$
for $d$ odd.
\end{proposition}
\begin{proof}
We have by the universal coefficient theorem for
$K$-theory \cite{AndUCT,MR0388375}
that \begin{equation*} \Hom(K_*(X;\bF_p),\bF_p)\cong
K^*(X;\bF_p) \end{equation*} (for any finite CW complex $X$, and
hence also for $B\Gamma$ since it is a filtered colimit of such spaces).  By
the Atiyah-Segal Theorem \cite{MR0148722,MR0259946},
$K^*(B\Gamma;\bF_p)\cong \widehat{R(\Gamma)\otimes \bF_p}$, where
$R(\Gamma)\cong \bZ[\rho]/(\rho^p-1)$ is the representation ring of
$\Gamma$ (here $\rho$ is the character sending the standard
generator of $\Gamma$ to $e^{2\pi i/p}$) and
the ``hat'' denotes completion with
respect to the $I(\Gamma)$-adic topology, $I(\Gamma)$ the
augmentation ideal (the kernel of the map $R(\Gamma)\to\bZ$ sending
$\rho$ to $1$).  Now \begin{equation*}
R(\Gamma)\otimes \bF_p\cong \bF_p[\rho]/(\rho^p-1)\cong \bF_p[\alpha]/((\alpha+1)^p-1)\cong \bF_p[\alpha]/(\alpha^p), \end{equation*}
where $\alpha = \rho-1$ generates $I(\Gamma)$.  This is already
$I(\Gamma)$-adically complete, since $\alpha^p=0$ in
$R(\Gamma)\otimes \bF_p$.  Hence
$K^*(B\Gamma;\bF_p)\cong \bF_p[\alpha]/(\alpha^p)$, which has rank
$p$ over $\bF_p$. (This includes one copy of $K^*(\pt;\bF_p)$
pulled back via the map $B\Gamma\to\pt$.)  In
fact, we see that $I(\Gamma)$
is generated multiplicatively over $\bZ$ by $\alpha$, with relation
$(\alpha+1)^p-1=0$ or $p\alpha + \binom{p}{2}\alpha^2 + \cdots +
p\alpha^{p-1} = -\alpha^p$, so as pointed out in \cite[proof of
Proposition 8.1]{MR0148722}, $I(\Gamma)^k/I(\Gamma)^{k+1}$ is cyclic
of order $p$ generated by the class of $\alpha^k$, and one can see
from this that $\widetilde{K^0(B\Gamma)}$ is additively isomorphic
to the direct sum of $p-1$ copies (one for each of
$\alpha,\alpha^2,\cdots,\alpha^{p-1}$) of $\widehat\bZ_p$, the
$p$-adic integers.  Dualizing via the universal coefficient
theorem \cite{AndUCT,MR0388375} (here one should work at the
filtration level and note the degree shift due to the torsion), we
see that $K_1(B\Gamma)$ is a direct sum of $p-1$ copies of
$G=\varinjlim C_{p^r}=\mu_{p^\infty}$. Note that this is consistent
with our calculation that $\widetilde K_*(B\Gamma;\bF_p)$ has rank
$p-1$ over $\bF_p$.  As we already remarked, $\widetilde
K_0(B\Gamma)=0$ since $B\Gamma$ has no reduced homology in even
degrees.

Now we use this information to compute the structure group
$\cS_d(\Gamma)$.  After canceling the contribution of the basepoint,
$\widetilde K_*(B\Gamma)$ is concentrated in odd degrees,
$\widetilde K_*(\bC\Gamma)$ is free abelian and concentrated in even
degrees, and the analytic surgery sequence  becomes
\begin{equation}
  \label{eq:surgseqcyclic}
    0\to \widetilde K_{\text{even}} (\bC\Gamma)
    \cong \bZ^{p-1} \xrightarrow{\partial}
    \cS_{\text{odd}}(\Gamma) \to  
    \mu_{p^\infty}^{p-1}\cong K_{\text{odd}} (B\Gamma)\to 0.
\end{equation}
We claim that \eqref{eq:surgseqcyclic} actually splits into
$p-1$ non-split short exact sequences of the form
\[
  0\to \bZ \to \bZ\bigl[\tfrac{1}{p}\bigr] \to
  \bZ\bigl[\tfrac{1}{p}\bigr]/\bZ \to 0.
\]
In particular, $\cS_{\text{odd}}(\Gamma)$ is torsion-free. Here we
can apply \cite{higson-roe4}.  Indeed, \cite{higson-roe4}
gives for each $k$, $1\le k\le p-1$, a map $\Tr_{\rho^k,1}$ fitting into
a commutative diagram with exact rows
\[
  \xymatrix{
    0\ar[r]& \widetilde{K_{\text{even}} (\bC\Gamma)}\ar[r]^\partial
    \ar[d]_{\Tr_{\rho^k}-\Tr_1}&  \cS_{\text{odd}}(\Gamma)\ar[r]\ar[d]_{\Tr_{\rho^k,1}}
    &   K_{\text{odd}} (B\Gamma)\ar[r]\ar[d]_{\ind_{\rho^k,1}}& 0\\
     0 \ar[r]& \bZ\ar[r]&  \bR\ar[r]&  \bR/\bZ\ar[r]& 0}
\]
with the first downward arrow surjective and the last downward
arrow surjective onto $G=\bZ[\frac1p]/\bZ$.  Thus the middle downward
arrow surjects onto $\bZ[\frac1p]$.  The copies of the short exact
sequence at the bottom are easily seen to be distinct for the
$p-1$ values of $k$, and the result follows.
\end{proof}
\subsection{The case of $\bZ/p^r$}
The case of a general cyclic $p$-group $\Gamma=\bZ/p^r$
is similar. Here is the result.
\begin{proposition}
\label{prop:strsetZpr}  
If $\Gamma=\bZ/p^r$ with $r\ge1$, then 
$\widetilde K_0(B\Gamma)=0$ and
$K_1(B\Gamma)\cong \mu_{p^\infty}^{p^r-1}$.  
The analytic structure group
$\cS_d(\Gamma)$ for $\Gamma$ is $0$ for $d$ even and isomorphic
to a direct sum of $p^r-1$ copies of $\bZ\left[\frac{1}{p}\right]$
for $d$ odd.
\end{proposition}
\begin{proof}
This is essentially the same as the proof of Proposition \ref{prop:strsetZp}.
\end{proof}

\subsection{Elementary abelian $p$-groups}
As a first look at what happens for more complicated finite groups,
consider an elementary abelian $p$-group
$\Gamma=(\bZ/p)^r$, for $p$ a prime.  Using what we
know about the cyclic case, we have 
\[
R(\Gamma)\otimes \bF_p\cong 
\bF_p[\alpha_1,\cdots ,\alpha_r]/(\alpha_1^p,\cdots,\alpha_r^p).
\]
This is $I(\Gamma)$-adically complete and is thus isomorphic to 
$K^*(B\Gamma;\bF_p)$.  Arguing as in the cyclic case, we find that
$\widetilde{K^0(B\Gamma)}$ is additively isomorphic to the direct
sum of $p^r-1$ copies of $\widehat\bZ_p$
(recall that $K^1(B\Gamma)=0$ for any finite group
\cite[Corollary 7.3]{MR0148722}), and $K_1(B\Gamma)$ is a direct
sum of $p^r-1$ copies of $\mu_{p^\infty}$.  The conclusion of this analysis 
is the following:
\begin{proposition}
\label{prop:strsetZp2}  
If $\Gamma=(\bZ/p)^r$ with $p$ a prime, then 
$\widetilde K_0(B\Gamma)=0$ and
$K_1(B\Gamma)\cong \mu_{p^\infty}^{p^r-1}$.  The analytic 
structure group
$\cS_d(\Gamma)$ for $\Gamma$ is $0$ for $d$ even and isomorphic
to a direct sum of $p^r-1$ copies of $\bZ\left[\frac{1}{p}\right]$
for $d$ odd.
\end{proposition}
Note incidentally that $\cS_d(\Gamma)$ is the same
for $\Gamma=\bZ/p^r$ and for $\Gamma=(\bZ/p)^r$.
Now we are ready for the main result of this section.
\begin{theorem}
\label{thm:rangeofrho}   
Let $\Gamma=(\bZ/p)^r$ be an elementary abelian $p$-group.
Then the $\rho$-map \textup{\eqref{eq:rho}} has image
in $\cS_d(\Gamma)\cong \left(\bZ\left[\frac{1}{p}\right]\right)^{p^r-1}$
which is of full rank $p^r-1$, provided $d\ge 5$ is odd, and 
provided $p$ is odd if $d\equiv1\pmod4$.
\end{theorem}
\begin{proof}
Since $\cS_d(\Gamma)$ is torsion-free and given by a direct sum of groups
indexed by the non-trivial characters of $\Gamma$, and each such character
factors through a cyclic quotient of $\Gamma$ isomorphic to
$\bZ/p$, it suffices to consider the cyclic case $r=1$ and the summand
associated to the identity character sending the generator $1$ of
$\Gamma$ to $e^{2\pi i/p}$. When $d\equiv3 \pmod{4}$, the result is
just \cite[Example 2.13]{MR2366359}.  When $d\equiv1 \pmod{4}$, 
we need to exclude $p=2$ since in that case $\Gamma$ has only one
non-trivial character, and it's invariant under $g\mapsto g^{-1}$,
so that by \cite{MR1339924}, the rho-invariant vanishes identically.
However, for $p$ odd, $\Gamma$ has a virtual character odd
under $g\mapsto g^{-1}$, so the result follows from 
\cite{MR1339924,MR2366359}.
\end{proof}
It turns out that there is no difference here 
between the spin and spin$^c$ cases, i.e., the map $c$ on the lower left
in diagram \eqref{eq:1} turns out to be an isomorphism in the cases
of Theorem \ref{thm:rangeofrho}.

\section{Moduli spaces of GPSC metrics}
\label{sec:psc}
Let $(M,\sigma,L)$ be a closed,
connected spin$^c$ manifold.
We denote by $\Gamma$ the fundamental group of
$M$. Let $\mathrm{Diff}(M,\sigma)$ be the subgroup
of $\Diff(M)$ consisting of the diffeomorphisms that
preserve the spin$^c$ structure $\sigma$ (up to isomorphism).
One of the goals of this section is to study the set of path
components 
\begin{equation*}
\pi_0 \left(\mathcal{R}^{c,+}(M,\sigma,L)/G \right)
\end{equation*}
of the moduli space $\mathcal{R}^{c,+}(M,\sigma,L)/G$, where
$G\leq \Diff(M,\sigma)$. Here $\mathcal{R}^{c,+}(M,\sigma,L)$
denotes the set of pairs $(g,\nabla_L)$ with gpsc.

We shall investigate a more general question,
namely we estimate a lower
bound for the rank of the coinvariant groups ${\rm R}^{\spinc}_{*+1}
(M)_G$ and ${\rm Im}( {\rm R}^{\spinc}_{*+1} (M)_G\rightarrow {\rm
Pos}^{\spinc}_* (M)_G)$.\footnote{Assuming $G$ acts on an abelian 
group $A$,
the \emph{coinvariant group} $A_G$ is the quotient of $A$ by the
subgroup generated by the set $\{a-g\cdot a: a\in A, g\in G\}$.
This is the same as $H_0(G,A)$.}
(We shall explain why we are interested in these groups
later in this Section, and we shall also define
the action of $G$ on ${\rm Pos}^{\spinc}_* (M)$ and on
${\rm R}^{\spinc}_{*+1} (M)$.) Our results will be of two types:
\begin{itemize}
\item results valid for general fundamental groups containing at least one element of finite order;
\item results valid for fundamental groups satisfying  special properties, for example  Gromov hyperbolic groups.
\end{itemize}
As in the spin case we shall make use of
(higher) $\rho$ invariants associated to the
spin$^c$ Dirac operator of a generalized spin$^c$ metric. Our main
tools will will include the spin$^c$ Bordism Theorem,
Theorem \ref{thm:spincbordism}, the Classification Theorem,
Theorem \ref{thm:classification}, as well as extensions of a few
results from the spin to the spin$^c$ case. Instead 
of attempting to generalize all currently available
results from the spin case, we shall just
focus on a few crucial cases.

This Section is organized as follows: in
Subsection \ref{subsect:spinc-preserving} we define the subgroup of
spin$^c$ preserving diffeomorphisms and describe its
basic properties; in Subsection
\ref{subsect:induced-actions}
we define actions of this group on the groups appearing in the Stolz
spin$^c$ sequence; in Subsection
\ref{subsect:general-with-torsion} we state results valid for general
fundamental groups containing at least one element of finite order;
and finally, in Subsection
\ref{subsect:hyperbolic}, we treat the special case when the 
fundamental group is Gromov hyperbolic and we state sharper results
using higher $\rho$ invariants.

In the concluding two subsections, we shall be
concise and merely state the results, as their proofs are variations
on the spin case and are detailed in Appendix \ref{appendix2}.

\subsection{Diffeomorphisms preserving a spin$^c$ structure}\label{subsect:spinc-preserving}
Let $(M,\sigma)$ be a spin$^c$ manifold. 
We fix the spin$^c$-structure $\sigma$ through $({\rm P}_{\widetilde{\rm GL}^c} (M),\zeta\co {\rm P}_{\widetilde{\rm GL}^c} (M) \to {\rm P}_{GL^+} (M))$. See 
Appendix \ref{sec:spinc-structures} for more details.
As usual, we denote the associated line bundle by $L$.

Let $\psi\in \Diff(M)$ be an orientation-preserving diffeomorphism.
The differential of $\psi$ induces a smooth map $d\psi: {\rm P}_{GL^+} (M)\to  {\rm P}_{GL^+} (M)$.

\begin{definition} The orientation-preserving diffeomorphism 
$\psi$ is said to \emph{preserve the
 spin$^c$ structure $\sigma$} if 
 $\psi^*\sigma:= ((d\psi)^*  {\rm P}_{\widetilde{\rm GL}^c} (M), (d\psi)^* \zeta)$ is isomorphic to $\sigma$.
 We shall write that $\psi^*\sigma=\sigma$.
 \end{definition}
  
  \noindent
If $\psi^*\sigma=\sigma$, then, consequently, 
$\psi$ preserves the isomorphism class of the line bundle $L$.
The orientation-preserving
diffeomorphisms preserving the spin$^c$ structure $\sigma$ form a subgroup of $\Diff(M)$, denoted
$\Diff(M,\sigma)$. 

\begin{remark}
Note that in general $\Diff(M,\sigma)$ is {\em not} a
finite index subgroup of $\Diff(M)$.  A case where 
it is not is the manifold $M=\bC\bP^2\#T^4$, with associated line bundle $L$
coming from the bundle $\cO(1)$ on $\bC\bP^2$ and a non-trivial even
bundle on $T^4$ to be specified shortly.  
(This manifold has torsion-free $H^2$, so specifying
a spin$^c$ structure is the same as specifying an
orientation and the spin$^c$ line
bundle.) Note that $H^2(T^4)$ is free abelian of rank
$\binom 4 2 = 6$, and $GL(4,\bZ)$ acts on $H^1(T^4; \bZ)\cong \bZ^4$
in the obvious way.  If we let $x_1,\cdots,x_4$ be a free basis of
$H^1(T^4; \bZ)$, then the orbit of $2x_1\wedge x_2$ in $H^2(T^4;\bZ)$
under $GL(4,\bZ)$ is infinite.  But since $GL(4,\bZ)$ acts on $T^4$
by diffeomorphisms fixing a basepoint, it also acts by diffeomorphisms
on $M$ and its orbit on $c_1(L)\in H^2(M;\bZ)$ is infinite if
$c_1(L)$ restricts to $c_1(\cO(1))$, a generator of $H^2(\bC\bP^2;\bZ)$, and
to $2x_1\wedge x_2$ in $H^2(T^4;\bZ)$.  Thus in this case $\Diff(M,\sigma)$
is of infinite index in $\Diff(M)$.
\end{remark}

 Recall that we have fixed the spin$^c$-structure
$\sigma$ and that we have denoted the associated line bundle by $L$. We fix a unitary connection $\nabla_L$.
We shall need the following Proposition. In its statement we want to
be more precise about the notation for the spin$^c$ Dirac operator $D$
associated to $\sigma$ and to $(g,\nabla_L)$; as $\sigma$ is fixed, we denote it by $D_{(g,\nabla_L)}$. We
denote the induced operator on the universal cover by
$D^\Gamma_{(g,\nabla_L)}$. 
\begin{proposition}\label{prop:invariance-CG}
If $\varphi\in \Diff(M,\sigma)$ then $D_{(g,\nabla_L)}$ and
$D_{(\varphi^*g,\nabla_{\varphi^* L})}$, with $\nabla_{\varphi^* L}$
denoting the pulled-back connection, are unitarily equivalent;
similarly $D^\Gamma_{(g,\nabla_L)}$ and
$D^\Gamma_{(\varphi^*g,\nabla_{\varphi^* L})}$ are unitarily
equivalent.
\end{proposition}
\begin{proof}
The detailed proof of the corresponding statement in the spin case, given in 
\cite[Propositions 2 and 3]{dabrowski-dossena} or in \cite[Proposition 2.10]{MR2366359},
can easily be adapted to the spin$^c$ context.
\end{proof}

\subsection{Actions induced by the group ${\rm Diff}(M,\sigma)$ on the spin$^c$ Stolz sequence}\label{subsect:induced-actions}

Let $(M,\sigma,L)$ be a connected, closed,
totally non-spin spin$^c$ manifold.
Before studying
the various actions induced by $\Diff(M,\sigma)$ we
pause for a moment and go back to the concordance
set $\mathcal{P}^c (M,L)$ and to the groups ${\rm Pos}^{\spinc}_*
(M)$ and ${\rm R}^{\spinc}_{*+1} (M)$ from the Stolz
exact sequence. 

Recall that once we fix a concordance class
$[(g_0,\nabla^0)]\in \mathcal{P}^c (M,L)$, there is a
bijection
\begin{equation*}
\iota_{[(g_0,\nabla^0)]}\co \mathcal{P}^c (M,L) \to {\rm
 R}^{\spinc}_{*+1} (M)\equiv {\rm R}^{\spinc}_{*+1}
 (B\Gamma)\,,\quad\Gamma=\pi_1 (M).
\end{equation*}
We shall also be interested in the natural map 
$b\co \mathcal{P}^c (M,L)\to {\rm Pos}^{\spinc}_* (M)$ 
sending $[(g,\nabla)]$ to the
element $[((M,\sigma,L),{\rm Id}\co M\to M,
g,\nabla)]$, denoted briefly $[((M,\sigma,L),{\rm Id}_M, g,\nabla)]$.  
We denote by $b(\mathcal{P}^c (M,L))$ the image of the map $b$; 
the passage from $\mathcal{P}^c (M,L)$ to $b(\mathcal{P}^c (M,L))$ 
amounts to a change of equivalence relation on the pairs $(g,\nabla)$,
from concordance to bordism (hence the notation $b$ for the map).  
We observe that $b(\mathcal{P}^c (M,L))$ is not a subgroup of 
${\rm Pos}^{\spinc}_* (M)$ but rather a coset.  Indeed it is
a coset of $\mathrm{Im}( {\rm R}^{\spinc}_{*+1} (M)\xrightarrow{\partial}
{\rm Pos}^{\spinc}_* (M))$ since it is equal to
\begin{equation*}
 [(M, L,{\rm Id}_M,
 g_0,\nabla^0)] + {\rm Im}( {\rm R}^{\spinc}_{*+1}
 (M)\xrightarrow{\partial} {\rm Pos}^{\spinc}_* (M))
\end{equation*} once we fix a
basepoint $[g_0,\nabla^0]$ in $\mathcal{P}^c (M,L)$. This follows
immediately from the fact that
\begin{equation*}
b([(g,\nabla)])=\partial \circ \iota_{[g_0,\nabla^0]} ([(g,\nabla)])+
 [(M,L,{\rm Id}_M,g_0,\nabla^0)].
\end{equation*}
Following \cite{PSZ} we define the rank of $b(\mathcal{P}^c (M,L)))$ as 
the rank of ${\rm
 Im}( {\rm R}^{\spinc}_{*+1} (M)\xrightarrow{\partial} {\rm
 Pos}^{\spinc}_* (M))$ (which is an abelian group).

If $\psi\in \Diff(M,\sigma)$, then, as already remarked,  
$\psi$ preserves the spin$^c$ line bundle $L$ and so acts on 
$\mathcal{P}^c (M,L)$ by pull-back. It also acts on ${\rm
 Pos}^{\spinc}_* (M)$ as follows. If $x=[(X,
 u\co X\to M,g_X,\nabla_{L_X})]\in {\rm Pos}^{\spinc}_* (M))$, then
\begin{equation*}
\psi\cdot x:= [(X,\psi\circ u\co X\to M,g_X,\nabla_{L_X})\in {\rm
 Pos}^{\spinc}_* (M).
\end{equation*} 
Here we have adopted the concise notation for elements in 
${\rm Pos}^{\spinc}_* (M)$.

If  $x= [((M, \sigma, L),{\rm Id}_M,g,\nabla_L)]\in 
b(\mathcal{P}^c (M,L))$,
then it is not difficult to prove, see 
\cite[Prop. 11.8]{PSZ} for the spin case, that
$\psi\cdot x:= [(M,{\rm Id}_M\co M\to M,\psi_*(g),\nabla^{\psi_*L})]$ with $\psi_*$ denoting the natural action through $\psi^{-1}$.
Hence $b(\mathcal{P}^c (M,L))$ is sent to itself
by the action of $\Diff (M,\sigma)$, the surjective map
$b\co \mathcal{P}^c (M,L)\to b(\mathcal{P}^c (M,L))$ is equivariant
and descends to a surjective map $b_G\co \mathcal{P}^c
(M,L)/G\to b(\mathcal{P}^c (M,L))/G$ for any $G\leq \Diff(M,\sigma)$.
 
Similarly, we can define the action of $\Diff (M,\sigma)$ on ${\rm
 R}^{\spinc}_{*+1} (M)$ as follows. 
Let $\psi\in \Diff (M,\sigma)$. Then if 
$y=[(Y, u\co Y\to M, g_{\partial Y}, \nabla_{\partial Y, L_Y})]$, 
$\psi\cdot y:= [(Y, \psi\circ u\co Y\to M, g_{\partial Y}, 
\nabla_{\partial Y, L_Y})]$, 
where also in this case we have adopted the concise 
notation for the cycles of ${\rm  R}^{\spinc}_{*+1} (M)$. In the case  $y\in {\rm R}^{\spinc}_{*+1} (M)$, 
$y=[(M\times [0,1], \pr_1\co M\times [0,1] \to M, g\sqcup
 g_1, \nabla_L\sqcup \nabla_{1,L})]$, then
\begin{equation*}
\psi\cdot y= [(M\times
 [0,1],\pr_1\co M\times [0,1] \to M, \psi_*
 g\sqcup \psi_*g_1, \nabla^{\psi_* L}\sqcup \nabla^{1, \psi_*  L})]
\end{equation*}
This implies, in particular, that the map $\iota\co \mathcal{P}^c
(M,L)\times \mathcal{P}^c (M,L)\to {\rm R}^{\spinc}_{*+1} (M)$ is
equivariant with respect to the action of $\Diff(M,\sigma)$ on both
sides (where we consider the diagonal action on the left hand side)
and induces a bijection that we call $\iota_{0,G}$ between
$\mathcal{P}^c (M,L)/G$ and ${\rm R}^{\spinc}_{*+1} (M)_G$ once we
fix as a basepoint on $\mathcal{P}^c (M,L)/G$ the orbit of
$[g_0,\nabla^0]$.
  
\smallskip
\noindent
{\it We define the rank of $\mathcal{P}^c (M,L)/G$ as the 
rank of ${\rm R}^{\spinc}_{*+1} (M)_G$.}
   
\smallskip
\noindent
From the definitions it is obvious that
${\rm R}^{\spinc}_{*+1} (M)\xrightarrow{\partial}  {\rm Pos}^{\spinc}_* (M)$ is  equivariant and descends to a map
$\partial_G$ on the associated coinvariant groups. Taking into 
account the various maps, we see  that $b(\mathcal{P}^c (M,L))/G$ is a coset of ${\rm Im}( {\rm R}^{\spinc}_{*+1} (M)_G\xrightarrow{\partial_G} {\rm Pos}^{\spinc}_* (M)_G)$, since $b_G$ and 
$\partial_G\circ \iota_{0,G}$ differ by the image of the basepoint of 
$\mathcal{P}^c (M,L)/G$ into ${\rm Pos}^{\spinc}_* (M)_G$.
  
\smallskip
\noindent
{\em We define the
  rank of $b(\mathcal{P}^c (M,L))/G$ as the rank of the abelian group 
${\rm Im}( {\rm R}^{\spinc}_{*+1} (M)_G\xrightarrow{\partial_G} 
{\rm Pos}^{\spinc}_* (M)_G)$.}

   \smallskip
  \noindent
Notice that 
$${\rm rank} (\mathcal{P}^c (M,L)/G) \geq {\rm rank} (b(\mathcal{P}^c (M,L))/G)\,,\quad  \text{for any} \quad G\leq {\rm Diffeo}(M,L).$$
Moreover, as $\pi_0 \left(\mathcal{R}^{c,+}(M,\sigma,L)/G \right)$ surjects onto $\mathcal{P}^c (M,L)/G$, we see that 
 if ${\rm rank} (b(\mathcal{P}^c (M,L))/G)>0$ then the cardinality of $\pi_0 \left(\mathcal{R}^{c,+}(M,\sigma,L)/G \right)$ is infinite.

\smallskip
\noindent
{\em One of our goals in this Section will be to give lower bounds on the rank of $b(\mathcal{P}^c (M,L))/G$, for suitable finite index subgroups $G$ of  $ {\rm Diff}(M,\sigma)$.}

\subsection{Results for general fundamental groups containing torsion}\label{subsect:general-with-torsion}
We extend to the spin$^c$ case the results of \cite{MR2366359}, based in turn on  \cite{MR1339924}. We only state
the result here and refer to Appendix \ref{appendix2} for the proof; there it will be clear that a crucial role 
is played by the spin$^c$ bordism theorem.

\begin{theorem}\label{theo:any-torsion-group}
Let $(M,\sigma,L)$ be spin$^c$ and totally non-spin and assume 
that $(g,\nabla_L)$ is gpsc. Assume 
that $\pi_1 (M)\equiv \Gamma$ has an element of finite order and 
that $\dim M=4k+3$, $k\ge 1$. Then the set
$\pi_0 \left(\mathcal{R}^{c,+}(M,\sigma,L)/ {\rm Diff} (M,\sigma) \right)$ has infinite cardinality.
\end{theorem}

\subsection{Results when the fundamental group is Gromov hyperbolic}\label{subsect:hyperbolic}
$\;$\\
Let $(M^n,\sigma,L)$ be a connected 
closed spin$^c$ totally non-spin 
manifold with fundamental group $\Gamma$ and 
universal covering $\tM$. We assume that $n\geq 5$ and that 
$\mathcal{R}^{c,+}(M,\sigma,L) \ne \emptyset$.
We fix a base point $(g_0,\nabla^0)\in \mathcal{R}^{c,+}(M,\sigma,L)$.
In this Subsection we will present results on 
${\rm rank} (b(\mathcal{P}^c (M,L))/G)$. 
As in the previous subsection we only state the results and refer to Appendix \ref{appendix2} for details.
The proofs use higher rho numbers as well
as  the classification theorem,
identifying $\mathcal{P}^c (M,L)$ with ${\rm R}^{\spinc}_* (M)$.

\smallskip
\noindent
Recall the definition of $F\Gamma$ from
the discussion just before Proposition \ref{prop:from-XYZ} and in particular
the cohomology groups $H^{*}(\Gamma;F^p_0\Gamma)$, $p\in\{0,1\}$.

\begin{theorem}\label{theo:moduli1}
Let $\Gamma$ be a Gromov hyperbolic group with finite outer automorphism group
and let $(M^n,\sigma, L)$ be a closed connected spin$^c$ totally non-spin  manifold  with
$\pi_1(M)=\Gamma$.
We assume that $n\geq 5$ and that $\mathcal{R}^{c,+}(M,\sigma,L) \not= \emptyset$. Then
there exists a finite index subgroup $G\leq {\rm Diff}(M,\sigma)$ such that     
\begin{equation*}
     {\rm rank} (b(\mathcal{P}^c (M,L))/G)\geq 
      \sum_{k>0,\,p\in\{0,1\}} \rank(H^{n+1-4k-2p}(\Gamma;F^p_0\Gamma)). 
\end{equation*}
\end{theorem}
    
\noindent
Notice that as already remarked in \cite{PSZ} and proved by 
Bestvina and Feighn \cite{B-F}, the 
finiteness of ${\rm Out}(\Gamma)$ is true for a generic 
Gromov hyperbolic group.
        
\begin{theorem}\label{theo:moduli2}
Let $M$ be as in the previous Theorem and let $\Gamma$ be  Gromov hyperbolic. Then there 
exists a finite index subgroup $G\leq {\rm Diff}(M,\sigma)$ such that     
$ {\rm rank} (b(\mathcal{P}^c (M,L))/G)\geq \ell$  with   
 \begin{equation*}
      \ell :=
      \begin{cases}
        \abs{\{[\gamma]\subset\Gamma\mid \ord(\gamma)<\infty\}}; & n\equiv -1\pmod 4\\
        \abs{\{[\gamma]\subset\Gamma\mid [\gamma]\ne[\gamma^{-1}]\}}; &
        n\equiv 1\pmod 4 
      \end{cases}
    \end{equation*}
\end{theorem}

\begin{remark}
For spin manifolds, a lower bound  similar to this one was first proved by Xie and Yu in \cite{Xie-Yu-Moduli}.
\end{remark}


\appendix
\section{Spin$^c$ structures and Dirac operators}
\label{sec:spinc-structures}
\subsection{Spin$^c$ structures}
For the convenience of the reader, we summarize in this Appendix the basics on spin$^c$ manifolds.
We refer to \cite{MR0358873}, \cite[Appendix D]{lawson89:_spin} and \cite[Section 1.3.3]{nicolaescu-sw}
for details.

Let $(M,g)$ be an oriented Riemannian manifold.  Recall that the group 
$\Spin (n)$ comes with a natural 2-to-1 homomorphism
$\varphi\co \Spin(n)\to {\rm SO}(n)$. By
definition $\Spinc (n):= \Spin(n)\times_{\bZ/2} {\mathrm U}(1)$. This
group fits into three distinct short exact sequences:
\begin{equation}\label{ses-1}
1\rightarrow \bZ/2 \rightarrow {\rm Spin}^c (n) \xrightarrow{\varphi^c} {\rm SO}(n)\times {\mathrm U}(1)\rightarrow 1,
\end{equation} 
where
$\varphi^c ([u,\lambda])= (\varphi (u),\lambda^2)$;
\begin{equation}\label{ses-2}
1\rightarrow {\mathrm U}(1) \rightarrow {\rm Spin}^c (n) \xrightarrow{\eta^c} {\rm SO}(n)\rightarrow 1;
\end{equation} 
\begin{equation}\label{ses-3}
1\rightarrow {\rm Spin} (n) \rightarrow {\rm Spin}^c (n) \xrightarrow{\theta^c} {\mathrm U}(1)\rightarrow 1.
\end{equation} 
Following \cite{MR0358873} 
and employing  \eqref{ses-2} we define a spin$^c$ structure $\sigma$ for $M$ as a $\Spinc (n)$ principal bundle 
${\rm P}_{\Spinc} (M)\to M$ together with a smooth map 
${\rm P}_{\Spinc} (M)\xrightarrow{\xi} {\rm P}_{{\rm SO}} (M)$, where 
${\rm P}_{{\rm SO}} (M)$ is the principal ${\rm SO}(n)$-bundle of 
oriented orthonormal frames on $(M,g)$, such that
$\xi (p h)= \xi (p) \eta^c (h)$ for all $p\in {\rm P}_{\Spinc} (M)$ and all $h\in {\rm Spin}^c (n)$. Notice that
 ${\rm P}_{\Spinc} (M)\xrightarrow{\xi} {\rm P}_{{\rm SO}} (M)$ then defines a $U(1)$ principal bundle. Two spin$^c$ structures
$\sigma$ and $\sigma'$ are isomorphic if there exists an isomorphism $\beta: {\rm P}_{\Spinc} (M) \to {\rm P}_{\Spinc} (M)'$
of $\Spinc (n)$ principal bundles with the property that $\xi=\xi' \circ \beta$.

Given a spin$^c$ structure $\sigma$ on $M$, $\sigma=({\rm P}_{\Spinc} (TM),\xi)$,  we obtain from the sequence
\eqref{ses-2} a principal $\mathrm U(1)$-bundle over $M$, or
equivalently, a complex hermitian line bundle $L_\sigma\to M$,
whose first Chern class reduces to $w_2 (M)$ mod 2.
Conversely, it can be proved that if $w_2 (M)$ is the mod 2 reduction of a class in $H^2 (M,\bZ)$, then $M$ admits a spin$^c$-structure.

Equivalently, see \cite[Appendix D]{lawson89:_spin}, building on \eqref{ses-1}, 
we can define a spin$^c$ structure as a triple 
$({\rm P}_{\Spinc} (M),P_{{\rm U}(1)} (M), \Phi^c)$ where ${\rm P}_{\Spinc} (M)$ is a $\Spinc (n)$ principal bundle over $M$,
$P_{{\rm U}(1)} (M)$ a $U(1)$-principal bundle over $M$ 
and $\Phi:  {\rm P}_{\Spinc} (M)\to {\rm P}_{\rm SO} (M)\times_M P_{{\rm U}(1)} (M)$  a  $\bZ/2$-cover such that 
$\Phi^c (p h)= \Phi^c (h) (p) \varphi^c (h)$. The $U(1)$-principal bundle $P_{{\rm U}(1)} (M)$ defines the line bundle
$L$ associated to the spin structure.

\begin{remark}\label{rem:changing-the-line-bundle} Given a spin$^c$ structure as above,
$({\rm P}_{\Spinc} (M),P_{{\rm U}(1)} (M), \Phi^c)$, we can 
define a new spin$^c$ structure by considering $\alpha\in H^2 (M,\bZ)$ and a principal ${\rm U}(1)$-bundle
$P^\alpha_{{\rm U}(1)} (M)$ with first Chern class $\alpha$. Denote by $L_\alpha$  the corresponding 
hermitian  line bundle.
Then 
$$({\rm P}_{\Spinc} (M)\times_M P^\alpha_{{\rm U}(1)} (M))/ {\rm U}(1)\xrightarrow{\Phi^c_\alpha} {\rm P}_{\rm SO} (M)\times_M (P_{{\rm U}(1)} (M)\otimes P^{2\alpha}_{{\rm U}(1)} (M))$$
defines a new spin$^c$ structure with associated line bundle $L\otimes (L_\alpha\otimes L_\alpha)$, \end{remark}

As in \cite[Remark 1.19]{lawson89:_spin} and \cite{dabrowski-percacci} we can give a metric-independent definition.
Consider ${\rm GL}^+(n)$, the connected component of the identity
in $\mathrm{GL}(n,\bR)$, and let 
$\rho\co \widetilde{{\rm GL}}^+ (n)\to {\rm GL}^+(n)$ be the 
unique nontrivial double cover.
Let ${\rm P}_{GL^+} (M)$ the bundle of oriented frames, a  ${\rm GL}^+(n)$ principal bundle.
Consider $\bC^{*}=\{z\in\bC\;|\; z\not= 0\}$. 
We define $$\widetilde{{\rm GL}}^c (n):= \widetilde{{\rm GL}}^+ (n)\times \bC^{*}/\{\pm (1,1)\}$$
We have a short exact sequence of groups
\begin{equation}\label{ses-2-bis}
1\rightarrow \bC^{*} \rightarrow \widetilde{{\rm GL}}^c (n) \xrightarrow{\eta^c} {\rm GL}^+(n)\rightarrow 1.
\end{equation} 
A spin$^c$ structure for the smooth oriented
manifold $M$ is a $\widetilde{{\rm GL}}^c (n)$-principal bundle 
${\rm P}_{\widetilde{\rm GL}^c} (M)$ over $M$ 
together with
a map $\zeta\co {\rm P}_{\widetilde{\rm GL}^c} (M) \to {\rm P}_{GL^+} (M)$ 
with the property that $\zeta (p g)=\zeta(p) \eta^c (g)$. 
Isomorphic spin$^c$-structures are defined by an obvious modification of the definition given
above. The metric-dependent definition is obtained by observing that 
${\rm P}_{{\rm SO}} (M) \hookrightarrow {\rm P}_{GL^+} (M)$
and defining 
\[
{\rm P}_{\Spinc} (M):= \zeta^{-1}{\rm P}_{{\rm SO}} (M)\quad\text{and}\quad \xi:= \zeta |_{{\rm P}\Spinc (TM)}\,.
\]

Let $\mathcal{L}_M:=\{\alpha\in H^2 (M,\bZ)\;\;\text{such that}\;\; \alpha \equiv w_2 (M) \;\;{\rm mod} \;\;2\}$ and let us denote
by ${\rm Spin}^c (M)$ the collection  of isomorphism classes of spin$^c$-structures. If $\sigma$ is isomorphic
to $\sigma'$ then from standard properties of principal bundles 
we have that  $L_\sigma$ is isomorphic to $L_{\sigma'}$.
Thus the  map 
\[
{\rm Spin}^c (M)\ni [\sigma]\to c_1 (L_\sigma)\in \mathcal L_M
\]
is well defined; it can be proved that this map is surjective. 
Notice however that this map is \emph{not}
injective unless $H^2 (M,\bZ)$ has no 2-torsion elements. See 
\cite[Section 1.3.3]{nicolaescu-sw}.  
Nevertheless, for most
of what we do in this paper, only the image of a spin$^c$ structure
in $\mathcal{L}_M$ will matter; we call it the spin$^c$ line bundle
of the spin$^c$ structure.

\subsection{The spin$^c$-Dirac operator associated to a spin$^c$-structure}
Given a $\bC \ell_n$-module $V$ we can consider the associated representation $\Delta: {\rm Spin}^c (n)\to {\rm GL}(V)$.
Given a spin$^c$-structure $\sigma=({\rm P}_{\Spinc} (M),\xi)$, this defines a spinor bundle 
${\rm P}_{\Spinc} (M)\times_\Delta V$. If the representation is irreducible then this spinor
bundle is said to be fundamental. It can be proved that up to isomorphisms there exists only one fundamental
spinor bundle, denoted  $S(M)\to M$ (a more precise notation would be $S_\sigma(M)\to M$
but the spin$^c$-structure is usually fixed and for this reason we shall not follow this precise notation). 
In even dimension the spinor bundle is
$\bZ/2$-graded, $S(M)=S^+ (M)\oplus S^- (M)$. 

Given a unitary  connection $\nabla_{L_\sigma}$ on $L_\sigma$ 
and employing the Levi-Civita connection associated to the metric 
$g$ on $M$ we can define a connection on $S(M)$ and $S(M)$
becomes, in this way, a Dirac bundle.  The associated Dirac operator 
$D$ is, by definition, the spin$^c$-Dirac operator.
See again \cite[Appendix D]{lawson89:_spin} for all this. Notice that $D$ depends on $\sigma$, the metric $g$ and the 
connection $\nabla_{L_\sigma}$.

\section{The Higson-Roe sequence and higher rho numbers: definitions}
\label{appendix1}
In this second Appendix we give more information about the Higson-Roe
surgery sequence and one of its variants we alluded to back in
Subsection \ref{subsect:alternative}, the one through
pseudodifferential operators used in \cite{PSZ}; we also define higher
rho numbers when $\Gamma$ is Gromov hyperbolic.

\subsection{More on the mapping of the spin$^c$ Stolz sequence to the Higson-Roe sequence}
Recall that the Higson-Roe sequence is the $K$-theory sequence
associated to the short exact sequence of $C^*$-algebras
\begin{equation*}
0\to C^*
(\widetilde{X})^\Gamma\to D^* (\widetilde{X})^\Gamma\to D^*
(\widetilde{X})^\Gamma/C^* (\widetilde{X})^\Gamma \to 0.
\end{equation*}
For the definition of the $C^*$-algebras $C^* (\tX)$ and $D^* (\tX)$,
we refer to the original work of Higson and Roe; a quick introduction
can be found in \cite{MR3286895}. The basic idea is that $D^*(\tX)$ is
the closure of the pseudo-local operators of finite propagation, and
$C^*(\tX)$ is the ideal inside $D^*(\tX)$ obtained as the closure of
the operators that are of finite propagation and locally compact.
Regarding the universal Higson-Roe surgery sequence, figuring in
Theorem \ref{theo:commute-r-pos}, we recall that
\begin{equation*}
      K_*(C^* (E\Gamma)^\Gamma):=  \dirlim_{X\subset E\Gamma\text{ $\Gamma$-compact}}
      K_*(C^*(X)^\Gamma);\quad 
      K_*(D^* (E\Gamma)^\Gamma):= \dirlim_{X\subset E\Gamma\text{ $\Gamma$-compact}} K_*(D^*(X)^\Gamma).
\end{equation*}
Here, $E\Gamma$ is any contractible CW-complex with free cellular
$\Gamma$-action, a  universal space for free $\Gamma$ actions. We  
recall that there are   canonical isomorphisms
\[
K_{*+1} (D^* (\widetilde{X})^\Gamma/C^* (\widetilde{X})^\Gamma)=K_* (X)\;\;\;\text{and}\;\;\; K_*(C^* (\tX)^\Gamma)=K_*(C^*_r\Gamma)=K_*(C^*(E\Gamma)^\Gamma).
\]
We set $\SG^\Gamma_* (\tX):=K_{n+1}(D^*(\tX)^\Gamma) $ and we call
this group the \emph{analytic structure group} associated to $X$.  The
analytic surgery sequence can then be written as
\begin{equation}\label{HR-compact-notation}
\cdots\rightarrow K_{n+1}(X)\rightarrow K_{n+1}(C^*_r \Gamma)\rightarrow \SG^\Gamma_* (\tX) \rightarrow K_{n}(X)\rightarrow \cdots
\end{equation}
Similarly, we set $\SG^\Gamma_* (E\Gamma):=  K_{*+1}(D^* (E\Gamma)^\Gamma)$ and the {\it universal} analytic surgery sequence can be written as
\begin{equation}\label{HR-compact-notation-universal}
\cdots\rightarrow K_{n+1}(B\Gamma)\rightarrow K_{n+1}(C^*_r \Gamma)\rightarrow \SG^\Gamma_* (E\Gamma) \rightarrow K_{n}(B\Gamma)\rightarrow \cdots
\end{equation}
Let us see how the vertical maps figuring in Theorem
\ref{theo:commute-r-pos} are defined.
The map $\beta^c$ associates to $[((M,\sigma,L),
f\co M\to X)]\in \Omega^{\spinc}_{*} (X)$ the class $f_* [D]\in K_*
(X)$, with $[D]\in K_* (M)$ the class associated to the spin$^c$ Dirac
operator defined by $(M,\sigma,L)$. It is
well-known that the transformation $\beta^c\co \Omega^{\spinc}_{*}
(X)\to K_* (X)$ is well-defined. We recall that
the transformation $\beta^c$ can be identified with the map of
homology theories obtained by composing $\alpha^c\co \MSpinc\to \ku$
and the periodization map $\per\co \ku\to \K$.

The map $\Ind^c_\Gamma$ sends a class $[((M,\sigma,L),\varphi\co M\to
X, g_{\partial M}, \nabla_{L,{\partial}})]\in \mathrm{R}^{\spinc}_{*}(X) $ to
the class in $K_{n+1} ( C^*(\tilde X)^\Gamma) $ defined as follows. We
extend $(g_{\partial M}, \nabla_{L,{\partial}})$ to $(g_M,\nabla_L)$
on all of $M$; we consider $M_\Gamma= \varphi^* \widetilde{X}$ and the
associated $\Gamma$-cover with cylindrical ends $M^\infty_\Gamma$; we
lift $(g_{M}, \nabla_{L})$ to $M^\infty_\Gamma$ and consider the
associated $\Gamma$-equivariant spin$^c$ operator $D_\Gamma$; because
of the gpsc assumption on $(g_{\partial M}, \nabla_{L,{\partial}})$
this has an index class in $K_* (C^* (M_\Gamma\subset
M^\infty_\Gamma)^\Gamma)$, a group that can be identified with $K_*
(C^* (M_\Gamma)^\Gamma)$; we denote by $\Ind_\Gamma (D_\Gamma)$ this
index class. We set
\[
\Ind^c_\Gamma [((M,\sigma,L),\varphi\co M\to X, g_{\partial M}, 
\nabla_{L,{\partial}})]:= (\varphi^\Gamma)_* (\Ind_\Gamma (D_\Gamma))
\in K_* (C^* (\widetilde{X})^\Gamma)
\]
with $\varphi^\Gamma\co M_\Gamma\equiv \varphi^* \tX\to \tX$ the
$\Gamma$-equivariant map covering $\varphi\co M\to X$. Notice that
there are different realizations of the index class $\Ind_\Gamma
(D_\Gamma)\in K_* (C^* (M_\Gamma)^\Gamma)$, for example a
Mishchenko-Fomenko description; see \cite{MR3286895} for a detailed
discussion.  The proof that $\Ind^c_\Gamma$ is well defined is the
same as in the spin case, based on a relative index theorem of
Bunke. (This employs the compatibility between the Mishchenko-Fomenko
description of the APS index class and the coarse index class, alluded
to above.)

Finally, if $[((M,\sigma,L),\varphi\co M\to
X,g,\nabla_L)]\in {\rm Pos}^{\spinc}_* (X)$, then
\[
\rho^c_\Gamma [( (M,\sigma,L),\varphi\co M\to X,g,\nabla_L)]:= (\varphi^\Gamma)_* (\rho (D_\Gamma))\in K_{*+1} (D^*(\tX)^\Gamma)
\]
with $\rho (D_\Gamma)\in K_{*+1} (D^* (M_\Gamma)^\Gamma)$ the rho class
associated to the $\Gamma$-equivariant $L^2$-invertible spin$^c$ Dirac
operator $D_\Gamma$ associated to the lift of $(g,\nabla_L)$ to
$M_\Gamma$. Thus if $M$ is odd dimensional, then
\[
\rho (D_\Gamma):=\left[\frac{1}{2}\left(\frac{D_\Gamma}{|D_\Gamma|} 
     + 1\right)\right]\in K_{0} (D^* (M_\Gamma)^\Gamma),
\]
whereas if $M$ is even dimensional, so that $D_\Gamma$ acts on the sections of a $\mathbb{Z}_2$-graded bundle, then
\[
\rho (D_\Gamma):=[U \chi (D_\Gamma)_+]\in K_{1} (D^* (M_\Gamma)^\Gamma)
\]
with $\chi$ a odd smooth real function equal to the sign function on the spectrum of $D_\Gamma$
and $U$ a suitable isometry $L^2 (M_\Gamma, S^-_\Gamma)\to L^2(M_\Gamma, S^+_\Gamma)$
with $S_\Gamma$ the spinor bundle over $M_\Gamma$.

This homomorphism is well-defined thanks to the {\em delocalized APS
index theorem in $K$-theory}, one of the main results
in \cite{MR3286895} (see Theorem 1.14 there).  This latter result is
also used crucially in order to show the commutativity of the central
square in \eqref{eq:StolzToAnaX}.  We omit further details because the
analysis here is very similar to the one carried out in the spin case
in \cite{MR3286895} (and for the even case in \cite{zenobi-jta}).

\medskip
Let us now see the description of the Higson-Roe sequence through pseudodifferential operators; we shall see in the next
section why it is useful.
Consider  $\Psi^0_{\Gamma,c} (\widetilde{X})$,  the algebra of 0-th order $\Gamma$-equivariant pseudodifferential operators of $\Gamma$-compact support and let
 $\Psi^0_\Gamma(\widetilde{X})$ denote the $C^*$-closure of $\Psi^0_{\Gamma,c} (\widetilde{X})$.
Consider the classic short exact sequence  of pseudodifferential operators:
\begin{equation*} 
 0\to C^*_{red}(\widetilde{X}\times_\Gamma\widetilde{X})\rightarrow
\Psi^0_{\Gamma}({\widetilde{X}})\xrightarrow{\sigma} C(S^*M) 
\to 0
\end{equation*}
where $\sigma$ is the principal symbol map and
$C^*_{red}(\widetilde{X}\times_\Gamma\widetilde{C})$ is the
$C^*$-closure of $\Psi^{-\infty}_{\Gamma,c}(\widetilde{X})$, the ideal
of smoothing $\Gamma$-equivariant operators with $\Gamma$-compact
support.  This induces a long exact sequence in $K$-theory for the
mapping cone $C^*$-algebras:
\begin{equation}\label{PSZ-HR}
\xymatrix@C-1em{\cdots\!\ar[r]^(.25){\partial}\!& \!K_*(0\hookrightarrow C^*_{red}(\widetilde{X}\times_\Gamma\widetilde{X}))\ar[r]^(.45){i_*}\!& \!K_*(C(X)\xrightarrow{\mathfrak{m}}\Psi^0_{\Gamma}({\widetilde{X}}))\ar[r]^(.45){\sigma_*}& K_*(C(X)\xrightarrow{\mathfrak{\pi^*}} C(S^*X))\ar[r]^(.75){\partial}&\cdots	}
\end{equation}
where $\pi\co S^*X\to X$ is the bundle projection of the cosphere
bundle and $\mathfrak{m}$ is obtained by lifting a function from $X$
to $\widetilde{X}$ and then seeing it as a multiplication operator on
$L^2$.  We observe that $K_*(0\hookrightarrow
C^*_{red}(\widetilde{X}\times_\Gamma\widetilde{X}))$ is equal to
$K_{*}(C^*_{red}(\widetilde{X}\times_\Gamma \widetilde{X})\otimes
C_0(0,1))$.  There is an analogous and equivalent sequence if we take
pseudodifferential operators acting on the sections of a
$\Gamma$-equivariant vector bundle $\widetilde{E}$.  As for the other
alternative approaches, there is a natural isomorphism from this
sequence to the Higson-Roe sequence; in particular
\begin{equation}\label{PSZ:structure-group}\SG^\Gamma_{*} (\widetilde{X})\simeq K_*(C(X)\xrightarrow{\mathfrak{m}}\Psi^0_{\Gamma}(\widetilde{X}, \widetilde{E}))\,.
\end{equation}
Moreover, the rho class  previously defined and the rho class defined in \cite{PSZ} correspond under this
isomorphism. See \cite[Section 5.3]{zenobi-jncg}.

\subsection{Higher rho numbers}\label{subsect:higher-rho}
One of the great advantages of expressing the Higson-Roe sequence in terms of pseudodifferential operators
is that it is then possible to map the entire sequence to a sequence in noncommutative de Rham homology; more precisely, given any dense holomorphically closed subalgebra $\mathcal{A}\Gamma$ of $C^*_r\Gamma$ there are Chern character homomorphisms defining a commutative diagram 
{\small 	\begin{equation}\label{mktoh}
 \xymatrix{
 \cdots\ar[r]^(.2){\partial}& K_*(0\to C^*_{red}({\tX}\times_\Gamma {\tX}))\ar[r]^(.50){i_*}\ar[d]^{\Ch_\Gamma}& K_*(C(X)\xrightarrow{\mathfrak{m}}\Pdo^0_{\Gamma}({\tX}))\ar[r]^(.50){\sigma_*}\ar[d]^{\Ch_\Gamma^{\del}}& K_* (C(X)\xrightarrow{\mathfrak{\pi^*}} C(S^*X))\ar[d]^{\Ch_\Gamma^{e}}\ar[r]^(.75){\partial}&\cdots\\
 	\cdots\ar[r]&H_{[*-1]}(\mathcal{A}\Gamma)\ar[r]& H_{[*-1]}^{\del} (\mathcal{A}\Gamma)\ar[r]^{\delta}& H_{[*]}^{e} (\mathcal{A}\Gamma)\ar[r]&\,\cdots.}
            \end{equation}}%
See \cite[Section 6]{PSZ}. We shall not digress in order to  recall the necessary definitions, except to say that here the superscript
$e$ denotes localization at the identity and the superscript
$\del$ denotes ``delocalization'' away from the identity.
But we do need to mention a
consequence of \eqref{mktoh}. Recall the decomposition of 
the periodic cyclic cohomology of $\bC\Gamma$, $HP^*(\bC\Gamma)$, 
as the direct product
of groups  $HP^* (\bC\Gamma;\langle x \rangle )$ localized at the conjugacy classes of $\Gamma$:
$HP^* (\bC\Gamma)=\prod_{\langle x \rangle} HP^* (\bC\Gamma;\langle x \rangle )$.
If $\Gamma$ is now a Gromov hyperbolic group 
and $\mathcal{A}\Gamma$ is the Puschnigg algebra, then through the 
Chern character $\Ch_\Gamma^{\del}$ it is possible to define a 
pairing between the periodic cyclic cohomology group $HP^{*-1} (\bC\Gamma;\langle x \rangle )$ and $\SG^\Gamma_{*} (\widetilde{X})$. 
More precisely, we can define a pairing
\begin{equation}\label{S-pairing}
\SG^\Gamma_*(\tX)\times HP^{*-1}(\bC\Gamma;\langle x \rangle) \equiv  K_{*}\left(C(M)\to\Pdo^0_\Gamma({\tX})\right)\times HP^{*-1}(\mathbb{C}\Gamma;\langle x \rangle)\to \bC
\end{equation}
given by
\begin{equation}\label{S-pairing-bis}
(\xi,[\tau])\mapsto \langle \xi,[\tau] \rangle_{{\rm S}}:=\langle\Ch_\Gamma^{\del}(\xi), \tau\rangle
\end{equation}
Applied in particular to the rho class of an invertible
Dirac operator, this allows to define {\em higher rho numbers}, parametrized by the elements of the (delocalized)
periodic cyclic cohomology of $\bC\Gamma$. 
Notice that
this pairing is compatible with the (usual) pairing, denoted $\langle \cdot,\cdot \rangle_{{\rm K}}$, between{\linebreak} $K_{*+1}(C^*_{red}({\tX}\times_\Gamma {\tX}))\equiv 
K_{*+1}(C^*_{r}\Gamma)$ and
cyclic cohomology, defined through $\Ch_\Gamma$ on the left column of \eqref{mktoh}: 
\begin{equation}\label{ch-c*}
\langle \xi,[\tau] \rangle_{{\rm K}} := \langle \Ch_\Gamma(\xi),[\tau] \rangle\;\;\text{for}\;\;[\tau]\in
HP^{[*-1]}(\mathbb{C}\Gamma).
\end{equation}
More precisely, for $\xi\in K_{*+1}(C^*_{red}({\tX}\times_\Gamma {\tX}))$ and $[\tau]\in
HP^{[*-1]}(\mathbb{C}\Gamma; \langle x \rangle)$ we have
\begin{equation}\label{iota-compatibility}
\langle \Ch_\Gamma(\xi),[\tau] \rangle= \langle\Ch_\Gamma^{\del}(\iota(\xi)), [\tau]\rangle
\end{equation} where $\iota\colon K_{*+1}(C^*_{red}({\tX}\times_\Gamma {\tX})\to \SG^\Gamma_*(\tX)$ is the homomorphism appearing in the Higson-Roe sequence. 
 We can rewrite \eqref{iota-compatibility}
as
\begin{equation}\label{iota-compatibility-bis}
\langle \xi,[\tau] \rangle_{{\rm K}}= \langle \iota (\xi),[\tau]\rangle_{{\rm S}}.
\end{equation}
Unless confusion should arise we shall not use these precise notations for the parings.\\
Needless to say, the fact that  $\tau\in HP^{[*-1]}(\mathbb{C}\Gamma; \langle x \rangle)$ is extendable to 
a delocalized cyclic cocycle for the Puschnigg algebra, that is to an element in $HP^{[*-1]}_{\del}(\mathcal{A}\Gamma)$
that can then be paired with $H_{[*-1]}^{\del} (\mathcal{A}\Gamma)$,
is highly non-trivial and constitutes one of the
contributions of \cite{PSZ}. The case of 0-degree cocycles was established in the seminal paper of Puschnigg 
\cite{MR2647141}.

Notice that what we just explained 
pertains the Higson-Roe sequence and its mapping to homology and so 
does not see the difference
between spin and spin$^c$. The same is true for the next subsection.\\
For an alternative approach to the definition of  higher rho numbers see \cite{CWXY-higher-rho}.

\subsection{Further actions induced by ${\rm Diff}(M,\sigma)$}

\smallskip
\noindent
We go on explaining further actions induced by ${\rm Diff}(M)$ 
and ${\rm Diff}(M,\sigma)$; we  refer to \cite[Section 11]{PSZ} for details. 

\smallskip
\noindent
 First of all, we have a group 
homomorphism $\alpha: {\rm Diff}(M)\to {\rm Out}(\Gamma)$. Through $\alpha$ we have actions of ${\rm Diff}(M)$
on the noncommutative de Rham homology group figuring in \eqref{mktoh} and, dually, on the corresponding
cyclic cohomology groups $HC^* (\mathcal{A}\Gamma)$, $HC^*_e (\mathcal{A}\Gamma)$, $HC^*_{\del} (\mathcal{A}\Gamma)$. If $\psi\in  {\rm Diff}(M)$ and $\alpha(\psi)\in {\rm Out}(\Gamma)$,
we denote by $\alpha (\psi)_*$ and $\alpha(\psi)^*$ these two actions. Next, there is an action $\psi_*$  of 
$\psi\in {\rm Diff}(M,\sigma)$ on spaces of pseudodifferential 
operators acting on the spinor bundle
associated to $(M,\sigma)$
and this implies an action $\psi_*$
on the whole Higson-Roe sequence, once we adopt the expression of this sequence through
pseudodifferential operators acting on the sections of the spinor 
bundle $\widetilde{S}$ on $\tM$.  See  
Subsection \ref{subsect:higher-rho}, just before this subsection.
Moreover,
the Chern character homomorphisms in \eqref{mktoh} are equivariant with respect to these two actions defined  on the
upper row and the lower row. See \cite[Theorem 11.11]{PSZ}. This equivariance of the Chern characters has the following important consequence: if $x=\rho_c (D^\Gamma_{(g,\nabla^L)})\in \SG_*^\Gamma (\tM)$
is the rho class associated to the invertible Dirac operator $D^\Gamma_{g,\nabla^L}$ associated to a gpsc metric
$(g,\nabla^L)$, then
given $\tau\in HC^{*-1}_{\del}(\mathcal{A}\Gamma)$ the following formula holds
\begin{equation}\label{compatibility-through-Ch}
\langle \psi_* x,\tau \rangle = \langle x,\alpha(\psi)^* \tau \rangle
\end{equation}
This is in fact true for any $x\in \SG_*^\Gamma (\tM)$.
Formula \eqref{compatibility-through-Ch} and Proposition \ref{prop:invariance-CG} imply the following crucial result, see \cite[Corollary 11.16]{PSZ}:
\begin{equation}\label{compatibility-through-Ch-bis}
\langle \rho_c (D^\Gamma_{\psi^*g,\nabla^{\psi^*L}}),
 \tau \rangle=
\langle \rho_c (D^\Gamma_{(g,\nabla^L)}),\alpha(\psi^{-1})^*\tau \rangle. \end{equation}

\section{Moduli spaces: proofs}
\label{appendix2}
In this third appendix we give proofs for the results claimed in
Section \ref{sec:psc}.  While the proofs build heavily on the bordism
theorem in the spin$^c$ context and on the classification theorem for
the Stolz ${\rm R}^{\spinc}_*$ groups, the actual arguments, once we
have these two fundamental results, are variations of the argument
given in the spin case and for this reason we have decided to present
them in an Appendix.

\subsection{Proof of Theorem \ref{theo:any-torsion-group}} 
First of all, we recall that 
by the Atiyah-Patodi-Singer index theorem and its $L^2$-version 
on Galois coverings we know that
the Cheeger-Gromov rho invariant defines 
group homomorphisms
\begin{equation}\label{CG-rho}
\rho_{(2)}\co {\rm Pos}^{\spin}_n (B\Gamma)\to \bR\;\;\;\text{and}\;\;\; \rho_{(2)}\co {\rm Pos}^{\spinc}_n (B\Gamma)\to \bR\,;
\end{equation}
see \cite{MR2366359} for details.

Notice, incidentally, that by the work of Benameur-Roy \cite{BR},
these homomorphisms are also obtained from the rho homomorphisms
$\rho^c_{\Gamma}\co 
{\rm Pos}^{\spin}_n (B\Gamma)\to\SG^\Gamma_{n,{\rm max}} 
     (E\Gamma)$ and
$\rho^c_{\Gamma}\co {\rm Pos}^{\spinc}_n (B\Gamma)\to 
\SG^\Gamma_{n,{\rm max}} (E\Gamma)$
followed by the application of  
suitable traces on $\SG^\Gamma_{n,{\rm max}} (E\Gamma)$.

The following diagram is clearly commutative:
\begin{equation}\label{eq:3}
\xymatrix{& {\rm Pos}^{\spinc}_n (B\Gamma)\ar[d]^{\rho_{(2)}}\\
{\rm Pos}^{\spin}_n (B\Gamma) \ar[ur]^j \ar[r]^{\;\;\;\;\;\rho_{(2)}} & \,\bR}
\end{equation}
with $j$ as in \eqref{from-spin-to-spinc-pos-map}, that is
\begin{equation*}
{\rm Pos}^{\spin}_* (B\Gamma)\ni
[(M,\varphi\co M\to B\Gamma, g)]
\overset{j}{\longmapsto}
[(M,\varphi\co M\to B\Gamma, g,\nabla^{{\bf 1}})]
\in {\rm Pos}^{\spinc}_* (B\Gamma).
\end{equation*}
Let $u\co M\to B\Gamma$ be the classifying map for the universal 
cover of $M$. Recall the notation introduced 
in Subsection \ref{subsect:induced-actions}.
We consider the following subset in ${\rm Pos}^{\spinc}_n (B\Gamma)$:
\begin{equation}\label{charac}
[(M,u\co M\to B\Gamma,g,\nabla_L)]\in {\rm Pos}^{\spinc}_* (B\Gamma)\}.
\end{equation}
In the notation of Subsection \ref{subsect:induced-actions} this in nothing 
but $u_* (b(\mathcal{P}^c (M,L)))$.
Note that we have a surjective map
$\mathcal{R}^{c,+}(M,\sigma,L)\rightarrow u_* (b(\mathcal{P}^c (M,L)))$,
sending $(g,\nabla_L)$ to 
$ [(M,u\co M\to B\Gamma,g,\nabla_L)]$.

In the next Proposition we briefly denote by $[g,\nabla_L]$ an element in 
$u_* (b(\mathcal{P}^c (M,L)))$.
\begin{proposition}\label{prop:char}
Let $(M,\sigma,L)$ be  spin$^c$ and totally non-spin.
Suppose $[g,\nabla_L]\in u_* (b(\mathcal{P}^c (M,L)))$ and $x\in \ker
({\rm Pos}^{\spinc}_n (B\Gamma)\to \Omega^{\spinc}_n (B\Gamma))$. Then
$[g,\nabla^L]+ x \in u_* (b(\mathcal{P}^c (M,L)))$. Moreover, the
action of $\ker ({\rm Pos}^{\spinc}_n (B\Gamma)\to \Omega^{\spinc}_n
(B\Gamma))$ on $u_* (b(\mathcal{P}^c (M,L)))$ is free and transitive.
%
%
\end{proposition}

\begin{proof}
Reasoning as in \cite[Proposition 2.4]{MR2366359}, this is a
consequence of the Spin$^c$ bordism theorem.  Indeed, let $x=
[X, \phi: X\to B\Gamma, g_X, \nabla_{L_X}]$ be as in the statement of
the Proposition. We know that, by assumption, $(X, \phi: X\to
B\Gamma)$ is spin$^c$ null-bordant. Thus $(M,u)$ disjoint union with
$(X,\phi)$ is spin$^c$ bordant to $(M,u)$; on this disjoint union we
have a generalized positive curvature metric, which we can propagate
back to $(M,L,u)$ thanks to the Spin$^c$ bordism theorem,
Theorem \ref{thm:spincbordism}.  This new element in ${\rm
Pos}^{\spinc}_n (B\Gamma)$ does belong, by construction, to $u_*
(b(\mathcal{P}^c (M,L)))$.
 The fact that this action is free and transitive is proved as in 
 \cite[Proposition 2.4]{MR2366359}.
\end{proof}
\begin{proposition}
\begin{enumerate}
\item If $m$ is odd, then $\Omega^{\spinc}_m (B \bZ/n)$ is finite. Hence the kernel $K^c$ 
of the natural homomorphism ${\rm Pos}^{\spinc}_m
(B\bZ/n)\to \Omega^{\spinc}_m (B\bZ/n)$ is a finite index subgroup.
\item  If $m=4k+3$ then $\rho_{(2)}: {\rm Pos}^{\spinc}_m (B\bZ/n)\to \bR$, with $\bR$ endowed with its additive structure,
 is a non-trivial homomorphism and thus of infinite image. Hence,
from \textup{(1)}, it follows that $\rho_{(2)}|_{K^c}$ has infinite
image.
\item Let $\iota\co\bZ/n \to \Gamma$ be an injective group homomorphism and let
$$B \iota_* \co  {\rm Pos}^{\spinc}_m (B\bZ/n)
\to {\rm Pos}^{\spinc}_m (B\Gamma)$$ be the induced homomorphism. Then 
\begin{equation}\label{Bj}
\rho_{(2)}(B \iota_*(x))=\rho_{(2)}(x),\;\;\;\forall x\in {\rm Pos}^{\spinc}_m (B\bZ/n).
\end{equation}
\end{enumerate}
\end{proposition}
\begin{proof}
(1) We apply the Atiyah-Hirzebruch spectral sequence $H_p\bigl(\bZ/n,
\Omega^{\spinc}_q\bigr)\Rightarrow \Omega^{\spinc}_{p+q}(B\bZ/n)$.
The groups $\Omega^{\spinc}_*$ are finite except in even degree,
and homology of any finite group is also finite except in
degree $0$.  So the result immediately follows.\\
(2) By  \cite{MR1339924} we know that there exists $\kappa\in {\rm Pos}^{\spin}_m (B\bZ/n)$
such that $\rho_{(2)}(\kappa)\not= 0$. Then by \eqref{eq:3} we have that 
 $j(\kappa)\in {\rm Pos}^{\spinc}_m (B\bZ/n)$ and $\rho_{(2)}(j(\kappa))\not= 0$.\\
 (3) Exactly the same proof as in  \cite[Lemma 2.22]{MR2366359} can be given here.
\end{proof}
We are now ready to prove Theorem \ref{theo:any-torsion-group} that we
restate below for the benefit of the reader.
\begin{theorem}
Let $(M,\sigma,L)$ be spin$^c$ and totally non-spin and assume that
$(g,\nabla^L)$ has gpsc. Assume that $\pi_1 (M)\equiv \Gamma$ has an
element of finite order and that $\dim M=4k+3$, $k\ge 1$.  Then the
set $\pi_0 \left(\mathcal{R}^{c,+}(M,\sigma,L)/ {\rm
Diff}(M,\sigma) \right)$ has infinite cardinality.
\end{theorem}
\begin{proof}
We denote by $u\co M\to B\Gamma$ the classifying map 
for the universal cover of $M$.
By assumption, there exists $n>1$ in $\bN$ and an injection 
$\iota\co \bZ/n\hookrightarrow \Gamma$.
Consider the composition
\begin{equation}\label{long-composition}
\mathcal{R}^{c,+}(M,\sigma,L) \twoheadrightarrow  u_* (b(\mathcal{P}^c (M,L)))
\hookrightarrow
{\rm Pos}^{\spinc}_m (B\Gamma)\xrightarrow{\rho_{(2)}} \bR\,.
\end{equation}
We know that $\rho_{(2)}$ restricted to $K^c:= \operatorname{Ker}(
{\rm Pos}^{\spinc}_m (B\bZ/n)\to \Omega^{\spinc}_m (B\bZ/n))$ is non
trivial with infinite image, so let $x\in K^c$ be such that
$\rho_{(2)} (x)\ne 0$.

Then, by naturality, $B\iota _* (x) \in \operatorname{Ker}( {\rm
Pos}^{\spinc}_m (\Gamma)\to \Omega^{\spinc}_m (\Gamma))$. By
Proposition \ref{prop:char} we know that $B\iota _* (x)+
[(M,u,g,\nabla_L)] \in u_* (b(\mathcal{P}^c (M,L)))$ and, moreover,
\begin{equation*}
\rho_{(2)} \left(B\iota _*
(x)+ [(M,u,g,\nabla_L)]\right)= \rho_{(2)}
(x)+ \rho_{(2)}([(M,u,g,\nabla_L)])\,.
\end{equation*}
Since we know that the right hand side has infinite image as $x$ runs
in $K^c$, we conclude that $\rho_{(2)}\co u_* (b(\mathcal{P}^c
(M,L)))\to \bR$ has infinite image in $\bR$. Now, by standard stability
properties of the Cheeger-Gromov rho invariant, the composition of all
the maps in \eqref{long-composition} is locally constant on
$\mathcal{R}^{c,+}(M,\sigma,L)$ and therefore passes to $\pi_0
(\mathcal{R}^{c,+}(M,\sigma,L))$.  Thus $\pi_0
(\mathcal{R}^{c,+}(M,\sigma,L))$ has infinite cardinality.  Moreover,
by Proposition \ref{prop:invariance-CG} we know that if $\varphi\in
{\rm Diff}(M,\sigma)$ then
$\rho_{(2)}(D^\Gamma_{(g,\nabla^L)})= \rho_{(2)}(D^\Gamma_{(\varphi^*g,\nabla^{\varphi^* L})})$;
hence  such composition is constant along the orbits
of the action of ${\rm Diff} (M,\sigma)$; thus
the cardinality of $\pi_0 \left(\mathcal{R}^{c,+}(M,\sigma,L)/ {\rm Diff} (M,\sigma) \right)$ is infinite.
\end{proof}
\subsection{Proofs for  Theorem \ref{theo:moduli1}
and Theorem \ref{theo:moduli2}}
Let $(M,\sigma,L)$ be a connected spin$^c$ totally non-spin closed
$n$-manifold with fundamental group $\Gamma$ and universal cover
$\tM$. We assume that $n\geq 5$ and that
$\mathcal{R}^{c,+}(M,\sigma,L) \ne \emptyset$.  We fix a base point
$(g_0,\nabla^0)\in \mathcal{R}^{c,+}(M,\sigma,L)$. If
$(g,\nabla)\in \mathcal{R}^{c,+}(M,\sigma,L)$ then we shall briefly
denote by $\rho (g,\nabla)$ the rho class in $\SG^\Gamma_{n}
(E\Gamma)$ associated to the $L^2$-invertible Dirac operator
$D^\Gamma_{(g,\nabla)}$.

The cohomology groups $H^{*}(\Gamma;F_0\Gamma)$ appear in the statement of Theorem \ref{theo:moduli1}; 
there is, in general,  a monomorphism
\begin{equation}\label{from-F-to-P}
H^*(\Gamma;F_0\Gamma)\xrightarrow{\varphi} HP^*_{\del} (\bC\Gamma).
\end{equation}
(The definition of $F\Gamma$ was given just before
Proposition \ref{prop:from-XYZ}.)
In  case $\Gamma$ is Gromov hyperbolic, then 
$H^*(\Gamma,F_0\Gamma) \xrightarrow{\varphi}  HP^*_{\del} (\bC\Gamma)$
is an isomorphism;
moreover both groups can be described  as direct products of cohomology groups of  centralizers of elements of finite order, with the inclusion of these constituents
compatible with respect to the isomorphism $\varphi$. A completely 
analogous result holds at the homological level. For the details, 
that are somewhat cumbersome to write down explicitly, see 
\cite{PSZ}, Corollary 7.9,  Corollary 12.12, and
especially formulas (12.12), (12.13) and (12.14) there.

\medskip
\noindent
 Our first step is the analogue of \cite[Theorem 13.9]{PSZ}; 
 in fact, we take this opportunity to expand the rather brief 
 treatment of this point in
 \cite{PSZ}. We shall use  Proposition \ref{prop:from-XYZ}. Recall the map $\iota$ figuring in the Higson-Roe
 analytic surgery sequence: $\iota\co K_{n+1}(C^*_{r}\Gamma)\to \SG^\Gamma_{n} (E\Gamma)$.

\begin{proposition}\label{prop:XYZ}
Assume that $\Gamma$ is Gromov hyperbolic, so 
that the Baum-Connes assembly map is  bijective.%
\footnote{We could assume more generally that the
assembly map $\mu$ is rationally injective, but since 
our ultimate interest is in Gromov hyperbolic groups, we assume 
Gromov hyperbolicity from the start.}
Then for each subgroup $H \subset \bigoplus_{k>0,p\in\{0,1\}} H^{n+1-4k-2p}(\Gamma;F_0^p\Gamma)$ 
there is a subgroup $B\subset  K_{n+1}(C^*_{r}\Gamma)$ 
and for all $\beta\in B$ there are gpsc metrics
$(g_\beta,\nabla^\beta)$ such that 
\begin{equation}\label{affine}
\rho(g_\beta,\nabla^\beta)= \rho(g_0,\nabla^0)+\iota(\beta);\qquad\forall \beta\in B
\end{equation}
and such that the pairing $\langle \, , \, \rangle_{{\rm K}} $ between $H$ and $B$ is non-degenerate.
\end{proposition}

\noindent
The pairing $\langle \, , \, \rangle_{{\rm K}} $  here is the one defined by the isomorphism $\varphi$ in  \eqref{from-F-to-P} together with  the Chern character $\Ch_\Gamma$; we explained it in  \eqref{ch-c*}. Notice that the classes in $H $ are represented by  cocycles that extend to the Puschnigg algebra
  $\mathcal{A}\Gamma\subset C^*_{red}\Gamma$ through $\varphi$ in \eqref{from-F-to-P}.
\begin{proof}
Let $H$ be as in the statement and fix $H'  \subset \bigoplus_{k>0,p\in\{0,1\}} H_{n+1-4k-2p}(\Gamma;F_0^p\Gamma)$ of the same dimension and pairing non-trivially with $H$.
We apply Proposition \ref{prop:from-XYZ} 
and find  $B$ in the range of $\Ind_\Gamma^c\otimes\mathbb{C}$
equal to $c\circ \mu\circ {\rm ph}_\Gamma^{-1} (H')$, with $c$ denoting complexification, $\mu$  the Baum-Connes map and ${\rm ph}_\Gamma$ the Pontryagin isomorphism. Notice that 
the numerology of the indices in our homology direct sum in the statement of this Proposition  is the same as the one appearing in Proposition \ref{prop:from-XYZ}. We now take the inverse image of $B\subset
  K_{n+1}(C^*_{r}\Gamma)$
 under the map 
 $$\mathcal{P}^c (M,L)\xrightarrow{\iota_{(g_0,\nabla^0)}}
 R^{\spinc}_{n+1}(B\Gamma)\xrightarrow{\Ind^c_\Gamma} K_{n+1}(C^*_r\Gamma)$$ and define in this way a subset $\mathcal{P}_B:=\{(g_\beta,\nabla^\beta)\}_{\beta\in B}$ in $\mathcal{P}^c (M,L)$.
By construction 
we have that
$$
\begin{array}{c}
\rho^c \left( \partial \left(\iota_{(g_0,\nabla^0)} ((g_\beta,\nabla^\beta))\right)\right)=\rho^c ((g_\beta,\nabla^\beta)) - \rho^c ((g_0,\nabla^0))\quad\text{in}\quad \SG^\Gamma_n (E\Gamma)\,.
\end{array}
$$
We now use crucially the commutativity of \eqref{eq:StolzToAna} and the definition of $\mathcal{P}_B$
in order to conclude that  the left hand side of the last equation is also equal to $\iota (\beta)\in \SG^\Gamma_{n} (E\Gamma)$; thus $\forall \beta\in B$ we have that  $\rho(g_\beta,\nabla^\beta)= \rho(g_0,\nabla^0)+\iota(\beta)$ as required.
It remains to see that $B$ pairs non-trivially with $H$ through $\varphi$  in \eqref{from-F-to-P}. 
Assume that $B=c(B_{{\rm O}})$ with $B_{{\rm O}}\in KO_{n+1}(C^*_{r,\mathbb{R}}\Gamma)$. 
If $\beta\in B$ then by  construction $\beta= c(\beta_{{\rm O}})$ and 
$$\beta_{{\rm O}}=\mu ({\rm ph}_\Gamma^{-1} (h'))\;\;\;\text{with}\;\;\;h'\in H'.$$
Recall that  there exists a commutative diagram \begin{equation}\label{homology-compatibility}
    \xymatrix{   KO_p^\Gamma(\underline{E}\Gamma) \ar[r]^{\mu} \ar[d]^{{\rm ph}_\Gamma} &
      KO_p(C^*_{r,\mathbb{R}}\Gamma)\ar[r]^{c}\ar[d]^{\Ch_\Gamma} & K_p (C^*_r\Gamma)\ar[d]^{\Ch_\Gamma} \\
      \bigoplus_{k\in\mathbb{Z}} H_{p+4k}(\Gamma;F^0\Gamma) \oplus
      H_{p+2+4k}(\Gamma;F^1\Gamma) \ar[r]^(.6){\psi} &    
 HP_{p}(\mathcal{A}\Gamma)^0\oplus  HP_{p}(\mathcal{A}\Gamma)^1\ar[r]^(.6){c_P} & HP_{p} (\mathcal{A}\Gamma) }
  \end{equation}
Here the superscripts $0$ and $1$ in the periodic cyclic homology groups denote
the $\pm1$-eigenspaces under the endomorphism induced by
mapping $\gamma\in\Gamma$ to $\gamma^{-1}$.
Then the following equalities establish that $B$ pairs with $H$ in a non-degenerate way:
$$
\begin{aligned}\langle \Ch_\Gamma \beta,\varphi (h) \rangle= 
\langle \Ch_\Gamma c (\beta_{{\rm O}}),\varphi (h) \rangle=&
\langle c_P (\Ch_\Gamma (\beta_{{\rm O}})),\varphi (h) \rangle=
\langle c_P (\Ch_\Gamma \mu ({\rm ph}_\Gamma^{-1} (h')),\varphi (h) \rangle\\
=& \langle c_P (\psi \circ {\rm ph}
({\rm ph}_\Gamma^{-1} (h')))
,\varphi (h) \rangle=
\langle (c_P \circ \psi) (h'),\varphi (h) \rangle
\end{aligned}
$$
Indeed, the last member of this series of equalities is the pairing between elements of $H'$ and $H$ 
in the (co)homology of $\Gamma$ with values in $F\Gamma$
moved through injective maps to a pairing in periodic cyclic (co)homology. Since $H$ and $H'$ 
pairs in a non-degenerate way and the pairings are compatible, we are done.
\end{proof}

\begin{proposition}\label{prop:crucial-moduli} We make the same assumptions as in the previous proposition but 
require in addition that $H$ 
is pointwise fixed by a finite index subgroup $V$ of ${\rm Out}(\Gamma)$. Let $k$ be the dimension of $H$.
Then there exists 
a finite index subgroup $G$ of ${\rm Diff} (M,\sigma)$ such that 
$${\rm rank} (b(\mathcal{P}^c (M,L))/G)\geq k\,.$$
\end{proposition}

\begin{proof}
We adapt the proof of the original result in \cite{PSZ} (see Theorem 13.1 there).\\
Let $G$ be the subgroup 
$$G:={\rm Diff}(M,\sigma)\cap \alpha^{-1} (V)$$
with $\alpha: {\rm Diff}(M) \to {\rm Out}(\Gamma)$. 
Recall $b: \mathcal{P}^c (M,L)\to {\rm Pos}^{\spinc}_n (B\Gamma)$, with image 
$b(\mathcal{P}^c (M,L))$.
We can define a pairing $H\times b(\mathcal{P}^c (M,L))\to \mathbb{C}$, sending
$(h, b([g^\beta,\nabla^\beta]))$ to $\langle \varphi (h), \rho(g^\beta,\nabla^\beta) \rangle$,
with $\varphi(h)\in HP^* (\mathcal{A}\Gamma)$. From \eqref{affine} and the fact that
 $H$ and $B$ pair in a non-degenerate way 
we understand that this pairing is as non-degenerate as the pairing of $H$ and $B$ and so 
the rank of $b(\mathcal{P}^c (M,L))$ is at least $k$. Now we use the extra 
hypothesis 
on $H$ and the crucial formula 
 \eqref{compatibility-through-Ch-bis} to see that
 \begin{equation}\label{compatibility-through-Ch-ter}
\langle \rho_c (D^\Gamma_{\psi^*g,\nabla^{\psi^*L}}),
 \varphi(h) \rangle=
\langle \rho_c (D^\Gamma_{(g,\nabla^L)}),\varphi (h) \rangle  \end{equation}
for any $b[g, \nabla^L]\in b(\mathcal{P}^c (M,L))$ and any $h\in H$. We deduce that the pairing $H\times
b(\mathcal{P}^c (M,L))\to \mathbb{C}$
descends to a pairing $H\times (b(\mathcal{P}^c (M,L))/G)\to \mathbb{C}$;  using again \eqref{affine} 
we see  that this pairing 
is as non-degenerate as the pairing of $H$ with $B$ and thus
${\rm rank} (b(\mathcal{P}^c (M,L))/G)\geq k$, which is what we wanted to show.
\end{proof}

We are now ready to prove  Theorem \ref{theo:moduli1} and Theorem \ref{theo:moduli2}. 
These are in fact direct consequences of Proposition \ref{prop:crucial-moduli}.
We shall restate the two theorems for the benefit of the reader.

\begin{theorem}
    Let $\Gamma$ be a Gromov hyperbolic group with finite outer automorphism group
    and let $(M^n,\sigma,L)$ be a closed connected spin$^c$ totally non-spin  manifold  with
    $\pi_1(M)=\Gamma$.
    We assume that $n\geq 5$ and that $\mathcal{R}^{c,+}(M,\sigma,L) \not= \emptyset$. Then
    there exists a finite index subgroup $G\leq {\rm Diff}(M,\sigma)$ such that     \begin{equation*}
     {\rm rank} (b(\mathcal{P}^c (M,L))/G)\geq 
      \sum_{k>0,p\in\{0,1\}} \rank(H^{n+1-4k-2p}(\Gamma;F^p_0\Gamma)). 
    \end{equation*}
    \end{theorem}
    
    \noindent
    Notice that as already remarked in \cite{PSZ}, and proved by Bestvina and Feighn in
    \cite{B-F}, the 
    finiteness condition of ${\rm Out}(\Gamma)$ is generically true for a Gromov hyperbolic group.
    \begin{proof}
  As  ${\rm Out}(\Gamma)$ is finite we can take $V=\{1\}$ and $H=  \sum_{k>0,p\in\{0,1\}} H^{n+1-4k-2p}(\Gamma;F^p_0\Gamma)$ in Proposition \ref{prop:crucial-moduli}. \end{proof}

    \begin{theorem}
    Let $M$ be as in the previous Corollary and let $\Gamma$ be  Gromov hyperbolic. Then there 
    exists a finite index subgroup $G\leq {\rm Diff}(M,\sigma)$ such that     
    $ {\rm rank} (b(\mathcal{P}^c (M,L))/G)\geq \ell$
    with   
 \begin{equation*}
      \ell :=
      \begin{cases}
        \abs{\{[\gamma]\subset\Gamma\mid \ord(\gamma)<\infty\}}; & n\equiv -1\pmod 4\\
        \abs{\{[\gamma]\subset\Gamma\mid [\gamma]\ne[\gamma^{-1}]\}}; &
        n\equiv 1\pmod 4 
      \end{cases}
    \end{equation*}
    \end{theorem}
    
    \begin{proof}
    The proof is identical to the one given in \cite{PSZ}; it is a consequence of Proposition \ref{prop:crucial-moduli}.
    \end{proof}
%

\bibliographystyle{amsplain}
\bibliography{GPSCandStolzseq}

\end{document}